\documentclass[10pt,reqno]{amsart}
\usepackage{setspace}
\usepackage[utf8]{inputenc}
\usepackage{graphicx}
\graphicspath{ {./images/} }
\usepackage[pagewise]{lineno}
\usepackage{amssymb}
\usepackage[foot]{amsaddr}
\usepackage[all,arc]{xy}
\usepackage{enumerate}
\usepackage{mathrsfs}
\usepackage[margin=1in]{geometry} 
\usepackage{amsmath,amsthm}
\usepackage{mathtools}
\usepackage[T1]{fontenc}
\usepackage{fourier}
\usepackage{stmaryrd}
\usepackage{amsfonts}
\usepackage{cite}

\usepackage[backref=page]{hyperref}
\usepackage[capitalise,nameinlink]{cleveref}

\newcommand *\w{^\wedge}

\newcommand{\cprime}{\('\)}

\newtheorem{thm}{Theorem}[section]
\theoremstyle{definition}

\newtheorem{example}[thm]{Example}

\newtheorem{theorem}{Theorem}[section]
\newtheorem{lemma}[theorem]{Lemma}
\newtheorem{proposition}[theorem]{Proposition}
\newtheorem{definition}[theorem]{Definition}
\newtheorem{remark}[theorem]{Remark}

\usepackage{xcolor}
\hypersetup{
	colorlinks,
	linkcolor={blue!100!black},
	citecolor={red!100!black},
	urlcolor={green!100!black}
}

\DeclarePairedDelimiterX{\inp}[2]{\langle}{\rangle}{#1, #2}

\makeatletter
\def\@tocline#1#2#3#4#5#6#7{\relax
	\ifnum #1>\c@tocdepth 
	\else
	\par \addpenalty\@secpenalty\addvspace{#2}%
	\begingroup \hyphenpenalty\@M
	\@ifempty{#4}{%
		\@tempdima\csname r@tocindent\number#1\endcsname\relax
	}{%
		\@tempdima#4\relax
	}%
	\parindent\z@ \leftskip#3\relax \advance\leftskip\@tempdima\relax
	\rightskip\@pnumwidth plus4em \parfillskip-\@pnumwidth
	#5\leavevmode\hskip-\@tempdima
	\ifcase #1
	\or\or \hskip 1em \or \hskip 2em \else \hskip 3em \fi%
	#6\nobreak\relax
	\dotfill\hbox to\@pnumwidth{\@tocpagenum{#7}}\par
	\nobreak
	\endgroup
	\fi}
\makeatother

\usepackage{hyperref}

\numberwithin{equation}{section}

\usepackage{epigraph}

\title{A Class of De Giorgi Type and H\"older Continuity for Some Problems in Musielak-Orlicz-Sobolev Spaces}

\author{Hlel Missaoui$^{1}$}
\author{Anouar Bahrouni$^{2}$}
\author{Hichem Ounaies$^{3}$}
\address{$^{1}$ Mathematics Department,  Higher Institute of Computer Science of Mahdia, University of Monastir, 5111 Mahdia, Tunisia.}
\address{$^{2}$ $^{3}$ Mathematics Department, Faculty of Sciences, University of Monastir, 5019 Monastir, Tunisia.}
\email{$^{1}$ \tt Hlel.Missaoui@fsm.rnu.tn}
\email{$^{2}$ \tt bahrounianouar@yahoo.fr}
\email{$^{3}$ \tt Hichem.Ounaies@fsm.rnu.tn}
\date{}
\begin{document}
\begin{abstract}
In this paper, we introduce a new class of De Giorgi-type functions, denoted by \(\mathcal{B}_{G(x,t)}\), and establish the Hölder continuity of its elements under suitable assumptions on the generalized \(\textnormal{N}\)-function \(G(x,t)\). As an application, we prove the Hölder continuity of solutions to quasilinear equations in divergence form with \(G(x,t)\)-growth, covering both standard and critical growth regimes: boundedness up to the boundary and interior Hölder continuity are obtained under Dirichlet and Neumann boundary conditions, and Hölder continuity up to the boundary is obtained for the Dirichlet problem. Our results extend classical Hölder continuity results for the \((p,q)\)-Laplacian \cite{Li1991}, variable-exponent problems \cite{Fan1999}, double-phase problems \cite{Colombo2015}, and logarithmic double-phase problems \cite{kyho} to a unified framework. They also cover two-phase models with two H\"older continuous coefficients, in the spirit of the multi-phase results of \cite{DeFilippisOh2019} (established for minimizers), including a logarithmic variant, for solutions of quasilinear equations in divergence form.

\end{abstract}
\maketitle
\noindent \textbf{Keywor{\rm d}s.} Class of De Giorgi Type; H\"older Continuity; Musielak-Orlicz-Sobolev Spaces; Weak solutions.

\noindent\textbf{\textup{2020} Mathematics Subject Classification.} {Primary: 
35R09, 
35B65; 
Secondary: 
47G20, 
35K55
}

\smallskip
	
\tableofcontents

\section{Introduction}
\label{sec:intro}
\subsection{Overview}
The qualitative theory of partial differential equations studies the behavior of weak solutions
$u \in W^{1,1}_{\mathrm{loc}}(\Omega)$ of problems of the form
\begin{equation}\label{main_eq}
-\operatorname{div} A(x,u,D u) = B(x,u,D u) \quad \text{in } \Omega.
\end{equation}
Existence results, based on the calculus of variations and topological methods, usually give solutions only in Sobolev spaces. 
To use these solutions in applications, it is important to know if they are more regular. 
In particular, one wants Hölder continuity $u \in C^{0,\alpha}_{\mathrm{loc}}(\Omega)$. 
This property lies between boundedness ($L^\infty$) and differentiability ($C^{1,\alpha}$), and is crucial for uniqueness, stability, and free boundary problems.\\

The study of regularity began with the famous works of De Giorgi \cite{deGiorgi1957} and Nash \cite{Nash1958}, later completed by Moser \cite{Moser1960,Moser1961}. 
They proved that weak solutions of linear elliptic equations in divergence form,
\[
-\operatorname{div}(A(x)D u)=0,
\]
with bounded measurable coefficients $A(x)$ satisfying uniform ellipticity
$\lambda |\xi|^2 \le \langle A(x)\xi,\xi\rangle \le \Lambda |\xi|^2$,
are H\"older continuous. 
This result was very surprising, because it did not require smoothness of $A(x)$. 
It introduced two powerful techniques, the Moser iteration and the De Giorgi measure argument.
Later, attention moved to quasilinear equations, such as the $p$-Laplacian
\[
\Delta_p u := \operatorname{div}(|D u|^{p-2}D u).
\]
Uraltseva \cite{Uraltseva1968}, Uhlenbeck \cite{Uhlenbeck1977}, and Lieberman \cite{Li1988} proved that solutions of $-\Delta_p u=0$ are $C^{1,\alpha}$. 
But the situation is more delicate in the presence of a right-hand side: for $-\Delta_p u=f$ with $f \in L^q$ and $q<\frac{n}{p}$, solutions need not even be bounded. 
Moreover, for functionals with $(p,q)$-growth and a too large gap between $p$ and $q$, minimizers may fail to be regular \cite{MG2,marc1,Mar1}. 
This shows that the regularity program must proceed in two steps: first, prove $u\in L^\infty$ (by De Giorgi’s lemma or Moser iteration), then prove $u\in C^{0,\alpha}$.\\

These ideas and techniques were systematically extended to tackle the more general non-polynomial growth of operators in the Orlicz setting. The foundational work by Donaldson and Trudinger \cite{DonaldsonTrudinger1971} established the Orlicz-Sobolev embeddings underlying this theory. Lieberman \cite{Li1991} proved that weak solutions to quasilinear equations of the form $-\operatorname{div}(a(|D u|)D u) = 0$ are locally H\"older continuous ($C^{0,\alpha}$). This required the associated \textnormal{N}-function $G(t)$ to satisfy the condition $\Delta_2$, which is essential for the Sobolev-Poincar\'e inequalities and the scaling arguments of the De Giorgi-Nash-Moser theory. These contributions laid the foundation of the regularity theory in the Orlicz framework.\\

Although Orlicz spaces encompass many classical examples, there remain important cases that they do not cover. For this reason, considerable research has focused on the variable exponent framework, where the growth is governed by a function $p(x)$. The analysis of problems involving terms of the type $t^{p(x)}$ was initiated by Zhikov; see \cite{Zhikov1987,Zhikov1995}. Subsequently, Fan and Zhao \cite{Fan1999} uncovered a new regularity phenomenon by showing that weak solutions of equations of the form
\[
-\operatorname{div}\big(|Du|^{p(x)-2}Du\big)=0
\]
are locally H\"older continuous, provided that the exponent function $p(x)$ is log-H\"older continuous. This condition, namely
\[
|p(x)-p(y)| \le \frac{C}{-\ln |x-y|} \quad \text{for } |x-y|<1,
\]
was shown to be not only sufficient but essentially necessary for regularity, a fact later clarified and reinforced by the work of Acerbi and Mingione \cite{Acerbi2007}.\\

A further major advancement in the theory of non-standard growth problems came with the analysis of double-phase functionals, which introduce a new type of heterogeneity: the growth exponent remains fixed, but the structure of the operator changes depending on the location in the domain. For functionals of the form
\[
\operatorname{div}\!\left(|D u|^{p-2}D u + a(x)|D u|^{q-2}D u\right),
\]
with $1 < p \leq q$, the regularity theory reveals a novel feature: the regularity of solutions depends directly on the regularity of the coefficient $a(x)$. Colombo and Mingione~\cite{Colombo2015, Colombo2015b} proved that in the natural growth case, boundedness and H\"older continuity of minimizers and solutions hold provided $a(x)$ is H\"older continuous, i.e., $a(x) \in C^{0,\alpha}$. This condition is sharp: if $a(x)$ is merely continuous, minimizers may fail to be locally bounded~\cite{Colombo2015}. This phenomenon highlights a key principle: in problems with non-uniformly elliptic structures, the regularity of the coefficients controlling the growth is often necessary for the regularity of solutions. The boundedness result in this setting was also independently obtained by Baroni, Colombo, and Mingione~\cite{Baroni2016}.\\

Due to the vast scope and rapid development of regularity theory, we cannot cite all contributions individually. We acknowledge with gratitude the many valuable works not mentioned here and instead refer the reader to comprehensive surveys and bibliographies, such as~\cite{mr3, Mingione2006}. For detailed discussions of specific regularity results, we direct the reader to the following foundational works and the extensive literature contained therein: for the classical Laplacian and linear theory, see~\cite{deGiorgi1957, Nash1958, Moser1960, Moser1961}; for the $p$-Laplacian and $(p,q)$-growth, see~\cite{ GmeinederKristensen2024, Uraltseva1968,  KristensenTaheri2003, Li1988, Uhlenbeck1977,  KristensenMingione2008, DeFilippisKochKristensen2024,  KristensenMingione2010, Manfredi1986, Manfredi1988, marc1, marc2, Ac1, Ac2, mg3, MG3, NF, LE, MH, Pucci}; for Orlicz and other non-standard growth conditions, see~\cite{Li1991, Simon1,HsK2,HsK1,HHL21, Harjulehto2017, Cupini2017, Eleuteri2016, Chlebicka2021, Chlebicka2022, mgg}; for double-phase problems, see~\cite{Colombo2015, Colombo2015b, Baroni2016,defi,Simon1}; and for variable exponent $p(x)$-Laplacian problems, see~\cite{Fan1999, FanZhao1999, Acerbi2007, Coscia1999, KKPZ}.\\

 The Musielak--Orlicz growth framework represents the most general setting, as it
encompasses all previously studied models, including the standard $p$-growth,
Orlicz growth, variable exponent growth, and double-phase (or $(p,q)$) growth. A major advance in the regularity theory for generalized Orlicz growth was achieved
by Harjulehto, Hästö, and Lee \cite{HHL21}, who established Hölder continuity of
local minimizers of integral functionals with Musielak--Orlicz growth under very
general structural assumptions, namely (A0), (A1), and (A2) (see
Subsection~\ref{sec2}). They cover important models including the double-phase
case. This highlights the strength and generality of the variational approach
developed in \cite{HHL21} for the regularity of minimizers, which relies on
comparison estimates and harmonic approximation techniques. A further important
contribution in this direction is due to Ho, Kim, and Winkert \cite{kyho}, who
established Hölder regularity for logarithmic double-phase problems, thereby
extending the scope of the variational approach to functionals with logarithmic
perturbations.

The present paper is concerned with a different class of problems. We study weak
solutions of quasilinear elliptic equations and systems in divergence form \eqref{main_eq},
possibly including lower-order terms, for which a variational structure is not
available in general. Consequently, the methods developed in \cite{HHL21} cannot
be applied directly in our setting. Instead, we adopt a De Giorgi--type iteration
adapted to Musielak--Orlicz growth conditions. Within this non-variational
framework, an additional control on the generalized \textnormal{N}-function,
formulated in assumption \eqref{GG3}, is required in order to derive the
key structural estimates and to obtain Hölder regularity of solutions. We also
emphasize that our analysis treats Neumann boundary conditions, which are not
addressed in the variational framework of \cite{HHL21} or \cite{kyho}: for the Neumann problem we prove boundedness up to the boundary and interior Hölder continuity, while Hölder continuity up to the boundary is proved for the Dirichlet problem (and, more generally, for elements of the De Giorgi class with Hölder continuous trace). Our results cover the classical double-phase growth case
\(G(x,t) = |t|^p + a(x)|t|^q\), as well as a significantly broader class of
non-standard growth conditions, including the logarithmic double-phase case,
recently treated by Ho, Kim, and Winkert \cite{kyho}, as well as two-phase models with two H\"older continuous coefficients, in the spirit of \cite{DeFilippisOh2019}, and their logarithmic variants (Examples~\ref{ex9} and~\ref{ex11}). For the double-phase case, the coefficient \(a\) is required
to be H\"older continuous and the gap condition
\(q \le p + \frac{p\alpha}{n}\) (with \(\alpha\) the H\"older exponent of \(a\); in particular \(q - p < \alpha\)) is
required, in accordance with the classical theory of Colombo--Mingione
\cite{Colombo2015}. The precise conditions for all the models are stated in
Subsection~\ref{Examples}. Models with three or more phases are not covered by the two-pair structure \eqref{GG3}; see the end of Example~\ref{ex9}.

To the best of our knowledge, the literature contains only one other work
specifically devoted to the Hölder continuity of weak solutions to divergence-form
elliptic problems under Musielak--Orlicz growth conditions, namely the paper by
Wang, Liu, and Zhao \cite{WaLiuZhao}. In their analysis, the authors impose a
restrictive structural assumption on the growth function $G(x,t)$. More
precisely, they assume the existence of a strictly increasing differentiable
function $\xi : [0,+\infty) \to [0,+\infty)$ satisfying
\[
n\,\xi(s) > s\,\xi'(s) > \xi(s) \quad \text{for all } s \ge 0,
\]
such that
\begin{itemize}
    \item[(A11)] $G(x, s t) \ge \xi(s)\, G(x,t)$ for all $s \ge 0$, $t \in \mathbb{R}$,
    and $x \in \Omega$;
    \item[(A12)] $\xi, \xi^{-1}, \xi_{\ast}, \widehat{\widetilde{\xi}} \in
    \Delta_{\mathbb{R}^{+}}$.
\end{itemize}
A significant limitation of this framework is that assumption (A11) excludes
several fundamental models in the theory of non-standard growth. In particular, it
is not satisfied by variable exponent growth, $(p,q)$-growth equations, or
double-phase functionals, which are among the main motivating examples of the
Musielak--Orlicz setting. As a consequence, while the results of \cite{WaLiuZhao}
are obtained under a specific and interesting set of hypotheses, their approach
does not provide a unified treatment of these central growth models.
\\

Motivated by the preceding discussion, in this paper we establish the local
H\"older continuity of weak solutions to~\eqref{main_eq} under general
Musielak--Orlicz growth conditions on the operators $A$ and $B$.
The principal novelty of our work lies in providing a unified framework that
simultaneously covers the classical cases of the $p$-Laplacian, variable
exponent, and Orlicz-type equations, while also treating both subcritical and
natural growth (critical) cases with either Dirichlet or Neumann boundary
conditions. Note that equation \eqref{main_eq} involves a critical term, which
makes the study of regularity more delicate and necessitates a careful analysis
of the associated Sobolev conjugate for generalized Young functions.

Since no single notion of Sobolev conjugate is adapted to all the regimes we
consider, we work with two complementary constructions. The first is the Sobolev
conjugate introduced by Fan~\cite{Fan2012a}, which is defined through the
inverse $G^{-1}$ and is particularly well suited to the critical growth regime
and to the boundary-trace embeddings needed for the Neumann problem. Fan
introduced this concept under the assumption that $G$ satisfies a Lipschitz-type
condition in the spatial variable (see \eqref{GG5}). The second is the sharper formulation
developed by Cianchi and Diening~\cite{Cianchi2024}, which is obtained under
weaker regularity assumptions and yields a more precise description of the
embedding constants in the interior Sobolev embedding. We use the
Cianchi--Diening conjugate for the Dirichlet problem and the interior estimates,
and Fan's conjugate for the critical regime and the boundary-trace analysis of
the Neumann problem. These two constructions require different structural
assumptions on $G$. More precisely, the Cianchi--Diening Sobolev embedding is
available under assumption \eqref{GG2}, which guarantees the local comparability
of $G^{-1}$ needed for the interior Sobolev inequality, whereas Fan's
construction requires assumptions \eqref{GG4}--\eqref{GG5}, which ensure that
the conjugate $G_*$ is well defined and sufficiently regular for the
boundary-trace embedding. This is precisely why both sets of assumptions appear
in our main results: \eqref{GG2} for the Cianchi--Diening embedding, and
\eqref{GG4}--\eqref{GG5} for Fan's embedding. In this way, the two constructions
complement each other: each is applied precisely in the regime where its
structural properties are available, which allows us to cover the subcritical
and critical cases simultaneously. However, despite these advances, a complete
theory of Sobolev conjugates for generalized Young functions remains incomplete,
lacking crucial structural properties such as the $\Delta_2$-condition, which is
notoriously difficult to verify; this is precisely why a combined use of the two
conjugates is necessary in our setting.\\
 
 As explained above, our main contribution is the following. We establish the local Hölder continuity of weak solutions to \eqref{main_eq} under general Musielak–Orlicz growth conditions on the operators $A$, which in particular generalizes the variable-exponent type operators treated in \cite{Fan1999}. Moreover, our work closes the gap left in \cite{Fan1999}, where the critical cases were excluded. At the same time, it extends another class of equations in the framework of $(p,q)$-growth with a critical term; see Subsection~1.3.\\

Four main difficulties arise from the Musielak--Orlicz structure when establishing regularity.  
First, the energy is nonhomogeneous and exhibits explicit $x$-dependence alongside the natural $t$-dependence, which prevents the direct application of classical scaling and iteration techniques such as the De Giorgi method.  
Second, the regularity proofs are highly sensitive to the precise parameters controlling the growth and structure of the function $G(x,t)$, requiring careful balancing of assumptions to ensure the validity of key inequalities.  
Third, even when adapted versions of De Giorgi-type methods are applicable, they may not guarantee regularity, as demonstrated by the counterexamples of Giaquinta~\cite{MG2} and Marcellini~\cite{marc2, marc1}, which show that minimizers may fail to be regular under non-standard $(p,q)$-growth when the gap between $p$ and $q$ is too large.  
Finally, the critical case introduces additional challenges compared to the subcritical case, due to the lack of uniform integrability and the need for more refined estimates.  
These aspects collectively indicate that establishing regularity under non-standard growth conditions requires new ideas and a delicate analysis beyond the classical framework.\\

Our approach is motivated by the works of Lieberman~\cite{Li1991} and Fan--Zhao~\cite{Fan1999}, and in particular by the classical techniques of De~Giorgi~\cite{deGiorgi1957} and Ladyzhenskaya--Ural'tseva~\cite{OL}. 
These foundational contributions introduced powerful tools for proving Hölder continuity of solutions. 
In particular, they defined the class $\mathcal{B}(\Omega, M, \gamma, \gamma_1, \delta, 1/q)$ and showed that all functions in this class are Hölder-continuous. 
This framework is especially effective under standard $m$-growth conditions (see, e.g.~\cite{MG1, MG3, OL, CM}). 
For the treatment of critical cases, we further rely on the method developed by Ho, Kim, Winkert, and Zhang~\cite{KKPZ}, and we also make use of some fundamental properties of the Sobolev conjugate established in \cite{Ala3}.
\subsection{Assumptions and Main Results}\label{sec1.2}
For the sake of clarity, in this subsection, we present only the main assumptions, key results, and core definitions. Additional definitions and technical details will be introduced in the subsequent sections.

Let $\Omega$ be a bounded Lipschitz domain in $\mathbb{R}^n$ with \( n \geq 2 \), and let 
\[G(x,t) = \int_0^{\vert t\vert} g(x,s)\, \mathrm{d}s,\]
be a generalized \textnormal{N}-function (see Section~\ref{sec2} for the precise definitions) and consider the following structural conditions:\\
There exist constants $1 < g^- \leq g^+ < n$ such that for all $x \in \overline{\Omega}$, $t > 0$:
\begin{equation}\label{D22}
1 < g^- \leq \frac{g(x,t)t}{G(x,t)} \leq g^+ .
\tag{G0}
\end{equation}
There exists $F \geq 1$ such that for almost every $x \in \overline{\Omega}$:
\begin{equation}\label{GG1}
F^{-1} \leq G(x,1) \leq F.
\tag{G1}
\end{equation}
There exists $\mu \in (0,1]$ such that for every ball $B \subset \overline{\Omega}$ with $|B| \leq 1$, for every $t \in \left[1, \frac{1}{|B|}\right]$, and for almost all $x,y \in B$:
\begin{equation}\label{GG2}
\mu G^{-1}(x,t) \leq G^{-1}(y,t).
\tag{G2}
\end{equation}

Let \(\bar{\delta} \geq 1\), \(\eta \geq 0\), \(\bar{\delta} +\eta >1\), and \(\ell (t)\coloneqq \ln \left(\bar{\delta} +\eta t\right)\); thus \(\ell\) is nondecreasing, \(\ell \geq 0\) on \([0,\infty)\) and \(\ell (t)\geq \ell (1) = \ln \left(\bar{\delta} +\eta \right) > 0\) for \(t\geq 1\). There exist \(\Theta \geq 1\), \(\bar\alpha \in (0,1]\), \(\bar{a} >0\) and, for \(i = 1,2\), functions \(p_{i},q_{i},p_{i}^{\prime},q_{i}^{\prime}:\overline{\Omega} \to (1,n)\) and \(a_{i},b_{i},a_{i}^{\prime},b_{i}^{\prime}:\overline{\Omega} \to [0,\bar{a} ]\) such that:
\begin{enumerate}
\item[(i)]  for all \(x\in \overline{\Omega}\), \(t\geq 1\),
\begin{equation}\label{GG3} 
\sum_{i = 1}^{2}\left[a_{i}^{\prime}(x)t^{p_{i}^{\prime}(x)} + b_{i}^{\prime}(x)t^{q_{i}^{\prime}(x)}\ell (t)\right]\leq G(x,t)\leq \sum_{i = 1}^{2}\left[a_{i}(x)t^{p_{i}(x)} + b_{i}(x)t^{q_{i}(x)}\ell (t)\right];
\tag{G3}
\end{equation}
\item[(ii)] \(p_{i}\leq q_{i}\), \(p_{i}^{\prime}\leq q_{i}^{\prime}\) on \(\overline{\Omega}\), \(p^{- } = \min_{i}\inf_{\overline{\Omega}}p_{i}^{\prime} > 1\), \(p^{+}\coloneqq \max_{i}\sup_{\overline{\Omega}}p_{i}\), \(q^{+}\coloneqq \max_{i}\max \left\{\sup_{\overline{\Omega}}q_{i},\sup_{\overline{\Omega}}q_{i}^{\prime}\right\} < n\);
\item[(iii)]  \(a_{i}^{\prime}(x) + b_{i}^{\prime}(x)\geq d > 0\) for all \(x\) and each \(i = 1,2\);
\item[(iv)] there are \(L_{0} > 0\), \(R_{\mathrm{G}}\in \left(0,\frac{1}{2}\right)\) such that for all \(x_{0}\in \overline{\Omega}\), \(R< R_{\mathrm{G}}\), \(i = 1,2\)
\[
\sup_{B_{R}(x_{0})}p_{i} - \inf_{B_{R}(x_{0})}p_{i}^{\prime}\leq \frac{L_{0}}{|\ln(R)|},\qquad \sup_{B_{R}(x_{0})}q_{i} - \inf_{B_{R}(x_{0})}q_{i}^{\prime}\leq \frac{L_{0}}{|\ln(R)|};
\]
\item[(v)]  \(a_{i}^{\prime},b_{i}^{\prime}\in C^{0,\bar\alpha}(\overline{\Omega})\), with seminorms \([a_{i}^{\prime}]_{0,\bar\alpha},[b_{i}^{\prime}]_{0,\bar\alpha}\) and \(b_{i}\leq \Theta b_{i}^{\prime}\) on \(\overline{\Omega}\);
\item[(vi)]  one of the following two alternatives holds:
\begin{equation}\label{GG6a}
C_{*}\coloneqq \max_{i}\sup_{x\in \overline{\Omega} ,t\geq 1}t^{q_{i}(x) - p_{i}(x) - \bar\alpha}\ell (t)< \infty,
\tag{G3a}
\end{equation}
or
\begin{equation}\label{GG6b}
\beta_{0}\coloneqq \min_{i}\inf_{\overline{\Omega}}b_{i}^{\prime} > 0.
\tag{G3b}
\end{equation}
\end{enumerate}

For all $x \in \overline{\Omega}$,
\begin{equation}\label{GG4}
\int_0^1 \frac{G^{-1}(x,t)}{t^{\frac{n+1}{n}}}\,\mathrm{d}t < +\infty. \tag{G4}
\end{equation}
For the definition of $G^{-1}$ see Definition \ref{definv}.
\begin{align}\label{GG5}
    \begin{split}
      G_* \in &C^{0,1}_{\mathrm{loc}}(\operatorname{Dom}(G_*))\text{ and there exist positive constants } \delta_0 < \dfrac{1}{n},\ C_0\text{ and }0 < t_0 < \min_{x\in \overline{\Omega}}T(x)=+\infty\text{ such that }\\
      & |\nabla_x G_*(x,t)| \leq C_0 \bigl(G_*(x,t)\bigr)^{1+\delta_0}, \text{ for } x \in \Omega\text{ and } t \in [t_0,+\infty) \text{ provided } \nabla_x G_*(x,t)\text{ exists},
    \end{split}\tag{G5}
\end{align}
where 
\(
\displaystyle{G_*(x,t) = \int_0^t g_*(x,s)\, \mathrm{d}s}
\)
is the  Fan's Sobolev conjugate of \( G(x,t) \) (see Definition \ref{DefF}) and $T(x)$ is defined in \eqref{Tx}.\\

\begin{remark}\label{rem:usage}
Only \eqref{D22} and \eqref{GG3} are used in the proof of the regularity of the De Giorgi class (Theorems~\ref{Theorem 2.1} and~\ref{Theorem 2.2}). The conditions \eqref{GG1}, \eqref{GG2} (Dirichlet problem) and \eqref{GG4}--\eqref{GG5} (Neumann problem) enter only through the Sobolev and trace embeddings used to prove the boundedness of weak solutions (Theorems~\ref{Theorem 4.1w} and~\ref{Theorem 6.1w}).
\end{remark}

\begin{remark}[On \eqref{GG4}--\eqref{GG5}]\label{rem:G5LWZ}
Conditions \eqref{GG4} and \eqref{GG5} are the assumptions (P3)--(P5) of \cite{LWZ} for $A=G$; moreover $T\equiv+\infty$ because $g^+<n$, so that $(P^1_\partial)$ of \cite{LWZ} holds as well. In \cite{LWZ} the function $A$ is a generalized \textnormal{N}-function satisfying the integral condition \eqref{GG4}, whereas in \cite{Fan2012a} it is normalized to be linear on $[0,1]$; thus no normalization of $G$ is needed here. Proposition~\ref{ffembd} is \cite[Theorem~4.2]{LWZ}, which gives the embedding of $W^{1,G(x,t)}(\Omega)$ into $L^{[G_*]^{\frac{n-1}{n}}}(\partial\Omega)$, together with the inclusion $L^{[G_*]^{\frac{n-1}{n}}}(\partial\Omega)\hookrightarrow L^{\vartheta}(\partial\Omega)$ for $\vartheta\mathop{\prec}_{\partial\Omega}[G_*]^{\frac{n-1}{n}}$ (see \cite[Proposition~2]{LWZ}); likewise Proposition~\ref{fembb} follows from \cite[Theorem~2.4(i)]{LWZ}. Although \eqref{GG5} is a condition on $G_*$ for large $t$, the function $G_*$ is defined through $G$ on all of $(0,\infty)$; a criterion that can be checked directly on $G$ is given in Lemma~\ref{lem:G5crit}.
\end{remark}

We now introduce a De Giorgi-type class of functions tailored to our framework:
\begin{definition}\label{Definition 2.1}
Let \( M, \gamma, \gamma_1 \), and \( \delta \) be positive constants with \( 0<\delta \leq 2 \). A function \( u \) in the Musielak-Orlicz-Sobolev space \(W^{1,G(x,t)}(\Omega) \) (The definition of Musielak–Sobolev spaces is provided in Subsection 2.2), satisfying \( \mathop{\mathrm{ess\,sup}}_{\Omega} |u(x)| \leq M \), is said to belong to the class \( \mathcal{B}_{G(x,t)}(\Omega, M, \gamma, \gamma_1, \delta) \) if the functions \( w(x) := \pm u(x) \) satisfy the inequality
\begin{equation}\label{2.4}
\int_{A_{k,\rho - \sigma \rho}} G(x,|Dw|) \, \mathrm{d}x \leq \gamma \int_{A_{k,\rho}} G\left(x,\left| \frac{w(x) - k}{\sigma \rho} \right|\right) \, \mathrm{d}x + \gamma_1 |A_{k,\rho}|,
\end{equation}
for every ball \( B_\rho \subset \Omega \), every \( \sigma \in (0,1) \), and every threshold \( k \) satisfying
\begin{equation}\label{2.5}
k \geq \mathop{\mathrm{ess\,sup}}_{B_\rho} w(x) - \delta M,
\end{equation}
where \( A_{k,\rho} := \{x \in B_\rho : w(x) > k\} \).
\end{definition}
One of the central results of this work is the following theorem.
\begin{theorem}\label{Theorem 2.1}
Suppose that the function \( G(x,t) \) satisfies conditions \eqref{D22} and \eqref{GG3}. Then the De Giorgi-type class belongs to a Hölder space:
\[
\mathcal{B}_{G(x,t)}(\Omega, M, \gamma, \gamma_1, \delta) \subset C^{0,\alpha}(\Omega),
\]
where the Hölder exponent \( \alpha \in (0,1] \) depends only on the parameters \( g^-,\allowbreak n,\allowbreak p^{\pm},\allowbreak q^{\pm},\allowbreak \bar\alpha,\allowbreak \bar{a},\allowbreak \Theta,\allowbreak d,\allowbreak \ell(1),\allowbreak L,\allowbreak M,\allowbreak C_*\text{ (if \eqref{GG6a} is used)},\allowbreak \gamma,\allowbreak \gamma_1 \), and \( \delta \).
\end{theorem}
\begin{definition}[see Sect. 1, Chapter 1 \cite{OL}]\label{Definition 2.3}
We will say that \(u \in \mathcal{B}_{G(x,t)}(\overline{\Omega}, M, \gamma, \gamma_1, \delta)\) if \(u \in \mathcal{B}_{G(x,t)}(\Omega, M, \gamma, \gamma_1, \delta)\) and the inequalities \eqref{2.4} hold also for arbitrary ball \(B_\rho\) with \(B_\rho \cap \partial \Omega \neq \emptyset\), \(\sigma \in (0, 1)\) and
\begin{equation}\label{eqlevel}
k \geq \max \left\{ \mathop{\mathrm{ess\,sup}}_{\Omega_\rho} w(x) - \delta M, \mathop{\mathrm{ess\,sup}}_{(\partial \Omega)_\rho} w(x) \right\},
\end{equation}
where \(\Omega_\rho := B_\rho \cap \Omega\), \((\partial \Omega)_\rho := \partial \Omega \cap B_\rho\), \(A_{k, \rho} := \{x \in \Omega_\rho: w(x) > k\}\) and the supremum over $(\partial\Omega)_\rho$ refers to the trace of $w$.
\end{definition}
\begin{definition}[see Sect. 1, of Chap. I of \cite{OL}]\label{Definition 2.4}
We will say that the boundary \(\partial \Omega\) of \(\Omega\) satisfies the condition \eqref{AO}, if there exist two positive constants \(m_0\) and \(\theta_0\) such that for any ball \(B_\rho\) with center on \(\partial \Omega\) and radius \(\rho \leq m_0\) and for any connected branch \(\Omega'_\rho\) of \(B_\rho \cap \Omega\), the following inequality hol{\rm d}s:
\begin{equation}\label{AO}
|\Omega'_\rho| \leq (1 - \theta_0)|B_\rho|.\tag{A}
\end{equation}
\end{definition}
Our second main result is the following theorem.
\begin{theorem}\label{Theorem 2.2}
Let \(u \in \mathcal{B}_{G(x,t)}(\overline{\Omega}, M, \gamma, \gamma_1, \delta)\). If \(G(x,t)\) satisfies conditions \eqref{D22} and \eqref{GG3}, \(\partial \Omega\) satisfies the condition \eqref{AO} and \(\overline{u}:=u_{|_{\partial \Omega}} \in C^{0,\alpha_1}(\partial \Omega)\), then \(u \in C^{0,\alpha}(\overline{\Omega})\) where \[\alpha = \alpha\bigl(n,\; p^{\pm},\; q^{\pm},\; g^{-},\; g^{+},\; \bar{a},\;
\bar\alpha,\;\Theta,\; d,\; \ell(1), \; L,\; M,\; C_*,\; \delta,\; \gamma,\; \gamma_1,\;
\theta_0,\; \alpha_1\bigr) \in (0,1].\]
\end{theorem}

As an application of Theorems~\ref{Theorem 2.1} and \ref{Theorem 2.2}, we investigate the Hölder continuity of weak solutions to the elliptic equation in divergence form:
\begin{equation}\label{P}
\operatorname{div} A(x, u, Du) + B(x, u, Du) = 0, \quad x \in \Omega, \tag{$\mathcal{P}$}
\end{equation}
with Dirichlet boundary condition:
\begin{equation}\label{PD}
\left\lbrace\begin{array}{rcl}
\operatorname{div} A(x, u, Du) + B(x, u, Du) &=& 0, \quad x \in \Omega,\\
u &=& 0, \quad \text{on } \partial\Omega,
\end{array}\right. \tag{$\mathcal{PD}$}
\end{equation}
or Neumann boundary condition:
\begin{equation}\label{PN}
\left\lbrace\begin{array}{rcl}
\operatorname{div} A(x, u, Du) + B(x, u, Du) &=& 0, \quad x \in \Omega,\\
A(x, u, Du)\cdot\nu &=& C(x,u), \quad \text{on } \partial\Omega,
\end{array}\right. \tag{$\mathcal{PN}$}
\end{equation}
where \( \nu(x) \) denotes the outward unit normal vector to \( \partial\Omega \) at \( x \in \partial\Omega \). The functions
\[
A : \Omega \times \mathbb{R} \times \mathbb{R}^n \rightarrow \mathbb{R}^n,\quad
B : \Omega \times \mathbb{R} \times \mathbb{R}^n \rightarrow \mathbb{R},\quad
C : \partial\Omega \times \mathbb{R} \rightarrow \mathbb{R}
\]
are Carathéodory functions satisfying the following \( G(x,t) \)-structure conditions:
\begin{align}
A(x, u, \eta) \cdot \eta &\geq a_4\, G(x,|\eta|) - a_5\, H_1(x,|u|) - a_6, \tag{$\mathcal{A}_1$}\label{6.31}\\
|A(x, u, \eta)| &\leq a_1\, g(x,|\eta|) + a_2\, P(x,|u|) + a_3,  \tag{$\mathcal{A}_2$}\label{6.41}\\
|B(x, u, \eta)| &\leq b_1\, M(x,|\eta|) + b_2\, h_1(x,|u|) + b_3, \tag{$\mathcal{B}$}\label{6.51}\\
|C(x, u)| &\leq c_1\, h_2(x,|u|) + c_2. \tag{$\mathcal{C}$}\label{6.61}
\end{align}
The auxiliary functions \( M(x,t) \) and \( P(x,t) \) are defined as:
\begin{align}
M(x,t) &= \widetilde{H_1}^{-1}(x,G(x,t)), \label{6.71}\\
P(x,t) &= \widetilde{G}^{-1}(x,H_1(x,t)). \label{6.81}
\end{align}
For the definition of $\widetilde{H_1}^{-1}$ see Definitions \ref{definv} and \ref{CF}. The constants \( a_i \), \( b_j \), and \( c_k \) are all positive, and we fix parameters \( F_1 > 1 \) and \( F_2 > 1 \). The functions 
\(\displaystyle{
H_1(x,t) = \int_0^t h_1(x,s)\, \mathrm{d}s\text{ and }H_2(x,t) = \int_0^t h_2(x,s)\, \mathrm{d}s,}
\) are generalized \textnormal{N}-functions
which satisfy the following growth conditions:
\begin{align}
F_1^{-1} \leq H_1(x,1) \leq F_1, &\qquad \text{for all } x \in \overline{\Omega}, \tag{$\mathcal{HDN}_1$} \label{H111}\\
g^+ < h_1^- \leq \frac{h_1(x,t)\, t}{H_1(x,t)} \leq h_1^+\leq g_*^-:=\frac{ng^-}{n-g^-}, &\qquad \text{for all } x \in \overline{\Omega}\text{ and } t > 0, \tag{$\mathcal{HDN}_2$} \label{H211}
\end{align}
\begin{equation}
 H_1(x,t)\mathop{\prec}_{\overline{\Omega}}  G^*(x,t), \tag{$\mathcal{HD}_1$}\label{YoungTriple1}
\end{equation}
\begin{equation}
\int_{\Omega} H_1(x,t) \, \mathrm{d}x < \infty,\ \text{for } t>0, \tag{$\mathcal{HD}_2$}\label{YoungTriple4}
\end{equation}
where the symbol "$\displaystyle{\mathop{\prec}_{\overline{\Omega}}}$" will be introduced in Definition \ref{prec}, and \(
\displaystyle{G^*(x,t) = \int_0^t g^*(x,s)\, \mathrm{d}s}
\)
is the Sobolev conjugate of \( G(x,t) \)  defined in  Definition \ref{DefCL}.
\begin{align}
F_2^{-1} \leq H_2(x,1) \leq F_2, &\qquad \text{for all } x \in \partial\Omega, \tag{$\mathcal{HN}_1$} \label{NH111}\\
g^+ < h_2^- \leq \frac{h_2(x,t)\, t}{H_2(x,t)} \leq h_2^+\leq g_\sharp^-:=\frac{(n-1)g^-}{n-g^-}, &\qquad \text{for all }x\in \partial\Omega  \text{ and } t>0,\tag{$\mathcal{HN}_2$} \label{NH211}
\end{align}
\begin{equation}
 H_1(x,t)\mathop{\prec}_{\overline{\Omega}} G_*(x,t)   \quad \text{and}\quad H_2(x,t)\mathop{\prec}_{\partial\Omega}\left[G_*(x,t)\right]^{\frac{n-1}{n}}. \tag{$\mathcal{HN}_3$}\label{NYoungTriple4}
\end{equation}
\begin{remark}
The generalized \textnormal{N}-functions \( H_1(x,t) \) and \( H_2(x,t) \) may exhibit critical growth corresponding to the conjugate \( G^*(x,t) \) ( or \( G_*(x,t) \) ). In such a case, we have \( h_1^+ = g_*^- = \frac{ng^-}{n-g^-} \) and \( h_2^+ = g_\sharp^- = \frac{(n-1)g^-}{n-g^-} \) (the latter being the critical exponent for the trace embedding, since $\frac{n-1}{n}g_*^-=g_\sharp^-$).
\end{remark}

\begin{definition}\label{Definition 4.1}
Under assumptions \eqref{D22}--\eqref{GG2}, \eqref{6.31}--\eqref{6.41}, \eqref{6.51}, \eqref{H111}--\eqref{H211} and \eqref{YoungTriple1}--\eqref{YoungTriple4}, a function \( u \in W^{1,G(x,t)}(\Omega) \) is called a \emph{weak solution} of problem~\eqref{PD} if
\begin{equation}\label{4.9}
\int_{\Omega} A(x, u, Du) \cdot Dv \, \mathrm{d}x - \int_{\Omega} B(x, u, Du) v \, \mathrm{d}x = 0, \tag{$\mathcal{FDV}$}
\end{equation}
for all test functions \( v \in W_0^{1,G(x,t)}(\Omega) \).
\end{definition}

\begin{definition}\label{Definition 4.1w}
Under assumptions \eqref{D22}--\eqref{GG1}, \eqref{GG4}--\eqref{GG5}, \eqref{6.31}--\eqref{6.41}, \eqref{6.51}, \eqref{6.61}, \eqref{H111}--\eqref{H211} and \eqref{NH111}--\eqref{NYoungTriple4}, a function \( u \in W^{1,G(x,t)}(\Omega) \) is said to be a \emph{weak solution} of problem~\eqref{PN} if
\begin{equation}\label{FNV}
\int_{\Omega} A(x, u, Du) \cdot Dv \, \mathrm{d}x - \int_{\partial\Omega} C(x, u)v \, \mathrm{d}\sigma - \int_{\Omega} B(x, u, Du) v \, \mathrm{d}x = 0, \tag{$\mathcal{FNV}$}
\end{equation}
for all \( v \in W^{1,G(x,t)}(\Omega) \), where \( \mathrm{d}\sigma \) denotes the surface measure on \( \partial\Omega \).
\end{definition}

Since the boundedness of weak solutions is essential for establishing their regularity, in particular, Hölder continuity, we first prove the boundedness of weak solutions to problem~\eqref{P} under both Dirichlet and Neumann conditions.

\begin{theorem}\label{Theorem 4.1w}
Assume that conditions \eqref{D22}--\eqref{GG2}, \eqref{6.31}--\eqref{6.41}, \eqref{6.51}, \eqref{H111}--\eqref{H211} and \eqref{YoungTriple1}--\eqref{YoungTriple4} hold. If \( u \in W^{1,G(x,t)}(\Omega) \) is a weak solution to problem~\eqref{PD}, then \( u \in L^{\infty}(\Omega) \).
\end{theorem}
\begin{theorem}\label{Theorem 6.1w}
Assume that conditions \eqref{D22}--\eqref{GG1}, \eqref{GG4}--\eqref{GG5}, \eqref{6.31}--\eqref{6.41}, \eqref{6.51}, \eqref{6.61}, \eqref{H111}--\eqref{H211} and \eqref{NH111}--\eqref{NYoungTriple4} hold. If \( u \in W^{1,G(x,t)}(\Omega) \) is a weak solution to problem~\eqref{PN}, then \( u \in L^{\infty}(\Omega) \cap L^{\infty}(\partial\Omega) \).
\end{theorem}
Combining the previous theorems, we establish the Hölder continuity of weak solutions to problem~\eqref{P}: interior Hölder continuity under Dirichlet boundary conditions~\eqref{PD} or Neumann boundary conditions~\eqref{PN}, and Hölder continuity up to the boundary for the Dirichlet problem.
\begin{theorem}\label{Theorem 4.2}
Assume that conditions \eqref{D22}--\eqref{GG3}, \eqref{6.31}--\eqref{6.41}, \eqref{6.51}, \eqref{H111}--\eqref{H211} and \eqref{YoungTriple1}--\eqref{YoungTriple4}, hold. Let \( u \in W^{1,G(x,t)}(\Omega) \) be a weak solution to the problem~\eqref{PD}, and suppose that
\begin{equation}\label{M}
\mathop{\mathrm{ess\,sup}}_{x \in \Omega} |u(x)| \leq M. \tag{$\mathcal{M}$}
\end{equation}
Then \( u \in \mathcal{B}_{G(x,t)}(\overline{\Omega}, M, \gamma, \gamma_1, \delta) \), where
\[
\gamma = \gamma\big(a_1,a_2, a_3, a_4, g^-, g^+\big), \quad
\gamma_1 = \gamma_1\big(a_2, a_3, a_4, a_5,a_6, b_1,b_2,b_3, g^+,h^+_1,F,F_1,M\big), \quad
\delta = 
\begin{cases}
\min\left\{ \dfrac{a_4}{4b_3M}, 2 \right\}, & \text{if } b_3 > 0, \\[1.2ex]
2, & \text{if } b_3 = 0.
\end{cases}
\]
\end{theorem}
Throughout this paper, we assume without loss of generality that $M \geq 1$. 
\begin{theorem}\label{Theorem 4.21}
Assume that conditions \eqref{D22}--\eqref{GG1}, \eqref{GG3}--\eqref{GG5}, \eqref{6.31}--\eqref{6.41}, \eqref{6.51}, \eqref{6.61}, \eqref{H111}--\eqref{H211} and \eqref{NH111}--\eqref{NYoungTriple4} hold. Let \( u \in W^{1,G(x,t)}(\Omega) \) be a weak solution to problem~\eqref{PN}, and suppose that \eqref{M} holds. Then \( u \in \mathcal{B}_{G(x,t)}(\Omega, M, \gamma, \gamma_1, \delta) \) (interior class), with \( \gamma, \gamma_1, \delta \) as in Theorem~\ref{Theorem 4.2}.
\end{theorem}

\begin{theorem}\label{Theorem 4.3}
Suppose that the assumptions of Theorem~\ref{Theorem 4.2} or Theorem~\ref{Theorem 4.21} are satisfied. Then \( u \in C^{0,\alpha}(\Omega) \), where \( \alpha \in (0,1] \) depends only on the data \( a_1,\dots,a_6, b_1,b_2,b_3, M, n, g^\pm, F, F_1\) (and on the data of \eqref{GG3}).
\end{theorem}

\begin{theorem}\label{Theorem 4.4}
Let \( u \in W^{1,G(x,t)}(\Omega) \) be a weak solution to problem~\eqref{PD} (respectively,~\eqref{PN}), under the assumptions of Theorem~\ref{Theorem 4.2} (respectively, Theorem~\ref{Theorem 4.21}). Then \( u \) is locally Hölder continuous in \( \Omega \). Moreover, for the Dirichlet problem~\eqref{PD}, if the boundary \( \partial \Omega \) satisfies condition~\eqref{AO}, then \( u \) is Hölder continuous up to the boundary, i.e., \( u \in C^{0,\alpha}(\overline{\Omega}) \) for some \( \alpha \in (0,1] \).
\end{theorem}
\begin{remark}
For the Neumann problem~\eqref{PN} the trace of $u$ is not known a priori to be H\"older continuous, and the boundary class of Definition~\ref{Definition 2.3} requires levels above the supremum of the trace. Hence boundary H\"older regularity for~\eqref{PN} is not addressed here; for~\eqref{PN} we obtain boundedness in $\Omega$ and on $\partial\Omega$ (Theorem~\ref{Theorem 6.1w}) and interior H\"older continuity.
\end{remark}
\subsection{Examples}\label{Examples}

The following examples illustrate the range of generalized
\textnormal{N}-functions covered by the structural assumptions
\eqref{D22}--\eqref{GG5}. In particular, we explain how the
two-pair representation in \eqref{GG3} is obtained and which
alternative, namely \eqref{GG6a} or \eqref{GG6b}, is used.
Here \(\ell\equiv1\) means \(\bar\delta=e\), \(\eta=0\), and
\(\ell(t)=\ln(e+t)\) means \(\bar\delta=e\), \(\eta=1\).
Recall that \eqref{GG3} stands for the whole list of conditions
(i)--(vi) of Subsection~\ref{sec1.2}. In every example we use the
symmetric choice \(a_i'=a_i\), \(b_i'=b_i\), \(p_i'=p_i\),
\(q_i'=q_i\) (and \(\Theta=1\)): then the two bounds in (i) coincide,
(iv) is the log-H\"older continuity of \(p_i\) and \(q_i\), and (v)
reduces to \(a_i,b_i\in C^{0,\bar\alpha}(\overline\Omega)\). All six
conditions (i)--(vi) are checked in each example.
Conditions \eqref{GG2} (needed for the Dirichlet problem) and
\eqref{GG5} (needed, together with \eqref{GG4}, for the Neumann problem)
are not automatic: they require a regular dependence of \(G\) on \(x\).
For each example we therefore state hypotheses ensuring them, or we
indicate that they have to be assumed; a criterion for \eqref{GG5}, based on
\cite[Proposition~3.1]{Fan2012a}, is given in Lemma~\ref{lem:G5crit}, and whenever possible we verify
its hypothesis \eqref{L:crit}. Without \eqref{GG2} and \eqref{GG5},
the regularity of the De Giorgi class (Theorems~\ref{Theorem 2.1}
and~\ref{Theorem 2.2}) is still available.

\begin{example}{\rm\bf{[\(p\)-growth conditions]}} \label{ex1}
Let \(p\in(1,n)\) and define
\[
G(x,t)=|t|^p.
\]
Then \(G\) is a generalized \textnormal{N}-function satisfying
\eqref{D22}--\eqref{GG5}. Indeed, we have
\[
\frac{t g(x,t)}{G(x,t)}=p,
\qquad t>0,
\]
so that \eqref{D22} holds with \(g^-=g^+=p\).
Moreover, \eqref{GG1} and \eqref{GG2} are immediate because
\(G\) is independent of \(x\).

To verify \eqref{GG3}, take \(\ell\equiv1\) and, for \(i=1,2\),
choose
\[
a_i'=a_i=\frac12,\qquad
b_i'=b_i=0,\qquad
p_i'=p_i=q_i'=q_i=p.
\]
Then the lower and upper bounds in \eqref{GG3} hold with equality,
and \eqref{GG6a} is satisfied for any
\(\bar\alpha\in(0,1]\).
The conditions (ii)--(v) in \eqref{GG3} hold as well: (ii) with
\(p^-=p^+=q^+=p\in(1,n)\); (iii) with \(d=\frac12\), because
\(a_i'+b_i'=\frac12\); (iv) and (v) trivially, because the exponents and
the coefficients are constant and \(b_i=b_i'\).

Furthermore, since
\(
G^{-1}(x,t)=t^{\frac{1}{p}},
\)
we obtain
\[
\int_0^1\frac{G^{-1}(x,t)}{t^{\frac{n+1}{n}}}\,\mathrm{d}t
=
\int_0^1 t^{\frac1p-\frac{n+1}{n}}\,\mathrm{d}t<\infty
\]
because \(p<n\). Hence \eqref{GG4} holds.

Finally, \(G_*\) is independent of \(x\); explicitly,
\[
G_*(x,t)=\Bigl(\frac{(n-p)\,t}{np}\Bigr)^{\frac{np}{n-p}},
\qquad T(x)=+\infty,
\]
and therefore \(\nabla_xG_*(x,t)=0\).
Thus \eqref{GG5} holds.

Consequently, our results apply to the classical \(p\)-growth
case, which has been extensively investigated by
Lady{\v{z}}enskaja--Solonnikov--Ural{\cprime}ceva \cite{OL}
and G.~Lieberman~\cite{Li1988}.
\end{example}
\begin{example}{\rm\bf{[Orlicz-type \(G(t)=t^p+t^q\)]}} \label{ex2}
Let \(1<p\leq q<n\) such that \(q-p\leq\bar\alpha\in (0,1]\) , and define
\[
G(x,t)=|t|^p+|t|^q,
\qquad x\in\Omega,\quad t\in\mathbb R.
\]
Then \(G\) is a generalized \textnormal{N}-function satisfying
\eqref{D22}--\eqref{GG5}.

Indeed,
\[
p\leq\frac{t g(x,t)}{G(x,t)}\leq q,
\qquad t>0,
\]
and \eqref{GG1}--\eqref{GG2} are immediate because \(G\)
is independent of \(x\).

For \eqref{GG3}, take \(\ell\equiv1\) and choose, for
\(i=1,2\),
\[
a_i'=a_i=b_i'=b_i=\frac12,
\qquad
p_i'=p_i=p,\qquad q_i'=q_i=q.
\]
Then
\[
\begin{aligned}
\sum_{i=1}^{2}
\left[a_i't^{p_i'}+b_i't^{q_i'}\ell(t)\right]
&=t^p+t^q=\sum_{i=1}^{2}
\left[a_it^{p_i}+b_it^{q_i}\ell(t)\right].
\end{aligned}
\]
Moreover, \(a_i'+b_i'=1\), and \eqref{GG6a} holds, since
\(q-p\leq\bar\alpha\) implies
\[
\sup_{t\geq1}t^{q-p-\bar\alpha}\leq1.
\]
(Also \eqref{GG6b} holds, with \(\beta_0=\frac12\).)
The conditions (ii), (iv) and (v) in \eqref{GG3} hold as well: (ii)
with \(p^-=p^+=p>1\) and \(q^+=q<n\) (recall \(p\leq q\)); (iv) and (v)
trivially, because the exponents and the coefficients are constant and
\(b_i=b_i'\).

Furthermore, \eqref{GG4} follows from
\(G^{-1}(x,t)\asymp t^{\frac{1}{p}}\) near zero and \(p<n\).
Since \(G\) is independent of \(x\), \eqref{GG5} is immediate.

Consequently, our results apply to the Orlicz case with two
power terms. Various aspects of this case have been studied
by G.~Lieberman~\cite{Li1991}.
\end{example}
\begin{example}{\rm\bf{[\(p(x)\)-growth conditions]}} \label{ex3}
Let \(p:\Omega\to\mathbb R\) be a $C^{0,1}_{\mathrm{loc}}(\overline{\Omega})$ function satisfying
\begin{equation}\label{var111}
1<p^-:=\inf_{x\in\overline\Omega}p(x)
\leq p^+:=\sup_{x\in\overline\Omega}p(x)<n.
\end{equation}
Assume that \(\frac{1}{p}\) satisfies the log-H\"older continuity condition
\begin{equation}\label{var1}
\left|\frac1{p(x)}-\frac1{p(y)}\right|
\leq
\frac{c}{\ln\left(e+\frac1{|x-y|}\right)}
\quad\text{for all }x,y\in\overline\Omega,
\end{equation}
for some constant \(c>0\). Define
\[
G(x,t)=|t|^{p(x)}.
\]
Then \(G\) is a generalized \textnormal{N}-function satisfying
\eqref{D22}--\eqref{GG5}.

Indeed,
\[
\frac{t g(x,t)}{G(x,t)}=p(x),
\]
so \eqref{D22} follows from \eqref{var111}.
Moreover, \eqref{GG1} holds with \(F=1\), and
\eqref{GG2} follows from the log-H\"older continuity in
\eqref{var1}.

For \eqref{GG3}, take \(\ell\equiv1\) and, for \(i=1,2\),
choose
\[
a_i'=a_i=\frac12,\qquad
b_i'=b_i=0,\qquad
p_i'=p_i=q_i'=q_i=p(x).
\]
Then \eqref{GG3} holds with equality, and \eqref{GG6a}
is satisfied for any \(\bar\alpha\in(0,1]\). Condition (iv) in \eqref{GG3} is
the log-H\"older continuity of \(p\), which follows from \eqref{var1}
because \(p^-\leq p\leq p^+\).
The conditions (ii), (iii) and (v) in \eqref{GG3} hold as well: (ii)
with \(p^-=\inf p>1\) and \(p^+=q^+=\sup p<n\), by \eqref{var111};
(iii) with \(d=\frac12\), because \(a_i'+b_i'=\frac12\); (v) trivially,
because the coefficients are constant and \(b_i=b_i'\).
Condition \eqref{GG4} follows from \(p^+<n\).

Finally, \(\nabla_xG(x,t)=\ln t\,\nabla p(x)\,G(x,t)\) and
\(|\ln t|\leq\frac{\bigl(1+|\ln G(x,t)|\bigr)}{p^-}\) (because
\(G(x,t)=t^{p(x)}\) and \(p(x)\geq p^-\)), so that \eqref{L:crit} holds with \(\varepsilon=0\) and
\(K=\frac{\|\nabla p\|_\infty}{p^-}\). Hence \eqref{GG5} follows from
Lemma~\ref{lem:G5crit}.

Therefore, our results also apply to the variable-exponent
setting, which has been extensively studied by
X.~Fan and D.~Zhao~\cite{Fan1999}.
\end{example}
\begin{example}{\rm\bf{[Logarithmic growth conditions]}} \label{ex4}
Let \(p:\Omega\to\mathbb R\) be a $C^{0,1}_{\mathrm{loc}}(\overline{\Omega})$ function satisfying \eqref{var111}--\eqref{var1}
and assume in addition that
\[
p^++\varphi_{\max}<n,
\qquad
\varphi_{\max}:=
\sup_{t>0}\frac{t}{(e+t)\ln(e+t)}.
\]
Define
\[
G(x,t)=|t|^{p(x)}\ln(e+|t|).
\]
Then \(G\) is a generalized \textnormal{N}-function satisfying
\eqref{D22}--\eqref{GG5}.

Indeed, for \(t>0\),
\[
\frac{t g(x,t)}{G(x,t)}
=
p(x)+\frac{t}{(e+t)\ln(e+t)}.
\]
Consequently,
\[
p^-\leq\frac{t g(x,t)}{G(x,t)}
\leq p^++\varphi_{\max}<n,
\]
which gives \eqref{D22}. Also,
\[
G(x,1)=\ln(e+1),
\]
so \eqref{GG1} holds, while \eqref{GG2} follows from
\eqref{var1}.

For \eqref{GG3}, choose
\(
\ell(t)=\ln(e+t)
\)
and, for \(i=1,2\), set
\[
a_i'=a_i=0,\qquad
b_i'=b_i=\frac12,\qquad
p_i'=p_i=q_i'=q_i=p(x).
\]
Then the lower and upper bounds in \eqref{GG3} hold with
equality (condition (iv) holds as in Example~\ref{ex3}). Moreover, \eqref{GG6b} holds with
\(\beta_0=\frac{1}{2}\).
The conditions (ii), (iii) and (v) in \eqref{GG3} hold as well: (ii)
with \(p^-=\inf p>1\) and \(p^+=q^+=\sup p<n\), by \eqref{var111};
(iii) with \(d=\frac12\), because \(a_i'=0\) and \(b_i'=\frac12\);
(v) trivially, because the coefficients are constant and \(b_i=b_i'\).

Condition \eqref{GG4} follows from
\(
G(x,t)\sim t^{p(x)}
\qquad\text{as }t\to0^+,
\)
and \(p^+<n\).

Finally, \(\nabla_xG(x,t)=\ln t\,\nabla p(x)\,G(x,t)\) (the factor
\(\ln(e+t)\) does not depend on \(x\)), and
\(|\ln t|\leq\frac{\bigl(1+|\ln G(x,t)|\bigr)}{p^-}\): for \(t\geq1\) because
\(G(x,t)\geq t^{p^-}\), and for \(t<1\) because
\(G(x,t)\leq\ln(e+1)\,t^{p^-}\) and \(\ln\ln(e+1)<1\). Thus \eqref{L:crit}
holds with \(\varepsilon=0\) and \(K=\frac{\|\nabla p\|_\infty}{p^-}\), and \eqref{GG5} follows from
Lemma~\ref{lem:G5crit}.

This example extends the classical variable-exponent case
by incorporating a logarithmic perturbation.
\end{example}

\begin{example}{\rm\bf{[\((p(x),q(x))\)-growth conditions]}} \label{ex5}
Let \(p,q:\Omega\to\mathbb R\) satisfy
\begin{equation}\label{var20}
p,q\in C^{0,1}_{\mathrm{loc}}(\overline\Omega),
\end{equation}
\begin{equation}\label{var21}
1<p(x)<q(x)<n
\quad\text{for all }x\in\overline\Omega,
\end{equation}
and
\begin{equation}\label{var22}
\frac{q^+}{p^-}<1+\frac1n,
\qquad
p^-:=\inf_{\overline\Omega}p,\quad
q^+:=\sup_{\overline\Omega}q.
\end{equation}
Define
\[
G(x,t)=|t|^{p(x)}+|t|^{q(x)}.
\]
Then \(G\) is a generalized \textnormal{N}-function satisfying
\eqref{D22}--\eqref{GG5}.

Indeed,
\[
p^-\leq\frac{t g(x,t)}{G(x,t)}\leq q^+<n,
\]
and \eqref{GG1} holds with \(F=2\). The Lipschitz continuity
in \eqref{var20} gives the inverse comparability in
\eqref{GG2}.

For \eqref{GG3}, take \(\ell\equiv1\) and, for \(i=1,2\),
choose
\[
a_i'=a_i=b_i'=b_i=\frac12,\qquad
p_i'=p_i=p(x),\qquad q_i'=q_i=q(x).
\]
Then \eqref{GG3} holds with equality (condition (iv) holds because \(p\) and
\(q\) are Lipschitz). Moreover, by
\eqref{var22},
\[
q(x)-p(x)\leq q^+-p^-
<
\frac{p^-}{n}<1.
\]
Thus \eqref{GG6a} holds with \(\bar\alpha=1\)
(\eqref{GG6b} holds as well, with \(\beta_0=\frac12\)).
The conditions (ii), (iii) and (v) in \eqref{GG3} hold as well: (ii)
with \(p^-=\inf p>1\), \(p^+=\sup p\) and \(q^+=\sup q<n\), by
\eqref{var21} and the continuity of \(p\) and \(q\) on
\(\overline\Omega\) (recall \(p<q\)); (iii) with \(d=1\), because
\(a_i'+b_i'=1\); (v) trivially, because the coefficients are constant and
\(b_i=b_i'\).

Condition \eqref{GG4} follows from \(p^+<n\), because
\(G(x,t)\geq t^{p(x)}\) and hence \(G^{-1}(x,\tau)\leq\tau^{\frac{1}{p^+}}\) for \(\tau\leq1\).
Finally, \(|\nabla_xG(x,t)|\leq L\,|\ln t|\,G(x,t)\) with
\(L:=\max\{\|\nabla p\|_\infty,\|\nabla q\|_\infty\}\), and
\(|\ln t|\leq\frac{\bigl(1+|\ln G(x,t)|\bigr)}{p^-}\) (for \(t\geq1\) because
\(G\geq t^{p^-}\), for \(t<1\) because \(G\leq2t^{p^-}\)); so
\eqref{L:crit} holds with \(\varepsilon=0\), and \eqref{GG5} follows from
Lemma~\ref{lem:G5crit}.
\end{example}
\begin{example}{\rm\bf{[Double-phase growth conditions]}} \label{ex6}
Let \(1<p<q<n\) and let
\(a\in C^{0,\bar\alpha}(\overline\Omega)\) be nonnegative
and bounded. Assume
\begin{equation}\label{double2}
q\leq\frac{(n+1)p}{n},
\end{equation}
and
\begin{equation}\label{double3bis}
\bar\alpha:=\frac np(q-p)\in(0,1]
\end{equation}
(since \(q>p\), \eqref{double3bis} is equivalent to \eqref{double2}).
Define
\[
G(x,t)=|t|^p+a(x)|t|^q.
\]
Then \(G\) is a generalized \textnormal{N}-function satisfying
\eqref{D22}--\eqref{GG4}. Moreover, \eqref{GG5} holds as well provided
that \(a\in C^{0,1}_{\mathrm{loc}}(\overline\Omega)\) and that either
\(q<\frac{(n+1)p}{n}\) (that is, \(\bar\alpha<1\)) or
\(a\geq a_0>0\) in \(\overline\Omega\).

Indeed,
\[
p\leq\frac{t g(x,t)}{G(x,t)}\leq q<n,
\]
and \eqref{GG1} follows from the boundedness of \(a\), because
\(1\leq G(x,1)=1+a(x)\leq1+\|a\|_\infty\).

The condition \eqref{GG2} follows from the fact that  \(a\in C^{0,\bar\alpha}(\overline\Omega)\) and  the gap condition \eqref{double3bis}.

For \eqref{GG3}, take \(\ell\equiv1\) and, for \(i=1,2\),
choose
\[
a_i'=a_i=\frac12,\qquad
b_i'=b_i=\frac12a(x),\qquad
p_i'=p_i=p,\qquad q_i'=q_i=q.
\]
Then
\[
\sum_{i=1}^{2}
\left[a_i't^{p_i'}+b_i't^{q_i'}\ell(t)\right]
=t^p+a(x)t^q=G(x,t),
\]
and \(a_i'+b_i'=\frac{1+a(x)}{2}\geq\frac{1}{2}\).
Furthermore,
\[
q-p-\bar\alpha
=(q-p)\left(1-\frac np\right)<0,
\]
because \(p<n\). Hence \eqref{GG6a} holds.
The conditions (ii), (iv) and (v) in \eqref{GG3} hold as well: (ii)
with \(p^-=p^+=p>1\) and \(q^+=q<n\) (recall \(p<q\)); (iv) trivially,
because the exponents are constant; (v) because \(a_i'=\frac12\),
\(b_i'=\frac12a\in C^{0,\bar\alpha}(\overline\Omega)\) and
\(b_i=b_i'\).

Condition \eqref{GG4} follows from the behavior of \(G(x,t)\)
near zero, where \(G(x,t)\sim t^p\).

Finally, let \(a\in C^{0,1}_{\mathrm{loc}}(\overline\Omega)\). Then
\(|\nabla_xG(x,t)|=|\nabla a(x)|\,t^q\). If \(a\geq a_0>0\),
then \(|\nabla_xG|\leq\left(\frac{\|\nabla a\|_\infty}{a_0}\right)\,G\), so that
\eqref{L:crit} holds with \(\varepsilon=0\). In general (\(a\) may vanish),
\(t^q\leq t^p\leq G\) for \(t\leq1\), while for \(t\geq1\) we have
\(t^p\leq G\) and hence \(t^q=t^{q-p}t^p\leq G^{\frac{q-p}{p}}G\); thus
\eqref{L:crit} holds with \(\varepsilon=\frac{q-p}{p}=\frac{\bar\alpha}{n}\),
and \(\varepsilon<\frac1n\) exactly when \(\bar\alpha<1\).
In both cases \eqref{GG5} follows from Lemma~\ref{lem:G5crit}.
(For \(\bar\alpha=1\) and \(a\) vanishing, \(\varepsilon=\frac1n\) is the
borderline case, which is not covered by Lemma~\ref{lem:G5crit}; there
\eqref{GG5} has to be assumed.)
\end{example}

\begin{example}{\rm\bf{[Logarithmic double-phase growth conditions]}} \label{ex7}
Let \(1<p<q<n\), let \(\bar\alpha\in(0,1]\) and let
\(a\in C^{0,\bar\alpha}(\overline\Omega)\) be nonnegative
and bounded. Assume the strict gap condition
\begin{equation}\label{double3log}
\frac np(q-p)<\bar\alpha
\end{equation}
(in particular \(q<\frac{(n+1)p}{n}\)) and
\[
q+\varphi_{\max}<n,
\qquad
\varphi_{\max}:=
\sup_{t>0}\frac{t}{(e+t)\ln(e+t)}.
\]
Define
\[
G(x,t)=|t|^p+a(x)|t|^q\ln(e+|t|).
\]
Then \(G\) is a generalized \textnormal{N}-function satisfying
\eqref{D22}--\eqref{GG4}. If, in addition, \(a\in C^{0,1}_{\mathrm{loc}}(\overline\Omega)\),
then \eqref{GG5} holds as well.

Indeed,
\[
\frac{t g(x,t)}{G(x,t)}
=
\frac{
pt^p+a(x)t^q\ln(e+t)
\left(q+\dfrac{t}{(e+t)\ln(e+t)}\right)
}{
t^p+a(x)t^q\ln(e+t)
}.
\]
It follows that
\[
p\leq\frac{t g(x,t)}{G(x,t)}
\leq q+\varphi_{\max}<n.
\]
Moreover, \(1\leq G(x,1)\leq1+\|a\|_\infty\ln(e+1)\), which gives \eqref{GG1}.

The condition \eqref{GG2} follows from the fact that  \(a\in C^{0,\bar\alpha}(\overline\Omega)\) and  the gap condition \eqref{double3log}.

For \eqref{GG3}, choose
\(
\ell(t)=\ln(e+t)
\)
and, for \(i=1,2\), take
\[
a_i'=a_i=\frac12,\qquad
b_i'=b_i=\frac12a(x),\qquad
p_i'=p_i=p,\qquad q_i'=q_i=q.
\]
Then the lower and upper bounds in \eqref{GG3} hold with
equality. Moreover, by \eqref{double3log}, \(q-p-\bar\alpha<0\),
and consequently
\[
\sup_{t\geq1}
t^{q-p-\bar\alpha}\ln(e+t)<\infty.
\]
Thus \eqref{GG6a} holds.
The conditions (ii)--(v) in \eqref{GG3} hold as well: (ii) with
\(p^-=p^+=p>1\) and \(q^+=q<n\) (recall \(p<q\) and
\(q+\varphi_{\max}<n\)); (iii) with \(d=\frac12\), because
\(a_i'+b_i'=\frac{1+a(x)}{2}\geq\frac12\); (iv) trivially, because the
exponents are constant; (v) because \(a_i'=\frac12\),
\(b_i'=\frac12a\in C^{0,\bar\alpha}(\overline\Omega)\) and
\(b_i=b_i'\).

Condition \eqref{GG4} follows from \(G(x,t)\sim t^p\)
as \(t\to0^+\).

Finally, let \(a\in C^{0,1}_{\mathrm{loc}}(\overline\Omega)\). Then
\(|\nabla_xG|=|\nabla a|\,t^q\ln(e+t)\). For \(t\leq1\) this is at most
\(\|\nabla a\|_\infty\ln(e+1)\,t^p\leq\|\nabla a\|_\infty\ln(e+1)\,G\);
for \(t\geq1\), \(t^p\leq G\), so
\[t^q\ln(e+t)\leq G^{1+\frac{q-p}{p}}\ln(e+G)\leq2\,G^{1+\frac{q-p}{p}}(1+\ln G).\] Hence
\eqref{L:crit} holds with \(\varepsilon=\frac{q-p}{p}\), which is smaller
than \(\frac1n\) by \eqref{double3log} (recall \(\bar\alpha\leq1\)), and
\eqref{GG5} follows from Lemma~\ref{lem:G5crit}.

This growth condition was also recently studied by Ho, Kim, and Winkert~\cite{kyho}.
\end{example}

\begin{example}{\rm\bf{[Variable-exponent double phase with logarithm]}}
\label{ex8}
Let \(p,q:\Omega\to\mathbb R\) satisfy
\[
1<p^-:=\inf_{\overline\Omega}p
\leq p(x)<q(x)\leq q^+:=\sup_{\overline\Omega}q,
\qquad
q^++\varphi_{\max}<n,
\]
where
\[
\varphi_{\max}:=
\sup_{t>0}\frac{t}{(e+t)\ln(e+t)}.
\]
Assume that \(p\) and \(q\) are log-H\"older continuous, namely
\[
\operatorname{osc}_{B_R(x_0)}p
\leq\frac{L_0}{|\ln R|},
\qquad
\operatorname{osc}_{B_R(x_0)}q
\leq\frac{L_0}{|\ln R|}
\]
for all \(x_0\in\overline\Omega\) and \(R<R_{\mathrm G}\) (this is exactly
\eqref{GG3}(iv) for the choice made below).
Let \(a\in C^{0,\bar\alpha}(\overline\Omega)\) be
nonnegative and bounded, \(\bar\alpha\in(0,1]\), and assume that one of
the following two alternatives holds:
\begin{enumerate}
\item[(a)] gap condition
\[
\bar\alpha>n\sup_{x\in\overline\Omega}\frac{q(x)-p(x)}{p(x)};
\]
\item[(b)] nondegenerate coefficient
\(a(x)\geq a_0>0\) for all \(x\in\overline\Omega\).
\end{enumerate}
Define
\[
G(x,t)=|t|^{p(x)}
+a(x)|t|^{q(x)}\ln(e+|t|).
\]
Then \(G\) is a generalized \textnormal{N}-function satisfying
\eqref{D22}--\eqref{GG4}. If, in addition, \(p,q,a\in C^{0,1}_{\mathrm{loc}}(\overline\Omega)\),
then \eqref{GG5} holds as well.

Indeed, \(\frac{tg(x,t)}{G(x,t)}\) is a weighted average of \(p(x)\)
and \(q(x)+\frac{t}{(e+t)\ln(e+t)}\), so that
\(p^-\leq\frac{tg}{G}\leq q^++\varphi_{\max}<n\), and
\(1\leq G(x,1)\leq1+\|a\|_\infty\ln(e+1)\).

The condition \eqref{GG2} follows from the log-H\"older continuity of \(p\) and \(q\) and, in case~(a), from the fact that \(a\in C^{0,\bar\alpha}(\overline\Omega)\) and the gap condition, while in case~(b) it follows from \(a\geq a_0>0\).

For \eqref{GG3}, take
\[
\ell(t)=\ln(e+t)
\]
and, for \(i=1,2\), choose
\[
a_i'=a_i=\frac12,\qquad
b_i'=b_i=\frac12a(x),\qquad
p_i'=p_i=p(x),\qquad
q_i'=q_i=q(x).
\]
Then \eqref{GG3} holds with equality.
Moreover, (ii) holds with \(p^-=\inf p>1\), \(p^+=\sup p\) and
\(q^+=\sup q<n\) (recall \(p<q\)); (iii) holds with \(d=\frac12\),
because \(a_i'+b_i'=\frac{1+a(x)}{2}\geq\frac12\); (iv) is the
log-H\"older continuity of \(p\) and \(q\) assumed above; and (v) holds
because \(a_i'=\frac12\), \(b_i'=\frac12a\in C^{0,\bar\alpha}(\overline\Omega)\)
and \(b_i=b_i'\). Condition (vi) is verified next, in the cases~(a)
and~(b).

In case~(a), let \(\bar{\gamma}:=\sup_{\overline\Omega}\frac{q-p}{p}\), so that \(n\bar{\gamma}<\bar\alpha\leq1\). Then the alternative \eqref{GG6a} holds. Indeed, since
\(p(x)<q(x)\leq q^+<n\) we have \(q(x)-p(x)\leq\bar{\gamma} p(x)\leq\bar{\gamma} q^+<n\bar{\gamma}<\bar\alpha\),
so that \(\varepsilon_0:=\bar\alpha-\bar{\gamma} q^+>0\) satisfies
\[
q(x)-p(x)-\bar\alpha\leq-\varepsilon_0
\]
for all \(x\in\overline\Omega\), and hence
\[
\sup_{\substack{x\in\overline\Omega\\t\geq1}}
t^{q(x)-p(x)-\bar\alpha}\ln(e+t)<\infty.
\]
The non-strict condition
\(\sup_{\overline\Omega}(q-p)\leq\bar\alpha\)
does not suffice in general, because at equality
\(\ln(e+t)\) is unbounded as \(t\to\infty\).

In case~(b), the alternative \eqref{GG6b} holds, since
\(b_i'=\frac{a(x)}{2}\geq \frac{a_0}{2}\).

Condition \eqref{GG4} follows from the behavior
\(G(x,t)\sim t^{p(x)}\) as \(t\to0^+\).

For \eqref{GG5}, let \(p,q,a\in C^{0,1}_{\mathrm{loc}}(\overline\Omega)\) (this implies the
log-H\"older continuity above). Then
\[
\nabla_xG=\ln t\,\Bigl(\nabla p\,t^{p}+a\nabla q\,t^{q}\ln(e+t)\Bigr)
+\nabla a\,t^q\ln(e+t),
\]
and, as in Examples~\ref{ex3} and~\ref{ex4},
\(|\ln t|\leq C(1+|\ln G|)\) (use \(G\geq t^{p^-}\) for \(t\geq1\) and
\(G\leq c\,t^{p^-}\) for \(t<1\)), so that the first part is bounded by
\(K(1+|\ln G|)G\). In case~(b), the last term is at most
\(\left(\frac{\|\nabla a\|_\infty}{a_0}\right)\,G\), and \eqref{L:crit} holds with
\(\varepsilon=0\). In case~(a), the last term is at most \(K(1+\ln G)G^{1+\bar{\gamma}}\)
for \(t\geq1\) and at most \(KG\) for \(t\leq1\) (as in Example~\ref{ex7}),
so that \eqref{L:crit} holds with \(\varepsilon=\bar{\gamma}<\frac1n\).
Thus \eqref{GG5} follows from Lemma~\ref{lem:G5crit}.

This case was also studied by Ho, Kim, and Winkert~\cite{kyho}.
\end{example}

\begin{example}{\rm\bf{[Two-phase growth conditions]}} \label{ex9}
Let \(1<p_1<p_2<n\), and let
\(a_1,a_2\in C^{0,\bar\alpha}(\overline\Omega)\) be
nonnegative and bounded. Assume
\begin{equation}\label{multi1}
0<d_0\leq a_1(x)+a_2(x)\leq A_0<\infty
\quad\text{for all }x\in\overline\Omega,
\end{equation}
and
\begin{equation}\label{multi3}
\frac{n}{p_1}\,(p_2-p_1)\leq\bar\alpha,
\qquad \bar\alpha\in(0,1].
\end{equation}
Consider
\begin{equation}\label{multi2}
G(x,t)=\sum_{j=1}^{2}a_j(x)|t|^{p_j}.
\end{equation}
Then \(G\) is a generalized \textnormal{N}-function satisfying
\eqref{D22}--\eqref{GG4}. Moreover, \eqref{GG5} holds as well if
\(a_1,a_2\in C^{0,1}_{\mathrm{loc}}(\overline\Omega)\) and \(a_1,a_2\geq a_0>0\).

Indeed,
\[
p_1\leq\frac{t g(x,t)}{G(x,t)}\leq p_2<n,
\]
and \eqref{GG1} follows from \eqref{multi1}, since
\(G(x,1)=a_1(x)+a_2(x)\).

The condition \eqref{GG2} follows from the fact that  \(a_1,a_2\in C^{0,\bar\alpha}(\overline\Omega)\) and  the  conditions \eqref{multi1}--\eqref{multi3}.

For \eqref{GG3}, take \(\ell\equiv1\) and use the symmetric
pairing
\[
a_i'=a_i=\frac12a_1,\qquad
b_i'=b_i=\frac12a_2,\qquad
p_i'=p_i=p_1,\qquad
q_i'=q_i=p_2,
\qquad i=1,2.
\]
Then
\[
\sum_{i=1}^{2}
\left[a_i't^{p_i'}+b_i't^{q_i'}\ell(t)\right]
=a_1(x)t^{p_1}+a_2(x)t^{p_2}=G(x,t),
\]
so the bounds in \eqref{GG3} hold with equality.
Moreover,
\[
a_i'(x)+b_i'(x)
=\frac12(a_1(x)+a_2(x))
\geq\frac{d_0}{2}>0.
\]
Finally, \eqref{multi3} gives
\[
q_i-p_i-\bar\alpha
=p_2-p_1-\bar\alpha\leq-\Bigl(\frac{n}{p_1}-1\Bigr)(p_2-p_1)<0,
\]
and therefore \eqref{GG6a} holds.
The conditions (ii), (iv) and (v) in \eqref{GG3} hold as well: (ii)
with \(p^-=p^+=p_1>1\) and \(q^+=p_2<n\); (iv) trivially, because the
exponents are constant; (v) because \(a_i'=\frac12a_1\) and
\(b_i'=\frac12a_2\) belong to \(C^{0,\bar\alpha}(\overline\Omega)\), and
\(b_i=b_i'\).

Condition \eqref{GG4} follows from the behavior of \(G\) near
zero: by \eqref{multi1}, \(G(x,t)\geq d_0t^{p_2}\) for \(t\leq1\), so that
\(G^{-1}(x,\tau)\leq\left(\frac{\tau}{d_0}\right)^{\frac{1}{p_2}}\) for small \(\tau\), and \(p_2<n\).

For \eqref{GG5}, if \(a_1,a_2\in C^{0,1}_{\mathrm{loc}}(\overline\Omega)\) and
\(a_1,a_2\geq a_0>0\), then
\[|\nabla_xG|\leq\sum_j\|\nabla a_j\|_\infty t^{p_j}\leq\left(\frac{L}{a_0}\right)\,G,\] with
\(L:=\max_j\|\nabla a_j\|_\infty\), so that
\eqref{L:crit} holds with \(\varepsilon=0\) and \eqref{GG5} follows from
Lemma~\ref{lem:G5crit}. If the coefficients may vanish, then
\(\frac{|\nabla_xG|}{G}\) is no longer controlled for small \(t\), and
\eqref{GG5} has to be assumed or verified separately.

The genuinely multi-phase setting with three or more
independent phases, studied by De~Filippis and Oh
\cite{DeFilippisOh2019}, is not automatically covered by
the two-pair representation \eqref{GG3}.
\end{example}
\begin{example}{\rm\bf{[Two-phase with pure and logarithmic powers]}}
\label{ex11}
Let \(1<p_1<p_2<n\) and \(1<q_1<q_2<n\), and assume
in addition that
\[
p_j\leq q_j,\qquad j=1,2.
\]
Let \(a_1,a_2,b_1,b_2\in C^{0,\bar\alpha}(\overline\Omega)\)
be bounded functions satisfying
\[
a_j(x)\geq a_0>0,\qquad
b_j(x)\geq a_0>0,
\qquad j=1,2,
\]
for all \(x\in\overline\Omega\), where
\(\bar\alpha\in(0,1]\).
Set
\[
q^+:=\max_{j=1,2}\max\{p_j,q_j\}
\]
and assume
\[
q^++\varphi_{\max}<n,
\qquad
\varphi_{\max}:=
\sup_{t>0}\frac{t}{(e+t)\ln(e+t)}.
\]
Consider
\[
G(x,t)=
\sum_{j=1}^{2}
\left[
a_j(x)|t|^{p_j}
+b_j(x)|t|^{q_j}\ln(e+|t|)
\right].
\]
Then \(G\) is a generalized \textnormal{N}-function satisfying
\eqref{D22}--\eqref{GG4}. Moreover, \eqref{GG5} holds as well if
\(a_j,b_j\in C^{0,1}_{\mathrm{loc}}(\overline\Omega)\), \(j=1,2\).

Indeed, the logarithmic derivative is a weighted average
of the exponents \(p_j\) and
\[
q_j+\frac{t}{(e+t)\ln(e+t)}.
\]
Consequently,
\[
\min\{p_1,p_2\}
\leq\frac{t g(x,t)}{G(x,t)}
\leq q^++\varphi_{\max}<n,
\]
which gives \eqref{D22}. The boundedness and positive
lower bounds of the coefficients yield \eqref{GG1}.

For \eqref{GG2}, no continuity of the coefficients is needed: with
\(A:=\max_j\max\{\sup a_j,\sup b_j\}\) and the \(x\)-independent function
\[
\Phi(t):=\sum_{j=1}^2\bigl[t^{p_j}+t^{q_j}\ln(e+t)\bigr]
\]
we have \(a_0\Phi\leq G(x,\cdot)\leq A\Phi\) for all \(x\), hence
\(\Phi^{-1}\left(\frac{s}{A}\right)\leq G^{-1}(x,s)\leq\Phi^{-1}\left(\frac{s}{a_0}\right)\). Since
\(\Phi(\lambda t)\geq\lambda^{p_1}\Phi(t)\) for \(\lambda\geq1\), we get
\(\Phi^{-1}\left(\frac{s}{a_0}\right)\leq\left(\frac{A}{a_0}\right)^{\frac{1}{p_1}}\Phi^{-1}\left(\frac{s}{A}\right)\), and therefore
\(G^{-1}(x,s)\leq\left(\frac{A}{a_0}\right)^{\frac{1}{p_1}}G^{-1}(y,s)\) for all \(x,y\) and \(s\).
Thus \eqref{GG2} holds with \(\mu=\left(\frac{a_0}{A}\right)^{\frac{1}{p_1}}\).

For \eqref{GG3}, take
\[
\ell(t)=\ln(e+t)
\]
and use the pairing
\[
\begin{aligned}
&a_i'=a_i=a_i(x),\qquad
b_i'=b_i=b_i(x),\\
&p_i'=p_i=p_i,\qquad
q_i'=q_i=q_i,\qquad i=1,2.
\end{aligned}
\]
Then the lower and upper bounds in \eqref{GG3} hold
with equality. Moreover,
\[
a_i'(x)+b_i'(x)\geq2a_0>0,
\]
and
\[
\beta_0
=\min_{i=1,2}\inf_{\overline\Omega}b_i'(x)
\geq a_0>0.
\]
Thus \eqref{GG6b} holds.
The conditions (ii), (iv) and (v) in \eqref{GG3} hold as well: (ii)
because \(p_j\leq q_j\), with \(p^-=p_1>1\), \(p^+=p_2\) and
\(q^+=\max_jq_j<n\) (this is the number \(q^+\) defined above, since
\(p_j\leq q_j\)); (iv) trivially, because the exponents are constant;
(v) because \(a_i'=a_i\) and \(b_i'=b_i\) belong to
\(C^{0,\bar\alpha}(\overline\Omega)\), and \(b_i=b_i'\).

Condition \eqref{GG4} follows from the behavior of \(G(x,t)\)
near zero: \(G(x,t)\geq a_0t^{p_2}\) for \(t\leq1\), so that
\(G^{-1}(x,\tau)\leq\left(\frac{\tau}{a_0}\right)^{\frac{1}{p_2}}\) for small \(\tau\), and \(p_2<n\).

Finally, if \(a_j,b_j\in C^{0,1}_{\mathrm{loc}}(\overline\Omega)\), then, since
\(a_j,b_j\geq a_0\),
\[
|\nabla_xG(x,t)|\leq\frac{L}{a_0}\sum_{j=1}^{2}
\left[a_j(x)t^{p_j}+b_j(x)t^{q_j}\ln(e+t)\right]=\frac{L}{a_0}\,G(x,t),
\qquad L:=\max_j\max\{\|\nabla a_j\|_\infty,\|\nabla b_j\|_\infty\},
\]
so that \eqref{L:crit} holds with \(\varepsilon=0\) and \eqref{GG5} follows
from Lemma~\ref{lem:G5crit}.
\end{example}

Following the same structure, further examples of generalized
\textnormal{N}-functions can be constructed, provided that
the two-pair representation \eqref{GG3}, including all the conditions
(i)--(vi), and the remaining
structural assumptions \eqref{D22}--\eqref{GG5} are verified.

\subsection{Structure of the paper}
The paper is organized as follows. In Section~2, we introduce notations and recall the fundamental definitions and properties of Musielak--Orlicz and Musielak--Orlicz--Sobolev spaces. 
Section~3 presents a new De Giorgi-type class and establishes key regularity results for this class, which provide the theoretical foundation for our main results, Theorems~\ref{Theorem 2.1} and~\ref{Theorem 2.2}. 
Building on this framework, Section~4 proves the boundedness of weak solutions to problem~\eqref{P} in both subcritical and critical cases (Theorems~\ref{Theorem 4.1w} and~\ref{Theorem 6.1w}). 
Furthermore, using the regularity results for the De Giorgi-type class, we establish the $C^{0,\alpha}$ regularity of weak solutions (Theorems~\ref{Theorem 4.2}--\ref{Theorem 4.4}).

\section{Preliminaries}
Here, we introduce the fundamental notations and definitions used throughout the paper. A central concept is the Musielak-Orlicz-Sobolev space, whose properties are essential for the results that follow.
\subsection{Notations}
The following notations will be used throughout this paper:
\begin{itemize}
    \item $\mathbb{R}^n$ denotes the $n$-dimensional Euclidean space.
    \item $\Omega \subset \mathbb{R}^n$ is a bounded Lipschitz domain.
    \item For a measurable set $E \subset \mathbb{R}^n$, $\mathrm{mes}(E)$ or $|E|$ denotes its $n$-dimensional Lebesgue measure.
    \item If $u$ is a measurable function defined on $\Omega$, and $E \subset \Omega$ is measurable, we define:
   \begin{equation}\label{os1}
    \max_E u := \mathop{\mathrm{ess\,sup}}_{x \in E} u(x), \quad
    \min_E u := \mathop{\mathrm{ess\,inf}}_{x \in E} u(x), \quad
    \mathrm{osc}_E u := \max_E u - \min_E u.
    \end{equation}
    \item For any $x_0 \in \mathbb{R}^n$ and $\rho > 0$, we define:
    \[
    B_\rho(x_0) := \{x \in \mathbb{R}^n : |x - x_0| < \rho\},
    \]
    \[\Omega_\rho(x_0) := B_\rho(x_0) \cap \Omega, \quad \text{and}\quad(\partial \Omega)_\rho(x_0) := \partial \Omega \cap B_\rho(x_0),\]
    which we may denote simply by $B_\rho$, $\Omega_\rho$ and $(\partial \Omega)_\rho$ when the center is clear from context. Balls $B_{\rho_1}$ and $B_{\rho_2}$ are assumed to be concentric. We denote by $\omega_n := |B_1|$ the measure of the unit ball in $\mathbb{R}^n$.
    \item For $0 < \alpha \leq 1$, we denote:
    \[
    C^{0,\alpha}(\overline{\Omega}) := \left\{ u : u \text{ is Hölder continuous on } \overline{\Omega} \text{ with exponent } \alpha \right\},
    \]
    and
    \[
    C^{0,\alpha}(\Omega) := \left\{ u : u \in C^{0,\alpha}(\overline{\Omega'}) \text{ for all } \Omega' \Subset \Omega \right\}.
    \]
    \item For any function \( f : \Omega \rightarrow \mathbb{R} \), we define the positive part and the negative part of \( f \) as:
\[
f^+ = \max\{f, 0\}, \quad f^- = \max\{-f, 0\}.
\]
\end{itemize}
We also adopt the following notation
\begin{equation}\label{eqbh11}
L:=e^{L_{0}},\qquad d_{*}:=\frac{d}{4},\qquad d_{\ell}:=d_{*}\min\{1,\ell(1)\},
\end{equation}
\begin{equation}\label{eqbh12}
H_{*}:=\max_{i}\max\left\{[a_{i}^{\prime}]_{0,\bar\alpha},[b_{i}^{\prime}]_{0,\bar\alpha}\right\},\quad
[b]_{*}:=\max_{i}[b_{i}^{\prime}]_{0,\bar\alpha},\quad
\mu_{*}:=\max\left\{q^{+}+1,p^{+}+\bar\alpha\right\},
\end{equation}
\begin{equation}\label{eqbh13}
R_{*}:=\min\left\{R_{\mathrm{G}},\frac{1}{4M},\frac{1}{2}\left(\frac{d_{*}}{H_{*}}\right)^{\frac{1}{\bar\alpha}}\right\},
\end{equation}
where entries with $H_{*}=0$ are read as $1$, and if only alternative \eqref{GG6b} is used, $R_{*}$ is further replaced by $\min\left\{R_{*},\left(\frac{\beta_{0}}{4H_{*}}\right)^{\frac{1}{\bar\alpha}}\right\}$. Since $d\leq a_{i}^{\prime}+b_{i}^{\prime}\leq 2\bar{a}$, we have $d_{*}\leq \frac{\bar{a}}{2}$. We will use the elementary inequality
\begin{equation}\label{elem:log}
\ell(\lambda t)\leq \lambda\ell(t),\qquad \lambda\geq 1,\ t\geq 0,
\end{equation}
which follows from $\ln(\bar\delta+\lambda s)\leq \lambda\ln(\bar\delta+s)$ with $s=\eta t$ (indeed, by Bernoulli's inequality, $(\bar\delta+s)^\lambda\ge\bar\delta^\lambda(1+\lambda s/\bar\delta)\ge\bar\delta+\lambda s$ since $\bar\delta\ge1$, $\lambda\ge1$).
\subsection{Musielak-Orlicz-Sobolev space \texorpdfstring{$W^{1,G(x,t)}$}{W1G(x,t)}}\label{sec2}
In this subsection, we recall some definitions and fundamental properties of Musielak-Orlicz and Musielak-Orlicz-Sobolev spaces. 
For a comprehensive bibliography on Musielak-Orlicz-Sobolev spaces, we refer the reader to \cite{Ala2, TD, Chlebicka2021, Diening2011, Harjulehto2019, Fan2012a, Fan2012b, Musielak1983, TS1, TS11, TS12, TS2, TS3, TS4, TS5, TS6, TS7, TS8, TS9, TS10}.

\begin{definition}
Let $\Omega$ be an open subset of $\mathbb{R}^n$. A function $G:\Omega \times \mathbb{R} \to \mathbb{R}$ is called a generalized \textnormal{N}-function if it satisfies the following conditions:
\begin{enumerate}
    \item[(1)] For a.e. $x\in \Omega$, the function $G(x,t)$ is even, continuous, strictly increasing, and convex in $t$, and for each $t\in \mathbb{R}$, $G(x,t)$ is measurable in $x$;
    \item[(2)] $\displaystyle\lim_{t \to 0} \frac{G(x,t)}{t} = 0$, for a.e. $x\in \Omega$;
    \item[(3)] $\displaystyle\lim_{t \to \infty} \frac{G(x,t)}{t} = \infty$, for a.e. $x\in \Omega$;
    \item[(4)] $G(x,t) > 0$ for all $t > 0$ and all $x\in \Omega$, and $G(x,0) = 0$ for all $x\in \Omega$.
\end{enumerate}
\end{definition}

\begin{remark}[Remark~2.1, \cite{Ala1}]
We give an equivalent definition of a generalized \textnormal{N}-function that admits an integral representation. For $x \in \Omega$ and $t \geq 0$, let $g(x,t)$ be the right-hand derivative of $G(x,\cdot)$ at $t$, and define $g(x,t) = -g(x,-t)$ for $t < 0$. Then for each $x \in \Omega$, the function $g(x,\cdot)$ is odd, real-valued, satisfies $g(x,0) = 0$, and $g(x,t) > 0$ for $t > 0$, is right-continuous and nondecreasing on $[0, +\infty)$, and
\[
G(x,t) = \int_0^{|t|} g(x,s)\, \mathrm{d}s, \quad \text{for all } x\in \Omega, \ t\in \mathbb{R}.
\]
\end{remark}
\begin{definition}\label{definv}
Let $G:\Omega\times[0,\infty)\to[0,\infty)$ be a generalized $\mathrm{N}$-function.
Since for a.e.\ $x\in\Omega$ the mapping $t\mapsto G(x,t)$ is continuous and
strictly increasing on $[0,\infty)$, we can define $G^{-1}(x,\cdot)$ as the
continuous inverse of $G(x,\cdot)$. In particular, it holds that
\[
G^{-1}(x,G(x,t)) = G(x,G^{-1}(x,t)) = t,
\quad \text{for a.e.\ } x\in\Omega \text{ and all } t\ge 0.
\]
\end{definition}

\begin{definition}
A generalized \textnormal{N}-function $G$ satisfies the \emph{$\Delta_2$-condition} if there exist $C_0 > 0$ and a nonnegative function $\varphi \in L^1(\Omega)$ such that
\[
G(x,2t) \leq C_0 G(x,t) + \varphi(x), \quad \text{for a.e. } x\in \Omega \text{ and all } t \geq 0.
\]
\end{definition}
\begin{definition}
A generalized \textnormal{N}-function $G(x,t)$ is said to satisfy:
\begin{enumerate}
    \item[\textnormal{(A0)}] if there exists $\mu \in (0,1]$ such that
    \[
    \mu \leq G^{-1}(x,1) \leq \frac{1}{\mu}, \quad \text{for a.e. } x \in \mathbb{R}^n;
    \]
    
    \item[\textnormal{(A1)}] if there exists $\mu \in (0,1]$ such that
    \[
    \mu G^{-1}(x,t) \leq G^{-1}(y,t),
    \]
    for every $t \in \left[1, \frac{1}{|B|}\right]$, for a.e. $x,y \in B$, and every ball $B \subset \mathbb{R}^n$ with $|B| \leq 1$;
    
    \item[\textnormal{(A2)}] if for every $s > 0$ there exist $\mu \in (0,1]$ and $\phi \in L^1(\mathbb{R}^n) \cap L^\infty(\mathbb{R}^n)$ such that
    \[
    \mu G^{-1}(x,t) \leq G^{-1}(y,t),
    \]
    for a.e. $x,y \in \mathbb{R}^n$ and for all $t \in [\phi(x)+\phi(y), s]$.
\end{enumerate}
\end{definition}

\begin{definition}\label{prec}
Let $G_1(x,t)$ and $G_2(x,t)$ be two generalized \textnormal{N}-functions.
\begin{itemize}
    \item[(1)] We say that $G_1(x,t)$ increases essentially slower than $G_2(x,t)$ near infinity, and we write $G_1 \prec\prec G_2$, if for any $k > 0$
    \[
    \lim_{t \to \infty} \frac{G_1(x, kt)}{G_2(x,t)} = 0, \quad \text{uniformly in } x \in \Omega.
    \]
    \item[(2)] We say that $G_1(x,t)$ is $\displaystyle{\mathop{\emph{weaker}}_{\Omega}}$ than $G_2(x,t)$, denoted by $\displaystyle{G_1 \mathop{\prec}_{\Omega} G_2}$, if there exist constants $C_1, C_2 > 0$ and a nonnegative function $\varphi \in L^1(\Omega)$ such that
    \[
    G_1(x,t) \leq C_1 G_2(x, C_2 t) + \varphi(x), \quad \text{for a.e. } x \in \Omega \text{ and all } t \geq 0.
    \]
\end{itemize}
\end{definition}

\begin{definition}[Complementary function]\label{CF}
Let $G:\Omega\times[0,\infty)\to[0,\infty)$ be a generalized $\mathrm{N}$-function.
The function $\widetilde{G}:\Omega\times[0,\infty)\to[0,\infty)$ defined by
\begin{equation}\label{Cf}
\widetilde{G}(x,t)
:= \sup_{s \ge 0} \bigl\{ ts - G(x,s) \bigr\},
\quad \text{for a.e.\ } x\in\Omega \text{ and all } t\ge 0,
\end{equation}
is called the \emph{complementary function} (or \emph{conjugate function}) of $G$.
Moreover, $\widetilde{G}$ is also a generalized $\mathrm{N}$-function.
\end{definition}

\begin{remark}
From the definition of the complementary function $\widetilde{G}(x,t)$, we derive the following Young-type inequality:
\begin{equation}\label{Yi}
st \leq G(x,s) + \widetilde{G}(x,t), \quad \text{for all } x \in \Omega, \ s,t \geq 0.
\end{equation}
\end{remark}

The next lemma is taken from Bahrouni--Bahrouni--Missaoui \cite[Lemma~2.3]{Ala1}.

\begin{lemma}\label{lm1}
Let $G(x,t)$ be a generalized \textnormal{N}-function that satisfies \eqref{D22}. Suppose that $t \mapsto g(x,t)$ is continuous and increasing on $\mathbb{R}$ for a.e. $x \in \Omega$. Then
\begin{align}
\widetilde{G}(x, g(x,s)) &\leq (g^+ - 1) G(x,s), \quad \text{for all } s \geq 0, \ x \in \Omega, \label{L1} \end{align}
and
\begin{align}
\frac{g^+}{g^+ - 1} =: \widetilde{g^-} &\leq \frac{\widetilde{g}(x,s)s}{\widetilde{G}(x,s)} \leq \widetilde{g^+} := \frac{g^-}{g^- - 1}, \quad \text{for all } x \in \Omega, \ s > 0, \label{D3}
\end{align}
where $\displaystyle \widetilde{G}(x,s) = \int_0^s \widetilde{g}(x,t)\, \mathrm{d}t$.
\end{lemma}

\begin{remark}\label{compl}
The condition \eqref{D22} implies that $G(x,t)$ and its complementary function $\widetilde{G}(x,t)$ satisfy the $\Delta_2$-condition.
\end{remark}
 Now, we  define the Musielak-Orlicz space as follows:
$$L^{G(x,t)}(\Omega):=\left\lbrace u:\Omega\longrightarrow \mathbb{R}\ \text{measurable :}\ \rho_{G}(\lambda u)<+\infty,\ \ \text{for some}\ \lambda>0\right\rbrace,$$
where
\begin{equation}\label{Mo}
  \rho_{G}(u):= \int_{\Omega}G(x, u){\rm d}x.
\end{equation}
The space $L^{G(x,t)}(\Omega)$ is endowed with the Luxemburg norm
\begin{equation}\label{No}
  \Vert u\Vert_{L^{G(x,t)}(\Omega)}:=\inf\left\lbrace \lambda>0:\ \rho_{G}\left(\frac{u}{\lambda}\right)\leq 1\right\rbrace.
\end{equation}
\begin{proposition}\label{HM}
    Let $G$ be a generalized \textnormal{N}-function satisfies the $\Delta_2$-condition, then
    $$L^{G(x,t)}(\Omega) =\left\lbrace  u: \Omega \longrightarrow \mathbb{R}\ \text{measurable :}\ \rho_{G}( u)<+\infty\right\rbrace.$$
\end{proposition}
\begin{proposition} \label{zoo}
     Let $G$ be a generalized \textnormal{N}-function satisfies \eqref{D22}, then the following assertions hold:
     \begin{itemize}
    \item [(1)] $\min \{s^{g^-}, s^{g^+}\}G(x,t)\leq  G(x,s t)\leq \max \{ s^{g^-}, s^{g^+}\}G(x,t),\text{ for a.e. } x \in \Omega$ $\text{ and all }  s, \ t \geq 0. $
     \item [(2)] $\min \{s^{\widetilde{g^-}}, s^{\widetilde{g^+}}\}\widetilde{G}(x,t)\leq  \widetilde{G}(x,s t)\leq \max \{ s^{\widetilde{g^-}}, s^{\widetilde{g^+}}\}\widetilde{G}(x,t),\text{ for a.e. } x \in \Omega$ $\text{ and all }  s, \ t \geq 0. $
 \item [(3)]  $\min \left\{\|u\|_{L^{G(x,t)}(\Omega)}^{g^-},\|u\|_{L^{G(x,t)}(\Omega)}^{g^+}\right\} \leq \rho_{G}(u) \leq\max \left\{\|u\|_{L^{G(x,t)}(\Omega)}^{g^-},\|u\|_{L^{G(x,t)}(\Omega)}^{g^+}\right\},$ \ for all $  u\in L^{G(x,t)}(\Omega)$.
\end{itemize}
\end{proposition}
As a consequence of \eqref{Yi}, we have the following result:
\begin{lemma}[H\"older's type
inequality]\label{H1}
  Let $\Omega$ be an open subset of $\mathbb{R}^n$ and $G$ be a generalized \textnormal{N}-function satisfies \eqref{D22}, then
  \begin{equation}\label{Ho}
     \left\vert \int_{\Omega} uv {\rm d}x \right\vert \leq 2 \Vert u\Vert_{L^{G(x,t)}(\Omega)}\Vert v\Vert_{L^{\widetilde{G}(x,t)}(\Omega)},\ \text{for all}\ u\in L^{G(x,t)}(\Omega)\ \text{and all}\ v\in L^{\widetilde{G}(x,t)}(\Omega).
  \end{equation}
\end{lemma}
The subsequent proposition deals with some topological properties of
the Musielak-Orlicz space, see \cite[Theorem 7.7 and Theorem
8.5]{Musielak1983}.
\begin{proposition}\label{AB}\ \\
\begin{enumerate}
\item [(1)] Let $G(x,t)$ be a generalized \textnormal{N}-function and $\Omega$ an open subset of
 $\mathbb{R}^n$. Then,
 \begin{enumerate}
     \item[(a)] the space
    $\left(L^{G(x,t)}(\Omega),\|\cdot\|_{G}\right)$ is a Banach space;
     \item[(b)] if $G(x,t)$ satisfies \eqref{D22}, then
$L^{G(x,t)}(\Omega)$ is a separable and reflexive space.
 \end{enumerate}
\item [(2)] Let $\Omega$ be an open bounded subset of $\mathbb{R}^n$  and $G_1(x,t)$ and $G_2(x,t)$  be two generalized \textnormal{N}-functions such that $\displaystyle{G_1(x,t) \mathop{\prec}_{\Omega} G_2(x,t)}$. Then,
$$
L^{G_2(x,t)}(\Omega) \hookrightarrow L^{G_1(x,t)}(\Omega).
$$
\end{enumerate}
\end{proposition}

Now, we are ready to define the  Musielak-Orlicz Sobolev space. Let
$G(x,t)$ be a generalized \textnormal{N}-function and $\Omega$ be an open subset of
$\mathbb{R}^n$. The  Musielak-Sobolev space is defined as follows
$$W^{1,G(x,t)}(\Omega):=\left\lbrace u\in L^{G(x,t)}(\Omega):\ |D u| \in L^{G(x,t)}(\Omega)\right\rbrace.$$
The space $W^{1,G(x,t)}(\Omega)$ is endowed with the norm
\begin{equation}\label{NM}
  \Vert u\Vert:=\Vert u\Vert_{L^{G(x,t)}(\Omega)}+\Vert D u\Vert_{L^{G(x,t)}(\Omega)},\ \ \text{for all}\ u\in W^{1,G(x,t)}(\Omega),
\end{equation}
where $\|D u\|_{L^{G(x,t)}(\Omega)} := \| |D u| \|_{L^{G(x,t)}(\Omega)}$.\\ We denote by
$W^{1,G(x,t)}_0(\Omega)$
 the completion of $C^\infty _0(\Omega)$ in $W^{1,G(x,t)}(\Omega)$.
\begin{remark}\label{Ref}
 If $G$  satisfies \eqref{D22}, then the space
$W^{1,G(x,t)}(\Omega)$ is a reflexive and separable Banach space with
respect to the norm $\Vert \cdot\Vert$.
\end{remark}
 As mentioned in the introduction, two distinct notions of Sobolev conjugate are
needed in our analysis: the Sobolev conjugate of Cianchi--Diening and Fan's
Sobolev conjugate. The reason is that the boundary-trace embedding
$W^{1,G(x,t)}(\Omega)\hookrightarrow L^{\vartheta(x,t)}(\partial\Omega)$ is,
to the best of our knowledge, available only for Fan's Sobolev conjugate
$G_*$ (see Proposition~\ref{ffembd} below, taken from \cite{LWZ}), whereas no
analogous trace result is currently known for the Cianchi--Diening conjugate
$G^*$. On the other hand, the Cianchi--Diening conjugate $G^*$ yields sharper
interior Sobolev embeddings under weaker structural assumptions
(Proposition~\ref{embb}), and is therefore the natural tool for the interior
estimates in the Dirichlet problem. For this reason, we develop both
constructions in parallel: the Cianchi--Diening conjugate is used for the
interior Sobolev embedding, and Fan's conjugate is used for the trace embedding
on $\partial\Omega$ required by the Neumann problem.

We begin by introducing the
Cianchi--Diening Sobolev conjugate from \cite{Ala3}, which refines the one proposed in
\cite{Cianchi2024} and yields several important results presented below; Fan's
conjugate is then introduced in Definition~\ref{DefF}.

\begin{definition}\label{DefCL}
The Sobolev conjugate of $G(x,t)$ is defined as the generalized \textnormal{N}-function $G^\ast(x,t)$ given by
			\begin{equation*}
  G ^ \ast (x,t)=G (x, N^{-1} (x,t)) , \text{ for a.e. } x \in \mathbb{R}^n \text{ and all } t \geq 0,
\end{equation*}
where
\begin{equation*}
  N(x,t)= \left( \int_{0}^{t} \left( \frac{\tau}{G (x, \tau)} \right)^{\frac{1}{n-1}} d \tau \right)^{\frac{n-1}{n}}, \text{ for } x \in \mathbb{R}^n \text{ and } t\geq 0.
\end{equation*}
 \end{definition}
 \begin{proposition}[\cite{Ala3}] \label{zoo*}
     Let $G(x,t)$ be a generalized \textnormal{N}-function that satisfies \eqref{D22}--\eqref{GG2} and $G^*(x,t)$ its Sobolev conjugate function, then the following assertions hold:
     \begin{itemize}
    \item [(1)] $\min \left\{s^{g^-_*}, s^{g^+_*}\right\}G^*(x,t)\leq  G^*(x,s t)\leq \max \left\{ s^{g^-_*}, s^{g^+_*}\right\}G^*(x,t),\text{ for a.e. } x \in \Omega$ $\text{ and all }  s, \ t \geq 0$;
 \item [(2)] $\displaystyle{
        1<g^-_*\leq \frac{g^*(x,t)t}{G^*(x,t)}\leq g^+_*,\ \ \text{for all}\ x\in \Omega\ \text{and all}\ t> 0}$;
    \item[(3)] $\min \left\{\|u\|_{L^{G^*(x,t)}(\Omega)}^{g^-_*},\|u\|_{L^{G^*(x,t)}(\Omega)}^{g^+_*}\right\} \leq \rho_{G^*}(u) \leq\max \left\{\|u\|_{L^{G^*(x,t)}(\Omega)}^{g^-_*},\|u\|_{L^{G^*(x,t)}(\Omega)}^{g^+_*}\right\},$ \ for all $  u\in L^{G^*(x,t)}(\Omega)$;
\end{itemize}
where $\displaystyle{g_*^- := \frac{ng^-}{n - g^-}}$, $\displaystyle{g_*^+ := \frac{ng^+}{n - g^+}}$ and $\displaystyle{G^*(x,t)=\int_0^t g^*(x,s){\rm d}s}$.
\end{proposition}
\begin{proposition}[\cite{Cianchi2024}]\label{embb}
Let $\Omega$ be a bounded domain in $\mathbb{R}^n$ and $G(x,t)$ be a generalized \textnormal{N}-function that satisfies \eqref{GG1}--\eqref{GG2}. Let $\vartheta(x,t)$ be a generalized \textnormal{N}-function such that $\displaystyle{\vartheta(x,t)\mathop{\prec}_{\Omega}G^*(x,t)}$
and
\begin{equation}
 \int_{\Omega} \vartheta(x,t) \, \mathrm{d}x < \infty,\ \text{for } t>0.
\end{equation}
\begin{enumerate}
\item[(1)] The embedding
\begin{equation}
W^{1,G(x,t)}_0(\Omega) \hookrightarrow L^{\vartheta(x,t)}(\Omega) 
\end{equation}
is continuous.
\item[(2)] Assume, in addition, that $\Omega$ is a bounded domain with a Lipschitz boundary $\partial \Omega$. Then, the embedding
\begin{equation}
W^{1,G(x,t)}(\Omega) \hookrightarrow L^{\vartheta(x,t)}(\Omega) 
\end{equation}
is continuous.
\end{enumerate}
\end{proposition}
\begin{remark}
If the function $G(x,t)$ satisfies assumptions \eqref{D22}--\eqref{GG2}, then it also fulfills conditions (A0)--(A1) as stated in L. Diening and A. Cianchi paper~\cite{Cianchi2024}, and it satisfies condition (A2) due to the boundedness of $\Omega \subset \mathbb{R}^n$. Consequently, the embedding result given in Proposition~\ref{embb} remains valid under assumptions \eqref{D22}--\eqref{GG2}.
\end{remark}
We now introduce the second conjugate, namely the Sobolev conjugate defined by
Fan \cite{Fan2012a}. Unlike the Cianchi--Diening conjugate $G^*$, which is
defined through the auxiliary function $N(x,t)$ and is used for the interior
Sobolev embedding, Fan's conjugate $G_*$ is defined by integrating $G^{-1}$
against the singular weight $t^{-\frac{n+1}{n}}$, and it is precisely this
construction that is compatible with the boundary-trace embedding on
$\partial\Omega$ (see Proposition~\ref{ffembd}).
\begin{definition}[\cite{Fan2012a}]\label{DefF}
Let $\Omega$ be an open bounded subset of $\mathbb{R}^n$ and $G(x,t)$ be a generalized \textnormal{N}-function that satisfies \eqref{D22}--\eqref{GG1} and \eqref{GG4}. Define a function $G_*^{-1} : \overline{\Omega} \times [0,+\infty) \to [0,+\infty)$ by
\[
G_*^{-1}(x,s) = \int_0^s \frac{G^{-1}(x,t)}{t^{\frac{n+1}{n}}}\,\mathrm{d}t,\quad \text{for } x \in \overline{\Omega} \text{ and } s \in [0,+\infty).
\]
Then under the assumption \eqref{GG4}, $G_*^{-1}$ is well defined, and for each $x \in \overline{\Omega}$, $G_*^{-1}(x,\cdot)$ is strictly increasing, $G_*^{-1}(x,\cdot) \in C^1((0,+\infty))$ and the function $G_*^{-1}(x,\cdot)$ is concave.\\
Set
\begin{equation}\label{Tx}
    T(x)=\lim_{s\to +\infty}G_*^{-1}(x,s),\quad\forall x\in \overline{\Omega}.
\end{equation}
Then $0<T(x)\leq +\infty$. We define  Fan's Sobolev conjugate function $G_* : \overline{\Omega} \times [0,+\infty) \to [0,+\infty)$ by
\[
G_*(x,t) =
\begin{cases}
s, & \text{for } x \in \overline{\Omega},\ t \in [0,T(x)) \text{ and } G_*^{-1}(x,s) = t, \\
+\infty, & \text{for } x \in \overline{\Omega} \text{ and } t\geq T(x).
\end{cases}
\]
The function $G_*(x,t)\in C^1((0,T(x)))$.  In particular, if $T(x)=+\infty,\quad\forall x\in \overline{\Omega}$,  $G_*(x,t)$ is a generalized \textnormal{N}-function.
\end{definition}
\begin{proposition}[\cite{Fan2012a}; \cite{LWZ}, Theorem 2.4]\label{fembb}
Let $\Omega \subset \mathbb{R}^n$ be a bounded Lipschitz domain, and let $G(x,t)$ be a generalized \textnormal{N}-function satisfying \eqref{D22}--\eqref{GG1} and \eqref{GG4}--\eqref{GG5}. If $\vartheta(x,t)$ is a generalized \textnormal{N}-function such that
\[
\vartheta(x,t)\mathop{\prec}_{\Omega} G_*(x,t),
\]
then the embedding
\[
W^{1,G(x,t)}(\Omega) \hookrightarrow L^{\vartheta(x,t)}(\Omega)
\]
is continuous.
\end{proposition}
\begin{proposition}[\cite{LWZ}, Theorem 4.2]\label{ffembd}
Let $\Omega \subset \mathbb{R}^n$ be a bounded Lipschitz domain, and let $G(x,t)$ be a generalized \textnormal{N}-function satisfying \eqref{D22}--\eqref{GG1} and \eqref{GG4}--\eqref{GG5}. Suppose that
\(
\left[G_*(x,t)\right]^{\frac{n-1}{n}}
\)
satisfies the $\Delta_2$-condition. Then, for any generalized \textnormal{N}-function $\vartheta(x,t)$ such that
\[
\vartheta(x,t)\mathop{\prec}_{\partial\Omega}
\left[G_*(x,t)\right]^{\frac{n-1}{n}},
\]
the boundary trace embedding
\[
W^{1,G(x,t)}(\Omega)
\hookrightarrow
L^{\vartheta(x,t)}(\partial\Omega)
\]
is continuous.
\end{proposition}
 Proceeding as in \cite[Lemma 2.2]{Fuk2004}, we obtain the following result:
 \begin{proposition} \label{fzoo*}
     Let $G(x,t)$ be a generalized \textnormal{N}-function that satisfies \eqref{D22}--\eqref{GG2} and \eqref{GG4}; and $G_*(x,t)$ be its Fan's Sobolev conjugate function; then the following assertions hold:
     \begin{itemize}
    \item [(1)] $\min \left\{s^{g^-_*}, s^{g^+_*}\right\}G_*(x,t)\leq  G_*(x,s t)\leq \max \left\{ s^{g^-_*}, s^{g^+_*}\right\}G_*(x,t),\text{ for a.e. } x \in \Omega$ $\text{ and all }  s, \ t \geq 0$;
 \item [(2)] $\displaystyle{
        1<g^-_*\leq \frac{g_*(x,t)t}{G_*(x,t)}\leq g^+_*,\ \ \text{for all}\ x\in \Omega\ \text{and all}\ t> 0}$;
    \item[(3)] $\min \left\{\|u\|_{L^{G_*(x,t)}(\Omega)}^{g^-_*},\|u\|_{L^{G_*(x,t)}(\Omega)}^{g^+_*}\right\} \leq \rho_{G_*}(u) \leq\max \left\{\|u\|_{L^{G_*(x,t)}(\Omega)}^{g^-_*},\|u\|_{L^{G_*(x,t)}(\Omega)}^{g^+_*}\right\},$ \ for all $  u\in L^{G_*(x,t)}(\Omega)$;
\end{itemize}
where $\displaystyle{g_*^- = \frac{ng^-}{n - g^-}}$, $\displaystyle{g_*^+ = \frac{ng^+}{n - g^+}}$ and $\displaystyle{G_*(x,t)=\int_0^t g_*(x,s){\rm d}s}$.
\end{proposition}

We conclude this section with a criterion for \eqref{GG5} that can be checked directly on \(G\) (recall that \eqref{GG5} is a condition on Fan's Sobolev conjugate \(G_*\), which is built from \(G\) on all of \((0,\infty)\)); it is used in the examples of Subsection~\ref{Examples}.

\begin{lemma}[A criterion for \eqref{GG5}]\label{lem:G5crit}
Let \(\Omega\subset\mathbb{R}^n\) be a bounded Lipschitz domain, and let \(G\in C^{0,1}_{\mathrm{loc}}(\overline{\Omega}\times(0,\infty))\) be a generalized \textnormal{N}-function satisfying \eqref{D22} and \eqref{GG1}. Assume that there exist \(K\geq0\) and \(\varepsilon\in[0,\frac1n)\) such that
\begin{equation}\label{L:crit}
|\nabla_xG(x,t)|\leq K\bigl(1+|\ln G(x,t)|\bigr)\bigl(1+G(x,t)\bigr)^{\varepsilon}\,G(x,t)
\tag{S5}
\end{equation}
for all \(x\in\Omega\) and \(t>0\) at which \(\nabla_xG(x,t)\) exists. Then \eqref{GG4} holds (so that Fan's Sobolev conjugate \(G_*\) of \(G\) is well defined), \(T\equiv+\infty\), and \eqref{GG5} holds.
\end{lemma}
\begin{proof}
We derive the desired statement from
\cite[Proposition~3.1(2)]{Fan2012a}. The only point in which \(G\)
differs from the functions considered there is the normalization, which
will be addressed below.

Since \(G(\cdot,1)\) is continuous, \eqref{GG1} holds for every
\(x\in\overline{\Omega}\). Integrating \eqref{D22} with respect to \(t\)
(see Proposition~\ref{zoo}(1) with \(t=1\)), we obtain, for all
\(x\in\overline{\Omega}\) and \(t>0\),
\[
F^{-1}\min\{t^{g^-},t^{g^+}\}
\leq G(x,t)
\leq F\max\{t^{g^-},t^{g^+}\}.
\]
In particular,
\[
G(x,t)\geq F^{-1}t^{g^-}\geq1
\qquad\text{for }t\geq t_1:=F^{\frac{1}{g^-}},
\]
and
\[
G(x,t)\leq Ft^{g^-}
\qquad\text{for }0<t\leq1.
\]
Moreover, for all \(x\in\overline{\Omega}\),
\[
G^{-1}(x,\tau)\geq
\left(\frac{\tau}{F}\right)^{\frac{1}{g^+}}
\qquad (\tau\geq F),
\]
\[
G^{-1}(x,\tau)\leq
F^{\frac{1}{g^-}}\tau^{\frac{1}{g^+}}
\qquad (0<\tau\leq1),
\]
and
\[
G^{-1}(x,\tau)\leq
(F\tau)^{\frac{1}{g^-}}
\qquad (\tau\geq1).
\]
Indeed, if \(\tau\geq F\) and
\(t:=(\tau/F)^{1/g^+}\geq1\), then
\[
G(x,t)\leq Ft^{g^+}=\tau,
\]
so that \(G^{-1}(x,\tau)\geq t\). On the other hand, if
\(\tau>0\) and \(s:=G^{-1}(x,\tau)\), then
\[
\tau=G(x,s)\geq F^{-1}s^{g^+}
\qquad\text{if }s\leq1,
\]
and
\[
\tau=G(x,s)\geq F^{-1}s^{g^-}
\qquad\text{if }s>1.
\]
Since \(F\geq1\), these inequalities yield the asserted upper bounds
for \(G^{-1}\).

By the second estimate for \(G^{-1}\), for every
\(x\in\overline{\Omega}\),
\[
\begin{aligned}
\int_0^1
\frac{G^{-1}(x,\tau)}{\tau^{\frac{n+1}{n}}}\,\mathrm{d}\tau
&\leq F^{\frac{1}{g^-}}
\int_0^1
\tau^{\frac{1}{g^+}-\frac{1}{n}-1}\,\mathrm{d}\tau\\
&<\infty,
\end{aligned}
\]
because \(g^+<n\). Thus \eqref{GG4} holds, and \(G_*\) is well defined
(Definition~\ref{DefF}). Furthermore, the first estimate for
\(G^{-1}\) gives
\[
G^{-1}(x,\tau)\tau^{-\frac{n+1}{n}}
\geq
F^{-\frac{1}{g^+}}
\tau^{\frac{1}{g^+}-\frac{1}{n}-1}
\qquad (\tau\geq F).
\]
Since \(1/g^+-1/n>0\), we conclude that
\[
T(x)=\int_0^\infty
G^{-1}(x,\tau)\tau^{-\frac{n+1}{n}}\,\mathrm{d}\tau
=+\infty
\qquad\text{for every }x\in\overline{\Omega}.
\]

We next verify hypothesis~(3.20) of
\cite[Proposition~3.1(2)]{Fan2012a}. Since \(G(x,t)\geq1\) for
\(t\geq t_1\), and
\[
(1+\ln u)(1+u)^\varepsilon
\leq C u^{\varepsilon+\eta}
\qquad (u\geq1)
\]
for every \(\eta>0\), assumption \eqref{L:crit} implies that, for every
\(\delta_1\in(\varepsilon,1/n)\), there exists \(c_1>0\) such that
\begin{equation}\label{L:Fan}
|\nabla_xG(x,t)|
\leq c_1G(x,t)^{1+\delta_1}
\qquad
\text{for all }x\in\Omega,\ t\geq t_1,
\end{equation}
at every point where \(\nabla_xG(x,t)\) exists. This is precisely
hypothesis~(3.20) of \cite[Proposition~3.1(2)]{Fan2012a}.

The remaining hypotheses of that proposition are also satisfied.
Indeed, (P1) and (P2) of \cite{Fan2012a} are clear, since \(\Omega\)
has the cone property, being a Lipschitz domain, and \(G\) is a
generalized \textnormal{N}-function, continuous on
\(\overline{\Omega}\times[0,\infty)\) and positive for \(t>0\).
Moreover, \(G\) satisfies the \(\Delta_2\)-condition by \eqref{D22},
with constant \(2^{g^+}\).

We also observe that \(G\) is locally Lipschitz continuous on
\(\overline{\Omega}\times[0,\infty)\). Indeed, for \(0<t\leq1\),
\[
G(x,t)\leq Ft^{g^-}
\]
and, by \eqref{D22},
\[
|\partial_tG(x,t)|
=g(x,t)
\leq\frac{g^+G(x,t)}{t}
\leq g^+Ft^{g^--1}.
\]
Furthermore, \(|\nabla_xG(x,t)|\) is bounded for \(0<t\leq1\)
by \eqref{L:crit}, since \(0<G(x,t)\leq F\) for such \(t\), and
\[
(1+|\ln u|)(1+u)^\varepsilon u
\]
is bounded on \((0,F]\). Since \(T\equiv+\infty\), as established
above, hypothesis (P4) of \cite{Fan2012a} holds automatically.
Consequently, the conclusion of
\cite[Proposition~3.1(2)]{Fan2012a}, namely hypothesis (P5) therein,
is precisely \eqref{GG5}.

It remains to address the normalization used in \cite{Fan2012a}.
There, the function is assumed to be linear on \([0,1]\), which is
condition (P3) of \cite{Fan2012a}, whereas \(G\) does not necessarily
satisfy this condition. However, in the proof of
Proposition~3.1(2), this normalization is used only to control the
behaviour near \(t=0\). More precisely, what is needed is that
\[
\int_0^s
\left|\frac{\partial G^{-1}(x,\tau)}{\partial x_j}\right|
\tau^{-\frac{n+1}{n}}\,\mathrm{d}\tau
\]
be finite and bounded uniformly in \(x\), for every \(s>0\).
We now verify this property using \eqref{L:crit}.

Differentiating
\[
G\bigl(x,G^{-1}(x,\tau)\bigr)=\tau
\]
with respect to \(x_j\), at almost every \((x,\tau)\) (since \(G\) and
\(G^{-1}\) are locally Lipschitz continuous), and using
\[
\partial_tG(x,t)\geq
g^-\frac{G(x,t)}{t},
\]
which follows from \eqref{D22}, we obtain from \eqref{L:crit}
\[
\left|
\frac{\partial G^{-1}(x,\tau)}{\partial x_j}
\right|
\leq
\frac{K}{g^-}
(1+|\ln\tau|)(1+\tau)^\varepsilon
G^{-1}(x,\tau).
\]
For \(0<\tau\leq1\), the right-hand side is bounded by
\[
C(1+|\ln\tau|)\tau^{\frac{1}{g^+}},
\]
thanks to the second estimate for \(G^{-1}\). Since
\(1/g^+>1/n\), it follows that
\begin{equation}\label{L:int}
\operatorname*{ess\,sup}_{x\in\Omega}
\int_0^1
\left|
\frac{\partial G^{-1}(x,\tau)}{\partial x_j}
\right|
\tau^{-\frac{n+1}{n}}\,\mathrm{d}\tau
<\infty,
\qquad j=1,\ldots,n.
\end{equation}
On the other hand, for \(1\leq\tau\leq s\), the same estimate, together
with the third upper bound for \(G^{-1}\), shows that the integrand is
bounded by a constant depending only on \(s\). Thus the required
uniform integrability holds for every \(s>0\).

Therefore, the proof of
\cite[Proposition~3.1(2)]{Fan2012a} applies to \(G\) itself and yields
\eqref{GG5}. This is also the approach used to verify \eqref{GG5} in
\cite[Examples~1 and~2]{LWZ}.
\end{proof}

\section{A New Class of De Giorgi-Type Functions \texorpdfstring{$\mathcal{B}_{G(x,t)}$}{$B_{G(x,t)}$}}

In this section, we introduce and analyze a new De Giorgi-type class of functions denoted by \( \mathcal{B}_{G(x,t)} \). Our main results are encapsulated in Theorems~\ref{Theorem 2.1} and~\ref{Theorem 2.2}, which establish key regularity properties for functions belonging to this class.

The proof strategies we adopt are inspired by the foundational works of X.~Fan~\cite{Fan1999} and G.~Lieberman~\cite{Li1991}. However, our approach involves significant technical refinements that are essential to overcome the challenges arising from the nonstandard growth behavior and spatial heterogeneity inherent in the Musielak–Orlicz framework.

We begin by collecting several auxiliary lemmas that will play a crucial role in the subsequent analysis.

\begin{lemma}[see Lemma 4.7 of Chap. II of \cite{OL}]\label{Lemma 2.1}
Let $\{Y_h\}_{h\geq 0}$ be a sequence of non-negative numbers satisfying the recursion relation
\[
Y_{h+1} \leq c b^h Y_h^{1+\varepsilon}, \quad h = 0, 1, 2, \dots,
\]
where $c > 0$, $b > 1$, and $\varepsilon > 0$. If
\[
Y_0 \leq \theta := c^{-\frac{1}{\varepsilon}} b^{-\frac{1}{\varepsilon^2}},
\]
then
\[
Y_h \leq \theta b^{-\frac{h}{\varepsilon}}, \quad \forall h =0,1,2\ldots,
\]
and hence $Y_h \to 0$ as $h \to \infty$.
\end{lemma}

\begin{lemma}[see Lemma 3.5 of Chap. II of \cite{OL}]\label{Lemma 2.2}
For any \(u \in W^{1,1}(B_\rho)\) and arbitrary numbers \(k\) and \(l\) with \(l > k\), the following inequality hol{\rm d}s:
\begin{equation}\label{2.2}
(l - k)|A_{l,\rho}|^{1-\frac{1}{n}} \leq \beta \frac{\rho^n}{|B_\rho \setminus A_{k,\rho}|} \int_{A_{k,\rho}\setminus A_{l,\rho}} |Du| \, {\rm d}x,
\end{equation}
where \(A_{k,\rho} = \{x \in B_\rho : u(x) > k\}\), \(\beta = \beta(n) > 1\) is a constant depending only on \(n\).
\end{lemma}

\begin{lemma}[see Lemma 2.3 of \cite{Fan1999} or Lemma 4.8 of Chap. II of \cite{OL}]\label{Lemma 2.3}
Suppose a function \(u(x)\) is measurable and bounded in some ball \(B_{R_0}\). Consider balls \(B_R\) and \(B_{bR}\) which have a common center with \(B_{R_0}\), where \(b > 1\) is a fixed constant, and suppose that for any \(0 < R \leq b^{-1} R_0\), at least one of the following two inequalities is valid:
\[
\operatorname{osc}_{B_R}u \leq c_1 R^\varepsilon, \quad \text{or}\quad \operatorname{osc}_{B_R}u \leq \theta \operatorname{osc}_{ B_{bR}}u,
\]
where \(c_1, \varepsilon \leq 1\) and \(\theta < 1\) are positive constants. Then \(u \in C^{0,\alpha}(B_{R_0})\), where \(\alpha = \min\{\varepsilon, -\log_b \theta\}\).
\end{lemma}
The following lemmas are fundamental to our subsequent analysis.
\begin{lemma}\label{lem:conv}
Let $G:\Omega\times[0,\infty)\to[0,\infty)$ be a generalized \textnormal{N}-function satisfying \eqref{D22}.
Then, for every $x\in\Omega$, every $t\ge 0$ and every $\lambda>0$,
\begin{equation}\label{eq:conv}
t \;\le\; \lambda \;+\; \frac{\lambda}{g^{-}\,G(x,\lambda)}\,G(x,t).
\end{equation}
\end{lemma}

\begin{proof}
Fix $x\in\Omega$ and $\lambda>0$. Since $G(x,\cdot)$ is convex on $[0,\infty)$ with $G(x,0)=0$,
for all $t\ge 0$,
\begin{equation}\label{eq:subgrad}
G(x,t) \;\ge\; G(x,\lambda) \;+\; g(x,\lambda)\,(t-\lambda),
\end{equation}
where $g(x,\lambda)=\partial_t G(x,\lambda)$. Rearranging,
\[
g(x,\lambda)\,(t-\lambda) \;\le\; G(x,t) - G(x,\lambda) \;\le\; G(x,t).
\]
By \eqref{D22},
\[
g(x,\lambda) \;\ge\; \frac{g^{-}\,G(x,\lambda)}{\lambda}.
\]
Combining the last two inequalities and dividing by $\dfrac{g^{-}\,G(x,\lambda)}{\lambda}>0$, we obtain (for $t\ge\lambda$; for $t<\lambda$ the claim is trivial)
\[
t-\lambda \;\le\; \frac{\lambda}{g^{-}\,G(x,\lambda)}\,G(x,t),
\]
which is equivalent to \eqref{eq:conv}.
\end{proof}
\begin{lemma}\label{lem22}
Under assumptions \eqref{D22} and \eqref{GG3}, let \(x_0 \in \Omega\), \(R \in (0, R_*)\), and \(B := B_R(x_0)\). Let
\(L, d_*, d_\ell, H_*, [b]_*, \mu_*, R_*\) be as in \eqref{eqbh11}--\eqref{eqbh13}. Then there exist a nondecreasing function \(\Phi:[0,\infty)\to[0,\infty)\) and constants \(f\in [p^-,q^+]\), \(\kappa = \kappa_*\) (depending only on \(n, p^\pm, q^\pm, d, \bar{a}, \Theta, \ell(1), M, L, \bar\alpha\) and on $C_*$ if \eqref{GG6a} is used), such that, with \(c_{\mathrm{sm}}:=2\bar{a}(1+\ell(1))\):
\begin{enumerate}
\item[(1)] \(\Phi(t) \leq G(x,t)\) for all \(x \in B\), \(t \geq 1\), and
\(\Phi(\sigma) \leq G(x,\sigma) + c_{\mathrm{sm}}\) for all \(x \in B\), \(\sigma \geq 0\);
\item[(2)] \(t^{\,f} \leq \lambda^{\,f}\left[1 + \frac{\Phi(t)}{\Phi(\lambda)}\right]\) for all \(t, \lambda > 0\);
\item[(3)] \(\Phi(t) \geq d_{\ell}\, t^{\,f} \geq d_{\ell}\) for \(t \geq 1\);
\item[(4)] \(G(x,\lambda t) \leq \kappa\, \lambda^{\mu_*}\, \Phi(t)\) for all \(x \in B\), \(\lambda \geq 1\), \(t \in \left[1, \frac{2M}{R}\right]\).
\end{enumerate}
\end{lemma}

\begin{proof}
Fix $x_0\in\Omega$, $R\in(0,R_*)$, and $B:=B_R(x_0)$.

\medskip

We begin by proving a local lower bound for a single pair; the two-pair case will then follow by summing the two estimates.

Fix one index $i\in\{1,2\}$ and drop it from the notation. Put
\[
\underline{p}:=\inf_B p',\quad \overline{p}:=\sup_B p,\quad
\underline{q}:=\inf_B q',\quad \overline{q}:=\sup_B q,
\]
and
\[
\Lambda(x,t):=a'(x)\,t^{p'(x)}+b'(x)\,t^{q'(x)}\ell(t),
\quad
U(x,t):=a(x)\,t^{p(x)}+b(x)\,t^{q(x)}\ell(t).
\]
By \eqref{GG3}(i), for all $x\in B$ and $t\ge 1$,
\begin{equation}\label{eq:LambdaU}
\Lambda(x,t)\;\le\;G(x,t)\;\le\;U(x,t)
\end{equation}
(here, when summing over the two indices, $\Lambda$ and $U$ denote the $i$-th summand; the inequality \eqref{eq:LambdaU} is used only after summation, see below).
Moreover, by \eqref{GG3}(ii),
\begin{equation}\label{eq:exponents}
p^-\le \underline{p}\le \overline{p}\le q^+,\qquad
p^-\le \underline{q}\le \overline{q}\le q^+.
\end{equation}

We distinguish three cases, according to whether the lower coefficient
\(a'(x_0)\) is bounded away from zero and, if so, whether the second
coefficient \(b'\) is bounded away from zero on \(B\).

\medskip

\noindent\emph{Case A1:} $a'(x_0)\ge \frac{d}{2}$ and $b_0:=\inf_B b'>4\,[b']_{0,\bar\alpha}R^{\bar\alpha}$.

Since $a'\in C^{0,\bar\alpha}(B)$ and $R<R_*\le \tfrac12\left(\frac{d_*}{H_*}\right)^{\frac{1}{\bar\alpha}}$, we have for all $x\in B$
\[
a'(x)\;\ge\;a'(x_0)-[a']_{0,\bar\alpha}(2R)^{\bar\alpha}
\;\ge\;\frac{d}{2}-\frac{d}{4}\;=\;d_*.
\]
Define
\[
\Phi(t):=d_*\,t^{\underline{p}}+b_0\,t^{\underline{q}}\,\ell(t).
\]
Then for all $x\in B$ and $t\ge 1$, using the definition of $\underline{p}$, $\underline{q}$, and $a'\ge d_*$, $b'\ge b_0$:
\begin{equation}\label{eq:LA1}
\Lambda(x,t)\;\ge\;d_*\,t^{\underline{p}}+b_0\,t^{\underline{q}}\,\ell(t)\;=\;\Phi(t).
\end{equation}

\medskip

\noindent\emph{Case A2:} $a'(x_0)\ge \frac{d}{2}$ and $b_0\le 4\,[b']_{0,\bar\alpha}R^{\bar\alpha}$.
(This case can occur only under alternative \eqref{GG6a}.)

As in Case A1, $a'\ge d_*$ on $B$. Define
\[
\Phi(t):=d_*\,t^{\underline{p}}.
\]
Then for all $x\in B$ and $t\ge 1$,
\begin{equation}\label{eq:LA2}
\Lambda(x,t)\;\ge\;a'(x)\,t^{p'(x)}\;\ge\;d_*\,t^{\underline{p}}\;=\;\Phi(t).
\end{equation}

\medskip

\noindent\emph{Case A3:} $a'(x_0)<\frac{d}{2}$.

By \eqref{GG3}(iii), $a'(x_0)+b'(x_0)\ge d$, hence
\[
b'(x_0)\;\ge\;d-a'(x_0)\;>\;\frac{d}{2}.
\]
By \eqref{GG3}(v) and $R<R_*\le \tfrac12\left(\frac{d_*}{H_*}\right)^{\frac{1}{\bar\alpha}}$,
\[
b'(x)\;\ge\;b'(x_0)-[b']_{0,\bar\alpha}(2R)^{\bar\alpha}
\;\ge\;\frac{d}{2}-\frac{d}{4}\;=\;d_*\qquad(x\in B).
\]
Define
\[
\Phi(t):=d_*\,t^{\underline{q}}\,\ell(t).
\]
Then for all $x\in B$ and $t\ge 1$,
\begin{equation}\label{eq:LB}
\Lambda(x,t)\;\ge\;b'(x)\,t^{q'(x)}\ell(t)\;\ge\;d_*\,t^{\underline{q}}\,\ell(t)\;=\;\Phi(t).
\end{equation}

In all three cases, we thus have
\begin{equation}\label{eq:Llower}
\Lambda(x,t)\;\ge\;\Phi(t)\qquad(x\in B,\,t\ge 1).
\end{equation}
We will use the following estimate on \(\tau^{\overline{p}-\underline{p}}\)
and \(\tau^{\overline{q}-\underline{q}}\), which plays a key role in the next Step.

By \eqref{GG3}(iv), for all $x_0\in\Omega$ and $R<R_{\mathrm{G}}$,
\[
\sup_{B_R(x_0)}p-\inf_{B_R(x_0)}p'\;\le\;\frac{L_0}{|\ln(R)|},
\qquad
\sup_{B_R(x_0)}q-\inf_{B_R(x_0)}q'\;\le\;\frac{L_0}{|\ln(R)|}.
\]
Fix $x_0$, $R<R_*$ and one index $i$. Set
\[
\omega:=\max\left\{\overline{p}-\underline{p},\;\overline{q}-\underline{q}\right\}
\;\le\;\frac{L_0}{|\ln(R)|}.
\]
For $\tau\in\left[1,2\frac{M}{R}\right]$, we have
\[
\ln\tau\;\le\;\ln(2M)+|\ln R|,
\]
 Since $R<R_*\le (4M)^{-1}$, we have
$|\ln(R)|>\ln(4M)\ge \ln(2M)+\ln 2> \ln(2M)$, hence
\[
\ln\tau\;\le\;2\,|\ln(R)|.
\]
Multiplying,
\[
\omega\,\ln\tau\;\le\;\frac{L_0}{|\ln(R)|}\cdot 2\,|\ln(R)|\;=\;2L_0.
\]
Exponentiating gives
\begin{equation}\label{eq:L2new}
\tau^{\overline{p}-\underline{p}}\;\le\;L^2,
\qquad
\tau^{\overline{q}-\underline{q}}\;\le\;L^2,
\qquad
L:=e^{L_0}.
\end{equation}
We now establish the key properties of the function 
$\Phi$.
Define
\[
\varepsilon :=
\begin{cases}
\underline{p}, & \text{in Cases A1 and A2},\\[2pt]
\underline{q}, & \text{in Case A3}.
\end{cases}
\]

\medskip
\begin{enumerate}
\item [(a)] \noindent\emph{ $\Phi$ is nondecreasing}:
In Case A1, $\Phi(t)=d_*t^{\underline{p}}+b_0t^{\underline{q}}\ell(t)$ is a sum of two nondecreasing functions of $t$ (since $\underline{p}>0$, $\underline{q}>0$, and $\ell$ is nondecreasing). In Case A2, $\Phi(t)=d_*t^{\underline{p}}$ is nondecreasing. In Case A3, $\Phi(t)=d_*t^{\underline{q}}\ell(t)$ is nondecreasing.
\item [(b)] \noindent\emph{ $\Phi(1)\le \bar{a}(1+\ell(1))$}:
In Case A1, $\Phi(1)=d_*+b_0\ell(1)\le \bar{a}+\bar{a}\ell(1)=\bar{a}(1+\ell(1))$, since $d_*\le \frac{\bar{a}}{2}\le \bar{a}$ and $b_0\le \bar{a}$. In Case A2, $\Phi(1)=d_*\le \bar{a}\le \bar{a}(1+\ell(1))$. In Case A3, $\Phi(1)=d_*\ell(1)\le \bar{a}\ell(1)\le \bar{a}(1+\ell(1))$.
\item [(c)] \noindent\emph{ $s\mapsto \frac{\Phi(s)}{s^{\varepsilon}}$ is nondecreasing on $(0,\infty)$}:
In Case A1, with $\varepsilon=\underline{p}$:
\[
\frac{\Phi(s)}{s^{\underline{p}}}
\;=\;d_*+b_0\,s^{\underline{q}-\underline{p}}\,\ell(s).
\]
Since $\underline{q}-\underline{p}\ge 0$ and $\ell$ is nondecreasing, $s\mapsto s^{\underline{q}-\underline{p}}\ell(s)$ is nondecreasing, hence $s\mapsto \frac{\Phi(s)}{s^{\underline{p}}}$ is nondecreasing.

In Case A2, with $\varepsilon=\underline{p}$, $\frac{\Phi(s)}{s^{\underline{p}}}=d_*$ is constant, hence nondecreasing.

In Case A3, with $\varepsilon=\underline{q}$, $\frac{\Phi(s)}{s^{\underline{q}}}=d_*\ell(s)$ is nondecreasing since $\ell$ is nondecreasing.
\item [(d)] \noindent\emph{ $\Phi(t)\ge d_\ell\,t^{\varepsilon}$ for $t\ge 1$}:
In Case A1, $\Phi(t)\ge d_*\,t^{\underline{p}}\ge d_\ell\,t^{\underline{p}}=d_\ell\,t^{\varepsilon}$, because $d_*\ge d_\ell$ (recall $d_\ell=d_*\min\{1,\ell(1)\}\le d_*$).

In Case A2, $\Phi(t)=d_*\,t^{\underline{p}}\ge d_\ell\,t^{\varepsilon}$.

In Case A3, $\Phi(t)=d_*\,t^{\underline{q}}\ell(t)\ge d_*\,\ell(1)\,t^{\underline{q}}\ge d_\ell\,t^{\underline{q}}=d_\ell\,t^{\varepsilon}$,
since $\ell(t)\ge \ell(1)$ for $t\ge 1$ (as $\ell$ is nondecreasing with $\ell(1)=\ln(\bar\delta+\eta)>0$) and $d_\ell=d_*\min\{1,\ell(1)\}$.
\item [(e)]\noindent\emph{ Upper bound on $U$}:
For all $x\in B$, $\lambda\ge 1$, and $t\in\left[1,2\frac{M}{R}\right]$,
\begin{equation}\label{eq:Uupper}
U(x,\lambda t)\;\le\;\kappa_0\,\lambda^{\mu}\,\Phi(t),
\end{equation}
where
\[
\mu :=
\begin{cases}
q^+ + 1, & \text{in Cases A1 and A3},\\[2pt]
p^+ + \bar\alpha, & \text{in Case A2}.
\end{cases}
\]
and $\kappa_0$ is a constant depending only on $n,p^\pm,q^\pm,d,\bar{a},\Theta,\ell(1),M,L,\bar\alpha, C_*$ (its precise value is given below).

\emph{Proof of \eqref{eq:Uupper}.}
Let $x\in B$, $\lambda\ge 1$, $t\in\left[1,2\frac{M}{R}\right]$; so $\lambda t\ge 1$.

\smallskip

\emph{Case A1.}
Since $b_0>4[b']_{0,\bar\alpha}R^{\bar\alpha}$, we have
\[
\sup_B b'\;\le\;b_0+2[b']_{0,\bar\alpha}R^{\bar\alpha}\;\le\;\tfrac32\,b_0,
\]
so $b(x)\le \Theta\,b_0\cdot \tfrac32 = \tfrac32\Theta\,b_0$, using $b\le \Theta b'$ by \eqref{GG3}(v). Since $a,b\le \bar{a}$, $\overline{p}\le \overline{q}\le q^+$, and $\ell(\lambda t)\le \lambda\ell(t)$ by \eqref{elem:log},
\begin{align*}
U(x,\lambda t)
&\le \bar{a}\,(\lambda t)^{\overline{p}}+\tfrac32\Theta\,b_0\,(\lambda t)^{\overline{q}}\ell(\lambda t)\\
&\le \bar{a}\,(\lambda t)^{\overline{p}}+\tfrac32\Theta\,b_0\,\lambda\,(\lambda t)^{\overline{q}}\ell(t)\\
&\le \lambda^{q^++1}\Bigl[\bar{a}\,t^{\overline{p}}+\tfrac32\Theta\,b_0\,t^{\overline{q}}\ell(t)\Bigr].
\end{align*}
By \eqref{GG3}(iv) and \eqref{eq:L2new} , $t^{\overline{p}}\le L^2\,t^{\underline{p}}$ and $t^{\overline{q}}\le L^2\,t^{\underline{q}}$. Hence
\[
U(x,\lambda t)\;\le\;\lambda^{q^++1}L^2\,\max\Bigl\{\frac{\bar{a}}{d_*},\,\frac{3\Theta}{2}\Bigr\}\,\Phi(t).
\]

\smallskip

\emph{Case A2.}
By \eqref{GG6a} evaluated at the argument $\lambda t\ge 1$,
\[
(\lambda t)^{q(x)}\ell(\lambda t)\;\le\;C_*\,(\lambda t)^{p(x)+\bar\alpha}\;\le\;C_*\,(\lambda t)^{\overline{p}+\bar\alpha}.
\]
Also, using $b_0\le 4[b']_{0,\bar\alpha}R^{\bar\alpha}$ and \eqref{GG3}(v),
\[
\sup_B b'\;\le\;b_0+2[b']_{0,\bar\alpha}R^{\bar\alpha}\;\le\;6[b']_{0,\bar\alpha}R^{\bar\alpha}=:\tilde{b}(R),
\]
so $b(x)\le \Theta\,\tilde{b}(R)$ for $x\in B$. Therefore,
\[
U(x,\lambda t)\;\le\;\bar{a}\,(\lambda t)^{\overline{p}}+\Theta\,C_*\,\tilde{b}(R)\,(\lambda t)^{\overline{p}+\bar\alpha}
\;\le\;\lambda^{p^++\bar\alpha}\,t^{\overline{p}}\Bigl[\bar{a}+\Theta\,C_*\,\tilde{b}(R)\,t^{\bar\alpha}\Bigr].
\]
Since $t\in\left[1,2\frac{M}{R}\right]$,
\[
\tilde{b}(R)\,t^{\bar\alpha}\;\le\;6[b']_{0,\bar\alpha}R^{\bar\alpha}\cdot\left(2\frac{M}{R}\right)^{\bar\alpha}\;\le\;6[b]_*(2M)^{\bar\alpha}.
\]
By \eqref{eq:L2new}, $t^{\overline{p}}\le L^2\,t^{\underline{p}}$, so
\[
U(x,\lambda t)\;\le\;\lambda^{p^++\bar\alpha}\,\frac{L^2\,(\bar{a}+6\Theta\,C_*[b]_*(2M)^{\bar\alpha})}{d_*}\,\Phi(t).
\]

\smallskip

\emph{Case A3.}
Using $a,b\le \bar{a}$, $\overline{p}\le \overline{q}\le q^+$, $1\le \frac{\ell(t)}{\ell(1)}$ (as $\ell(t)\ge \ell(1)>0$ for $t\ge 1$), and $\ell(\lambda t)\le \lambda\ell(t)$ by \eqref{elem:log},
\begin{align*}
U(x,\lambda t)
&\le\bar{a}(\lambda t)^{\overline{p}}+\bar{a}(\lambda t)^{\overline{q}}\ell(\lambda t)\\
&\le\bar{a}\lambda^{q^+}t^{\overline{q}}\,\frac{\ell(t)}{\ell(1)}+\bar{a}\lambda^{q^++1}t^{\overline{q}}\ell(t)\\
&\le\lambda^{q^++1}\bar{a}\Bigl(1+\frac{1}{\ell(1)}\Bigr)t^{\overline{q}}\ell(t).
\end{align*}
By \eqref{eq:L2new}, $t^{\overline{q}}\le L^2\,t^{\underline{q}}$, so
\[
U(x,\lambda t)\;\le\;\lambda^{q^++1}\,\frac{\bar{a}L^2\,\left(1+\frac{1}{\ell(1)}\right)}{d_*}\,\Phi(t).
\]

\smallskip

Combining the three branches, we obtain \eqref{eq:Uupper} with
\[
\kappa_0:=L^2\left(\frac{\bar{a}\,\left(4+\frac{1}{\ell(1)}\right)+6\Theta\,C_*[b]_*(2M)^{\bar\alpha}}{d_*}+\frac{3\Theta}{2}\right),
\]
and $\mu$ as defined above.
\end{enumerate}

We now combine the two one-pair estimates. Applying the one-pair result to
each index \(i = 1,2\), we obtain functions \(\Phi_1\) and \(\Phi_2\),
exponents \(\varepsilon_i\), constants \(\mu_i\), and \(\kappa_{0,i}\) as in
the one-pair estimates. Setting
\[
\Phi := \Phi_1 + \Phi_2,
\]
we verify that \(\Phi\) satisfies the four properties (1)--(4) of
Lemma~\ref{lem22}.

Define
\[
f := \min\{\varepsilon_1,\varepsilon_2\}, \qquad
\kappa := \kappa_*,
\]
where \(\kappa_*\) is chosen larger than \(\kappa_{0,1}\) and
\(\kappa_{0,2}\).

\medskip

\noindent\emph{Proof of (1).}
By \eqref{GG3}(i) and \eqref{eq:Llower}, for all $x\in B$ and $t\ge 1$,
\[
G(x,t)\;\ge\;\sum_{i=1}^2 \Lambda_i(x,t)\;\ge\;\sum_{i=1}^2 \Phi_i(t)\;=\;\Phi(t).
\]
Hence $\Phi(t)\le G(x,t)$ for $t\ge 1$.
For $0\le\sigma<1$, since $\Phi$ is nondecreasing (as a sum of nondecreasing functions),
\[
\Phi(\sigma)\;\le\;\Phi(1)\;=\;\Phi_1(1)+\Phi_2(1)\;\le\;2\bar{a}(1+\ell(1))\;=\;c_{\mathrm{sm}}.
\]
Since $G(x,\sigma)\ge 0$ and $c_{\mathrm{sm}}\ge 0$, this gives
\[
\Phi(\sigma)\;\le\;c_{\mathrm{sm}}\;\le\;G(x,\sigma)+c_{\mathrm{sm}}.
\]

\medskip

\noindent\emph{Proof of (2).}
By property (c), each $s\mapsto \frac{\Phi_i(s)}{s^{\varepsilon_i}}$ is nondecreasing. Since $\varepsilon_i\ge f$ and $s>0$, the function
\[
\frac{\Phi_i(s)}{s^{f}}\;=\;s^{\varepsilon_i-f}\cdot\frac{\Phi_i(s)}{s^{\varepsilon_i}}
\]
is a product of two nondecreasing nonnegative functions on $(0,\infty)$, hence nondecreasing. Therefore, the same holds for
\[
s\;\longmapsto\;\frac{\Phi(s)}{s^{f}}\;=\;\frac{\Phi_1(s)}{s^{f}}+\frac{\Phi_2(s)}{s^{f}}.
\]
Now let $t,\lambda>0$.

\emph{Case 1: $t\ge\lambda$.} By monotonicity of $\frac{\Phi(s)}{s^{f}}$,
\[
\frac{\Phi(t)}{t^{f}}\;\ge\;\frac{\Phi(\lambda)}{\lambda^{f}}
\;\Longrightarrow\;
t^{f}\;\le\;\lambda^{f}\,\frac{\Phi(t)}{\Phi(\lambda)}
\;\le\;\lambda^{f}\Bigl[1+\frac{\Phi(t)}{\Phi(\lambda)}\Bigr].
\]

\emph{Case 2: $t<\lambda$.} Then $t^{f}<\lambda^{f}\le \lambda^{f}\left[1+\frac{\Phi(t)}{\Phi(\lambda)}\right]$.

In both cases, (2) holds. Note also that $\Phi(s)>0$ for all $s>0$, since each $\Phi_i(s)>0$: in Case A1, $\Phi_i(s)\ge d_* s^{\underline{p}_i}>0$; in Case A2, $\Phi_i(s)=d_*s^{\underline{p}_i}>0$; in Case A3, $\Phi_i(s)=d_*s^{\underline{q}_i}\ell(s)>0$ because $\ell(s)>0$ for $s>0$ (as $\bar\delta+\eta>1$, $\bar\delta\ge 1$).

\medskip

\noindent\emph{Proof of (3).}
Choose $i^*\in\{1,2\}$ such that $\varepsilon_{i^*}=f$ (if $\varepsilon_1=\varepsilon_2$, take $i^*=1$). Then by property (d),
\[
\Phi(t)\;\ge\;\Phi_{i^*}(t)\;\ge\;d_\ell\,t^{\varepsilon_{i^*}}\;=\;d_\ell\,t^{f}
\qquad(t\ge 1).
\]
In particular, since $t^{f}\ge 1$ for $t\ge 1$, we also have $\Phi(t)\ge d_\ell$.

\medskip

\noindent\emph{Proof of (4).}
For $x\in B$, $\lambda\ge 1$, $t\in\left[1,2\frac{M}{R}\right]$, by \eqref{GG3}(i) and \eqref{eq:Uupper},
\[
G(x,\lambda t)\;\le\;\sum_{i=1}^2 U_i(x,\lambda t)
\;\le\;\sum_{i=1}^2 \kappa_{0,i}\,\lambda^{\mu_i}\,\Phi_i(t).
\]
Since $\mu_i\le \mu_*$ and $\lambda\ge 1$, we have $\lambda^{\mu_i}\le \lambda^{\mu_*}$. Hence
\[
G(x,\lambda t)\;\le\;\kappa_*\,\lambda^{\mu_*}\sum_{i=1}^2 \Phi_i(t)
\;=\;\kappa\,\lambda^{\mu_*}\,\Phi(t),
\]
which is (4).

\medskip

This completes the proof of Lemma~\ref{lem22}.
\end{proof}

\begin{lemma}\label{lem:iter}
Let $0<\epsilon\le 1$, $\alpha\in(0,\epsilon]$, $c_{1}>0$,
$\theta\in(0,4^{-\alpha}]$, $\rho_{1}>0$, and let
$\vartheta:(0,\rho_{1}]\to[0,\infty)$ be nondecreasing. Suppose that for every
$r\in\left(0,\frac{\rho_{1}}{4}\right]$, either $$\vartheta(r)\le c_{1}r^{\epsilon}$$ or
$$\vartheta(r)\le\theta\vartheta(4r).$$
Then for all $r\in(0,\rho_{1}]$,
\[
\vartheta(r)\le 4^{\alpha}\max\left\{c_{1}\rho_{1}^{\epsilon-\alpha},\;
\vartheta(\rho_{1})\rho_{1}^{-\alpha}\right\}\,r^{\alpha}.
\]
\end{lemma}

\begin{proof}
Put $C:=\max\left\{c_{1}\rho_{1}^{\epsilon-\alpha},\vartheta(\rho_{1})\rho_{1}^{-\alpha}\right\}$
and $r_{j}:=4^{-j}\rho_{1}$. We prove $\vartheta\left(r_{j}\right)\le Cr_{j}^{\alpha}$
by induction.\\
For $j=0$ this holds by the choice of $C$. If it holds for $j$,
then at $r_{j+1}\le\frac{\rho_{1}}{4}$: either
\[
\vartheta\left(r_{j+1}\right)\le c_{1}r_{j+1}^{\epsilon}
=c_{1}r_{j+1}^{\epsilon-\alpha}r_{j+1}^{\alpha}
\le c_{1}\rho_{1}^{\epsilon-\alpha}r_{j+1}^{\alpha}
\le Cr_{j+1}^{\alpha},
\]
or
\[
\vartheta\left(r_{j+1}\right)\le\theta\vartheta\left(r_{j}\right)
\le4^{-\alpha}Cr_{j}^{\alpha}
=Cr_{j+1}^{\alpha}.
\]
For $r\in\left(r_{j+1},r_{j}\right]$, monotonicity gives
$\vartheta(r)\le\vartheta\left(r_{j}\right)\le Cr_{j}^{\alpha}
=4^{\alpha}Cr_{j+1}^{\alpha}\le4^{\alpha}Cr^{\alpha}$.
\end{proof}
\begin{proof}[\textbf{Proof of Theorem \ref{Theorem 2.1}}]
Let $u \in \mathcal{B}_{G(x,t)}(\Omega, M, \gamma, \gamma_1, \delta)$, where $G(x,t)$ satisfies conditions  \eqref{D22} and \eqref{GG3}. Without loss of generality, assume $M \geq 1$, $\gamma \geq 1$ and $L\geq 1$.

To prove Theorem~\ref{Theorem 2.1}, it suffices to show that for every point $x_0 \in \Omega$ there exists a ball $B_{R_0}(x_0) \subset \Omega$ such that $u \in C^{0,\alpha}\left(B_{\frac{R_0}{2}}(x_0)\right)$ (for a suitable representative of $u$), with $\alpha$ the constant stated in the theorem.

Fix $x_0 \in \Omega$ and put $R_0:=\frac12\min\{1,R_*,\operatorname{dist}(x_0,\partial\Omega)\}$, so that $R_0<R_*$, $R_0<1$ and $B_{R_0}(y)\subset\Omega$ for every $y\in B_{R_0}(x_0)$. In what follows, $y\in B_{R_0}(x_0)$ is an arbitrary centre and we write $B_R=B_R(y)$; all constants will be independent of $y$ and $R$. Then, for an arbitrary $R \in (0, R_0]$, at least one of the functions $w:= \pm u$ satisfies the following condition.
\begin{equation}\label{2.8}
\left\vert\left\lbrace x \in B_{\frac{R}{2}} : w(x) > \max_{B_R} w(x) - \tfrac{1}{2} \mathrm{osc}_{B_R} u \right\rbrace\right\vert \leq \tfrac{1}{2} \, \left\vert B_{\frac{R}{2}}\right\vert.
\end{equation}
From this point on, we denote by $w$ the function (either $u$ or $-u$) that satisfies inequality \eqref{2.8}. Define
\begin{equation}\label{2.9}
\tau := \max\left\{2, \frac{2}{\delta} \right\}, \quad \psi := \tau^{-1} \, \mathrm{osc}_{B_R} u,
\end{equation}
and
\begin{equation}\label{2.10}
k' := \max_{B_R} w(x) - \psi.
\end{equation}
It follows that
\begin{equation}\label{2.11}
k' \geq \max_{B_R} w(x) - \tfrac{1}{2} \, \mathrm{osc}_{B_R} u,
\end{equation}
and, since $\mathrm{osc}_{B_R}u\le 2M$ and $\tau\ge \frac{2}{\delta}$ (so that $\psi\le \frac{\delta}{2}\mathrm{osc}_{B_R}u\le\delta M$; this is true for every $\delta>0$),
\begin{equation}\label{2.12}
k' \geq \max_{B_R} w(x) - \delta M.
\end{equation}
Note also that $\psi\le M$ because $\tau\ge2$.
\begin{remark}
Note that inequality~\eqref{2.12} guarantees that conditions~\eqref{2.4} and~\eqref{2.5} hold for all \( k \geq k' \) and for all \( \rho \leq R \).
\end{remark}
Under these assumptions and notations, the proof of Theorem~\ref{Theorem 2.1} is completed through the lemmas presented below, specifically Lemmas~\ref{Lemma 2.5} to~\ref{Lemma 2.8}.
\begin{lemma}\label{Lemma 2.5}
There is a positive constant \(\theta = \theta(g^-,n,p^{\pm},q^{\pm},d_{\ell},\kappa_{*},\ell(1),\bar{a},\gamma,\gamma_1)\) such that if
\begin{equation}\label{2.13}
\left|A_{k^0, \frac{R}{2}}\right| \leq \theta R^n,
\end{equation}
then at least one of the following two inequalities hol{\rm d}s:
\begin{equation}\label{2.14}
H \leq R,
\end{equation}
or
\begin{equation}\label{2.15}
\max_{B_{\frac{R}{4}}} w(x) \leq k^0 + \frac{H}{2},
\end{equation}
where
\[
0 < H < \psi, \quad k^0 = \max_{B_R} w(x) - H.
\]
\end{lemma}

\begin{proof}
Put $t_0:=\frac{H}{R}$. If inequality \eqref{2.14} does not hold, i.e.\ $H>R$, then $1<t_0<\frac{\psi}{R}\le\frac{M}{R}$. Let $\Phi,f$ be as in Lemma~\ref{lem22}. For $h=0,1,2,\ldots$ define
\[
\rho_h:=\frac{R}{4}+\frac{R}{2^{h+2}},\quad
k_h:=k^0+\frac{H}{2}-\frac{H}{2^{h+1}},\quad
Y_h:=R^{-n}\left|A_{k_h,\rho_h}\right|,
\]
and $D_{h+1}:=A_{k_h,\rho_{h+1}}\setminus A_{k_{h+1},\rho_{h+1}}$. Then $\rho_h\downarrow \frac{R}{4}$, $k_h\uparrow k^0+\frac{H}{2}$, $\rho_h-\rho_{h+1}=R2^{-(h+3)}$, $k_{h+1}-k_h=H2^{-(h+2)}$, and $k_h \geq k^0 \geq \max_{B_R} w - \psi$, so \eqref{2.4} applies.

\smallskip
\noindent On $A_{k_h,\rho_h}$ we have $0<w-k_h \leq H\left(\tfrac12+2^{-h-1}\right) \leq H$, hence
$$\frac{(w-k_h)}{(\rho_h-\rho_{h+1})} \leq 2^{h+3} t_0.$$ 
By monotonicity of $G(x,\cdot)$ and Lemma~\ref{lem22}(4) with $\lambda=2^{h+3}$, $t=t_0$, we get
\[
G\Bigl(x,\frac{w-k_h}{\rho_h-\rho_{h+1}}\Bigr) \leq \kappa \left(2^{h+3}\right)^{\mu_*}\Phi(t_0).
\]
By \eqref{2.4},
\begin{equation}\label{2.17new}
\int_{A_{k_h,\rho_{h+1}}} G(x,|Dw|)\,\mathrm{d}x \leq \left[\gamma\kappa\left(2^{h+3}\right)^{\mu_*}\Phi(t_0)+\gamma_1\right] R^n Y_h=:E_h.
\end{equation}

\smallskip
\noindent
Using Lemma \ref{lem:conv} with $\lambda=t_0$, $t=|Dw(x)|$, $x \in D_{h+1}$ and the fact that $G(x,t_0) \geq \Phi(t_0)$, integration over $D_{h+1}$ with $|D_{h+1}| \leq R^n Y_h$ and \eqref{2.17new} gives
\begin{align*}
\int_{D_{h+1}} |Dw|\,\mathrm{d}x &\leq \int_{D_{h+1}} t_0+\frac{t_0}{g^-\Phi(t_0)}G(x,|Dw|)\,\mathrm{d}x\\
&\leq t_0 R^n Y_h\Bigl[1+\frac{1}{g^-}\Bigl(\gamma\kappa(2^{h+3})^{\mu_*}+\frac{\gamma_1}{\Phi(t_0)}\Bigr)\Bigr].
\end{align*}
By Lemma~\ref{lem22}(3), $\Phi(t_0) \geq d_\ell$. Thus, with
\[
K:=1+\frac{1}{g^-}\Bigl(\gamma\kappa 2^{3\mu_*}+\frac{\gamma_1}{d_\ell}\Bigr),
\]
we obtain, for every $h \geq 0$,
\begin{equation}\label{2.20new}
\int_{D_{h+1}} |Dw|\,\mathrm{d}x \leq K 2^{h\mu_*} t_0 R^n Y_h.
\end{equation}

\smallskip
\noindent Require $\theta \leq \tfrac12 \omega_n 4^{-n}$. Since $B_{\frac{R}{4}} \subset B_{\rho_{h+1}}$ and $A_{k_h,\rho_{h+1}} \subset A_{k^0,\frac{R}{2}}$,
\[
\left|B_{\rho_{h+1}}\setminus A_{k_h,\rho_{h+1}}\right|\geq \left|B_{\frac{R}{4}}\setminus A_{k^0,\frac{R}{2}}\right| \geq \omega_n 4^{-n}R^n-\theta R^n \geq 2^{-2n-1}\omega_n R^n.
\]
Lemma~\ref{Lemma 2.2} with $k=k_h$, $l=k_{h+1}$, $\rho=\rho_{h+1} \leq R$ gives
\begin{equation}\label{2.21new}
Y_{h+1}^{1-\frac{1}{n}} \leq \beta 2^{2n+3}\omega_n^{-1}\frac{2^h}{H R^{n-1}}\int_{D_{h+1}} |Dw|\,\mathrm{d}x.
\end{equation}

\smallskip
\noindent Combining \eqref{2.20new} and \eqref{2.21new} with $t_0 \frac{R}{H}=1$, we obtain
\[
Y_{h+1}^{1-\frac{1}{n}} \leq \beta 2^{2n+3}\omega_n^{-1}K 2^{h(\mu_*+1)}Y_h,
\]
i.e. $$Y_{h+1} \leq c b^h Y_h^{1+\varepsilon}$$ with
\[
\varepsilon:=\frac{1}{n-1},\quad
b:=\left(2^{\mu_*+1}\right)^{\frac{n}{n-1}},\quad
c:=\left(\beta 2^{2n+3}\omega_n^{-1}K\right)^{\frac{n}{n-1}}.
\]
Set
\[
\theta:=\min\Bigl\{\tfrac12 \omega_n 4^{-n},\; c^{\frac{-1}{\varepsilon}}b^{\frac{-1}{\varepsilon^2}}\Bigr\}.
\]
Then $Y_0 \leq \theta$, so Lemma~\ref{Lemma 2.1} gives $Y_h \to 0$ as $h \to \infty$, hence
\[
\max_{B_{\frac{R}{4}}} w \leq \lim_{h\to \infty} k_h = k^0+\frac{H}{2}.
\]
The constants do not depend on $R$, $H$, or the centre $y$.
\end{proof}

\begin{lemma}\label{Lemma 2.6}
For any given \(\theta > 0\), there is a natural number \(s = s(n,p^{-},q^{+},\beta,\bar{a},\ell(1),d_{\ell},\kappa_{*},\gamma,\gamma_1,\theta) > 2\) such that either
\begin{equation}\label{2.29}
\psi \leq 2^s R
\end{equation}
or
\begin{equation}\label{2.30}
\left|A_{k^0, \frac{R}{2}}\right| \leq \theta R^n
\end{equation}
holds, where
\[
k^0 = \max_{B_R} w(x) - 2^{-s+1}\psi.
\]
\end{lemma}

\begin{proof}
Assume that \eqref{2.29} does not hold, i.e.\
\[
\psi>2^s R,
\]
where $s>2$ will be chosen later. Let $\Phi,f$ be as in Lemma~\ref{lem22}. For $t=0,\ldots,s-1$, define
\[
k_t:=\max_{B_R}w-2^{-t}\psi,\quad
D_t:=A_{k_t,\frac{R}{2}}\setminus A_{k_{t+1},\frac{R}{2}},\quad
\tau_t:=\frac{2^{-t+1}\psi}{R}.
\]
For $t \leq s-2$: $$\tau_t>2^{s-t+1}\geq 8\quad \text{and}\quad \tau_t \leq 2\frac{\psi}{R} \leq 2\frac{M}{R}.$$
Also $k_t \geq \max_{B_R}w-\psi$, so \eqref{2.4} applies, and $\left|A_{k_t,\frac{R}{2}}\right| \leq \tfrac12 \left|B_{\frac{R}{2}}\right|$ by \eqref{2.8}.

\smallskip
\noindent On $A_{k_t,R}$, we have $0<w-k_t \leq 2^{-t}\psi$, so $$\frac{2(w-k_t)}{R} \leq \tau_t.$$ By monotonicity of $G(x,\cdot)$ and Lemma~\ref{lem22}(4) with $\lambda=1$,
\[
G\Bigl(x,\frac{2(w-k_t)}{R}\Bigr) \leq G(x,\tau_t) \leq \kappa\Phi(\tau_t).
\]
By \eqref{2.4} with $\rho=R$ and $\sigma\rho=\frac{R}{2}$,
\[
\int_{A_{k_t,\frac{R}{2}}}G(x,|Dw|)\,\mathrm{d}x \leq \bigl[\gamma\kappa\Phi(\tau_t)+\gamma_1\bigr]|A_{k_t,R}|.
\]
By Lemma~\ref{lem22}(1) and using the fact that $\left|A_{k_t,\frac{R}{2}}\right| \leq \left|A_{k_t,R}\right|$,
\begin{equation}\label{2.33new}
\int_{A_{k_t,\frac{R}{2}}}\Phi(|Dw|)\,\mathrm{d}x \leq \bigl[\gamma\kappa\Phi(\tau_t)+\gamma_1+c_{\mathrm{sm}}\bigr]\left|A_{k_t,R}\right|.
\end{equation}

\smallskip
\noindent Applying Lemma~\ref{lem22}(2) with $\lambda=\tau_t$, $t=|Dw(x)|$, integrating over $A_{k_t,\frac{R}{2}}$ and using \eqref{2.33new}, we get
\[
\int_{A_{k_t,\frac{R}{2}}}|Dw|^{\,f}\,\mathrm{d}x \leq \tau_t^{\,f}\left|A_{k_t,R}\right|\left[1+\gamma\kappa+\frac{\gamma_1+c_{\mathrm{sm}}}{\Phi(\tau_t)}\right].
\]
By Lemma~\ref{lem22}(3), $\Phi(\tau_t) \geq d_\ell$. With $|A_{k_t,R}| \leq \omega_n R^n$,
\begin{equation}\label{2.34new}
\int_{A_{k_t,\frac{R}{2}}}|Dw|^{\,f}\,\mathrm{d}x \leq C_1 \tau_t^{\,f}R^n,\quad
C_1:=\omega_n\left[1+\gamma\kappa+\frac{\gamma_1+2\bar{a}(1+\ell(1))}{d_\ell}\right].
\end{equation}

\smallskip
\noindent On other side, by H\"older, $D_t \subset A_{k_t,\frac{R}{2}}$ and \eqref{2.34new},
\[
\int_{D_t}|Dw|\,\mathrm{d}x \leq \left(\int_{A_{k_t,\frac{R}{2}}}|Dw|^{\,f}\,\mathrm{d}x\right)^{\frac{1}{f}}|D_t|^{1-\frac{1}{f}} \leq C_1^{\frac{1}{f}}R^{\frac{n}{f}}\tau_t |D_t|^{1-\frac{1}{f}}.
\]

\smallskip
\noindent Applying Lemma~\ref{Lemma 2.2} with $k=k_t$, $l=k_{t+1}$, $\rho=\frac{R}{2}$, using $\left|B_{\frac{R}{2}}\setminus A_{k_t,\frac{R}{2}}\right| \geq \tfrac12\left|B_{\frac{R}{2}}\right|$ and $A_{k_{s-1},\frac{R}{2}} \subset A_{k_{t+1},\frac{R}{2}}$,
\[
2^{-t-1}\psi \left|A_{k_{s-1},\frac{R}{2}}\right|^{1-\frac{1}{n}} \leq 2\beta\omega_n^{-1}\int_{D_t}|Dw|\,\mathrm{d}x.
\]
Since $\tau_t=4 \cdot 2^{-t-1}\frac{\psi}{R}$, the factor $2^{-t-1}\psi$ cancels:
\[
\left|A_{k_{s-1},\frac{R}{2}}\right|^{1-\frac{1}{n}} \leq C_2 R^{\frac{n}{f}-1}|D_t|^{1-\frac1f},\quad
C_2:=8\beta\omega_n^{-1}C_1^{\frac{1}{f}}.
\]
Raising to the power $\frac{f}{f-1}$,
\begin{equation}\label{2.35new}
\left|A_{k_{s-1},\frac{R}{2}}\right|^{\frac{(n-1)f}{n(f-1)}} \leq C_2^{\frac{f}{f-1}}R^{\frac{n-f}{f-1}}|D_t|,\quad t=0,\ldots,s-2.
\end{equation}

\smallskip
\noindent Summing \eqref{2.35new} over $t=0,\ldots,s-2$, with $\sum_t|D_t| \leq \omega_n 2^{-n}R^n$ and $\frac{n-f}{f-1}+n=\frac{f(n-1)}{f-1}$, we get
\[
\left|A_{k_{s-1},\frac{R}{2}}\right| \leq \Bigl(\frac{C_3}{s-1}\Bigr)^{\epsilon(f)}R^n,\quad
C_3:=C_2^{\frac{f}{f-1}}\omega_n 2^{-n},\quad
\epsilon(r):=\frac{n(r-1)}{(n-1)r}.
\]
Uniformity in $f \in [p^-,q^+]$: $C_1^{\frac{1}{f}} \leq \max\left\{1,C_1\right\}$ and $r\mapsto \frac{r}{r-1}$ is decreasing, so
\[
C_3 \leq \bar C_3:=\left(\max\left\{1,8\beta\omega_n^{-1}\max\left\{1,C_1\right\}\right\}\right)^{\frac{p^-}{p^--1}}\omega_n 2^{-n}.
\]
Since $\epsilon$ is increasing, $s-1 \geq \bar C_3$ gives $$\left(\frac{C_3}{s-1}\right)^{\epsilon(f)} \leq \left( \frac{\bar C_3}{s-1}\right)^{\epsilon(p^-)}.$$
Choose $s>2$ such that $$s \geq \bar C_3 \theta^{-\frac{1}{\epsilon(p^-)}}+1.$$ 
Then $\left|A_{k_{s-1},\frac{R}{2}}\right| \leq \theta R^n$, i.e.\ the second alternative holds with $k^0=k_{s-1}$. The integer $s$ does not depend on $R$ or on the centre.
\end{proof}

\begin{lemma}\label{Lemma 2.7}
There is a number \(s = s(g^-,n, p^{\pm}, q^{\pm}, \beta, \bar{a}, \ell(1), d_{\ell}, \kappa_{*}, \gamma, \gamma_1) > 2\) such that
\begin{equation}\label{2.40}
\psi \leq 2^s \max \left\{ \max_{B_R} w(x) - \max_{B_{\frac{R}{4}}} w(x), R \right\},\text{    for any } R \in (0, R_0]. 
\end{equation}
\end{lemma}
\begin{proof}
Let \( R \in (0, R_0] \), and let
\[
\theta = \theta(g^-,n, p^{\pm}, q^{\pm}, d_{\ell}, \kappa_{*}, \ell(1), \bar{a}, \gamma, \gamma_1)
\]
be the constant appearing in Lemma~\ref{Lemma 2.5}. Applying Lemma~\ref{Lemma 2.6} with this choice of \( \theta \), we obtain a constant
\[
s = s(\theta, n, p^{\pm}, q^{\pm}, \beta, \bar{a}, \ell(1), d_{\ell}, \kappa_{*}, \gamma, \gamma_1) > 2
\]
such that at least one of the inequalities \eqref{2.29} or \eqref{2.30} holds.
If \eqref{2.29} holds, then \eqref{2.40} follows immediately.
Otherwise, \eqref{2.30} holds and $\psi>2^sR$. Set \(H = 2^{-s+1}\psi\) and \(k^0 = \max_{B_R}w - H\).
Then
\[
H > 2R > R \quad\text{and}\quad H<\psi\ (\text{as } s>2).
\]
Thus Lemma~\ref{Lemma 2.5} applies (its first alternative \eqref{2.14} is excluded), giving
\[
\max_{B_{\frac{R}{4}}}w \leq k^0+\frac{H}{2}
= \max_{B_R}w - H + \frac{H}{2}
= \max_{B_R}w-2^{-s}\psi,
\]
whence \eqref{2.40}.
\end{proof}

\begin{lemma}\label{Lemma 2.8}
For any \(R \in (0, R_0]\), at least one of the following two inequalities hol{\rm d}s:
\begin{equation}\label{2.41}
\operatorname{osc}_{B_R}u \leq \tau 2^s R,
\end{equation}
or
\begin{equation}\label{2.42}
\operatorname{osc}_{B_{\frac{R}{4}}}u \leq (1 - \tau^{-1} 2^{-s}) \operatorname{osc}_{B_R}u,
\end{equation}
where \(\tau = \max \left\{2, \frac{2}{\delta}\right\}\), \(s = s(g^-,n, p^{\pm}, q^{\pm}, \beta, \bar{a}, \ell(1), d_{\ell}, \kappa_{*}, \gamma, \gamma_1)\) is the constant as in Lemma \ref{Lemma 2.7}.
\end{lemma}
\begin{proof}
Apply Lemma~\ref{Lemma 2.7} with $w\in\{u,-u\}$ satisfying \eqref{2.8} and
\[
\psi=\tau^{-1}\operatorname{osc}_{B_R}u,
\qquad \tau=\max\Big\{2,\frac{2}{\delta}\Big\}.
\]
Inequality \eqref{2.40} becomes
\[
\tau^{-1}\operatorname{osc}_{B_R}u
\le 2^s\max\Big\{\max_{B_R}w-\max_{B_{\frac{R}{4}}}w,\;
R\Big\}.
\]
If the maximum is attained at the second entry, then
$\tau^{-1}\operatorname{osc}_{B_R}u\le 2^sR$, which gives \eqref{2.41} after multiplying by $\tau$.
Otherwise, 
\[
\max_{B_R}w-\max_{B_{\frac{R}{4}}}w\ge \tau^{-1}2^{-s}\operatorname{osc}_{B_R}u.
\]
Since $\operatorname{osc}w=\operatorname{osc}u$ and $\min_{B_{R/4}}w\ge\min_{B_R}w$,
$$\operatorname{osc}_{B_{\frac{R}{4}}}u\le \operatorname{osc}_{B_R}u-\left(\max_{B_R}w-\max_{B_{\frac{R}{4}}}w\right),$$
we obtain
\[
\operatorname{osc}_{B_{\frac{R}{4}}}u
\le \operatorname{osc}_{B_R}u-\tau^{-1}2^{-s}\operatorname{osc}_{B_R}u
=\big(1-\tau^{-1}2^{-s}\big)\operatorname{osc}_{B_R}u,
\]
which is \eqref{2.42}. The factor is positive because $\tau\ge 2$ and $s>2$.
\end{proof}
\emph{\textbf{Conclusion of the proof of Theorem~\ref{Theorem 2.1}:}}
We now return to the proof of Theorem~\ref{Theorem 2.1}.
Set
\[
\lambda:=1-\tau^{-1}2^{-s}\in(0,1)
\]
and
\[
\alpha:=\min\bigl\{1,-\log_4\lambda\bigr\}
=\min\bigl\{1,-\log_4(1-\tau^{-1}2^{-s})\bigr\},
\]
so that $\lambda\le4^{-\alpha}$. Let $y\in B_{R_0}(x_0)$ be an arbitrary centre and put $\vartheta_y(r):=\operatorname{osc}_{B_r(y)}u$, which is nondecreasing in $r$ and satisfies $\vartheta_y\le 2M$. For $r\in(0,R_0/4]$, Lemma~\ref{Lemma 2.8} applied at the radius $4r\le R_0$ yields either
\[
\vartheta_y(r)\le\vartheta_y(4r)\le 4\tau2^{s}\,r
\qquad\text{or}\qquad
\vartheta_y(r)\le\lambda\,\vartheta_y(4r).
\]
Hence the hypotheses of Lemma~\ref{lem:iter} hold with $\epsilon=1$, $c_1=4\tau 2^s$, $\theta=\lambda$ and $\rho_1=R_0$, and we obtain
\[
\operatorname{osc}_{B_r(y)}u \le C\,r^{\alpha},\qquad 0<r\le R_0,\quad C:=4^{\alpha}\max\{c_1R_0^{1-\alpha},2MR_0^{-\alpha}\},
\]
uniformly in $y\in B_{R_0}(x_0)$. Now let $x,y\in B_{\frac{R_0}{2}}(x_0)$ with $x\neq y$ (for the continuous representative of $u$, which exists by this very estimate at Lebesgue points). If $2|x-y|\le R_0$, then $x\in B_{2|x-y|}(y)$ and $|u(x)-u(y)|\le C(2|x-y|)^\alpha$; otherwise $$|u(x)-u(y)|\le 2M\le 2M(2|x-y|/R_0)^\alpha.$$ Consequently,
\[
u\in C^{0,\alpha}(B_{\frac{R_0}{2}}(x_0)),
\]
with
\[
\alpha=\min\bigl\{1,-\log_4(1-\tau^{-1}2^{-s})\bigr\}
=\alpha(g^-,n,p^\pm,q^\pm,\beta,\bar a,\ell(1),
d_\ell,\kappa_*,\gamma,\gamma_1),
\]
which is independent of $x_0$.

Since $x_0\in\Omega$ is arbitrary (and every $\Omega'\Subset\Omega$ is covered by finitely many balls $B_{\frac{R_0}{2}}(x_0)$), we conclude that $u\in C^{0,\alpha}(\Omega)$, in the sense of the Notations.
This completes the proof of Theorem~\ref{Theorem 2.1}.
\end{proof}


\begin{proof}[\textbf{Proof of Theorem \ref{Theorem 2.2}}]
  As the boundary Hölder regularity result, Theorem~\ref{Theorem 2.1}, has required more substantial modifications, we have highlighted all the changes, with particular emphasis on the boundary level selection and the exterior measure density estimate.\\ 

\medskip
Let $x_0\in\partial\Omega$. For $R>0$ and any $k\in\mathbb{R}$ , let  
\begin{equation*}
\Omega_R=B_R(x_0)\cap\Omega,
\quad
(\partial\Omega)_R=\partial\Omega\cap B_R(x_0), \quad A_{k,R}=\left\{x\in\Omega_R:w(x)>k\right\}.
\end{equation*}
be defined as in Definition \ref{Definition 2.3}, and write $\overline{u}:=u_{|_{\partial\Omega}}$ for the trace of $u$ (so $\overline{u}\in C^{0,\alpha_1}(\partial\Omega)$ by assumption), $[\overline{u}]_{\alpha_1}$ for its H\"older seminorm.\\
In the boundary case, the De Giorgi class is considered for levels satisfying
\begin{equation}\label{hlel5.1}
k\geq
\max\left\{
\max_{\Omega_R}w-\delta M,
\max_{(\partial\Omega)_R}w
\right\}.
\end{equation}
This condition is essential for the boundary argument. Indeed, if
\(
k\geq\max_{(\partial\Omega)_R}w,
\)
then the trace of $(w-k)_+$ vanishes on $(\partial\Omega)_R$, and hence no additional boundary superlevel term occurs. On the other hand, when the level is below the boundary supremum, the boundary values are controlled by the assumed Hölder regularity of the trace
\[
\overline{u}:=u_{|_{\partial\Omega}}\in C^{0,\alpha_1}(\partial\Omega).
\]

Before proving Theorem~\ref{Theorem 2.2}, we collect several auxiliary lemmas. First, we establish the boundary isoperimetric inequality (Lemma~\ref{lem:isoB}), which replaces the interior isoperimetric inequality (Lemma \ref{Lemma 2.2}) in the boundary setting. Then we prove the boundary analogues of Lemmas~\ref{Lemma 2.5} and~\ref{Lemma 2.6}, which drive the De Giorgi iteration near the boundary.
\begin{lemma}\label{lem:isoB}
Let $B_\rho$ be a ball centred on $\partial\Omega$ with $\rho\le m_{0}$, let $w\in W^{1,1}(\Omega_\rho)$, and let $k\ge\max_{(\partial\Omega)_{\rho}}w$ (in the sense of traces) and $l>k$. Then
\[
(l-k)|A_{l,\rho}|^{1-\frac{1}{n}}
\;\le\;
\frac{\beta}{\theta_{0}\omega_{n}}\int_{A_{k,\rho}\setminus A_{l,\rho}}|Dw|\,\mathrm{d}x.
\]
\end{lemma}

\begin{proof}
Let $\left\{\Omega_{j}'\right\}_{j}$ be the (at most countably many) connected components
of $\Omega_{\rho}$, and write
\[
A_{m}^{(j)}:=A_{m,\rho}\cap\Omega_{j}'.
\]
Fix $j$. We first extend $w$ to a function $\widetilde w$ defined on the whole
ball $B_{\rho}$ by setting
\begin{equation}\label{def:wtilde}
\widetilde w(x):=
\begin{cases}
w(x), & x\in\Omega_{j}',\\[2pt]
k, & x\in B_{\rho}\setminus\Omega_{j}'.
\end{cases}
\end{equation}
Then define
\begin{equation}\label{def:v}
v(x):=\max\{\widetilde w(x),\,k\},\qquad x\in B_{\rho}.
\end{equation}

\medskip

\noindent\textbf{Claim 1.} 
\begin{equation}\label{eq:videntity}
v \;=\; k+\widetilde{(w-k)_{+}} \qquad\text{on }B_{\rho}
\end{equation}
holds, where $\widetilde{(w-k)_{+}}$ denotes the function $(w-k)_{+}$ on
$\Omega_{j}'$, extended by $0$ to $B_{\rho}\setminus\Omega_{j}'$.

\medskip

\noindent\emph{Proof of Claim 1.}
We distinguish the two regions defining $\widetilde w$.

\smallskip

\emph{Case 1: $x\in\Omega_{j}'$.}
Then $\widetilde w(x)=w(x)$, and hence
\[
v(x)=\max\{w(x),\,k\}.
\]
On the other hand, since $(w-k)_{+}=\max\{w-k,0\}$ on $\Omega_{j}'$, we have
\[
k+\widetilde{(w-k)_{+}}(x)
=k+(w(x)-k)_{+}
=k+\max\{w(x)-k,0\}
=\max\{w(x),\,k\}.
\]
Thus $v(x)=k+\widetilde{(w-k)_{+}}(x)$ for $x\in\Omega_{j}'$.

\smallskip

\emph{Case 2: $x\in B_{\rho}\setminus\Omega_{j}'$.}
Then $\widetilde w(x)=k$ by \eqref{def:wtilde}, and hence
\[
v(x)=\max\{k,\,k\}=k.
\]
On the other hand, by definition of $\widetilde{(w-k)_{+}}$, we have
$\widetilde{(w-k)_{+}}(x)=0$ for $x\in B_{\rho}\setminus\Omega_{j}'$, so
\[
k+\widetilde{(w-k)_{+}}(x)=k+0=k.
\]
Thus $v(x)=k+\widetilde{(w-k)_{+}}(x)$ for
$x\in B_{\rho}\setminus\Omega_{j}'$.

\smallskip

Combining Cases 1 and 2, we obtain \eqref{eq:videntity}, proving Claim 1.

\medskip

\noindent\textbf{Claim 2.} The function $v$ belongs to
$W^{1,1}(B_{\rho})$, and its weak gradient is
\begin{equation}\label{eq:Dv}
Dv \;=\; Dw\,\chi_{\Omega_{j}'\cap\left\{w>k\right\}}.
\end{equation}

\medskip

\noindent\emph{Proof of Claim 2.}
Since $k\ge\max_{(\partial\Omega)_{\rho}}w$, the trace of $(w-k)_{+}$ vanishes on $(\partial\Omega)_{\rho}$. The boundary of $\Omega_j'$ relative to $B_\rho$ is contained in $\partial\Omega\cap B_\rho$ (a point of $\partial\Omega_j'\cap B_\rho$ lying in $\Omega$ would belong to $\Omega_\rho$ and hence to the same connected component). Therefore the extension of $(w-k)_{+}$ by zero outside $\Omega_j'$ belongs to
$W^{1,1}(B_{\rho})$, with weak gradient
\[
D\widetilde{(w-k)_{+}}
=
Dw\,\chi_{\Omega_{j}'\cap\{w>k\}}.
\]
Since the constant function $k$ is smooth on $B_{\rho}$ (hence in
$W^{1,1}(B_{\rho})$, with zero gradient), the identity \eqref{eq:videntity}
implies
\[
v \;=\; k+\widetilde{(w-k)_{+}} \;\in\; W^{1,1}(B_{\rho}),
\]
as a sum of two functions of $W^{1,1}(B_{\rho})$. Its weak gradient is
\[
Dv \;=\; D\bigl(k+\widetilde{(w-k)_{+}}\bigr)
\;=\; 0 + D\widetilde{(w-k)_{+}}
\;=\; Dw\,\chi_{\Omega_{j}'\cap\{w>k\}},
\]
which proves \eqref{eq:Dv}. This completes the proof of Claim 2.

\medskip

\noindent\textbf{Claim 3.} For every level $m\ge k$, the superlevel set
of $v$ in $B_{\rho}$ coincides with $A_{m}^{(j)}$, i.e.
\begin{equation}\label{eq:level}
\left\{x\in B_{\rho}:\,v(x)>m\right\} \;=\; A_{m}^{(j)}.
\end{equation}

\medskip

\noindent\emph{Proof of Claim 3.}
By \eqref{eq:videntity}, for every $x\in B_{\rho}$ and every $m\ge k$:
\[
v(x)>m
\iff
k+\widetilde{(w-k)_{+}}(x)>m
\iff
\widetilde{(w-k)_{+}}(x)>m-k.
\]
Since $m\ge k$, we have $m-k\ge 0$. We distinguish two cases.

\smallskip

\emph{Case A: $x\in B_{\rho}\setminus\Omega_{j}'$.}
Then $\widetilde{(w-k)_{+}}(x)=0$, and since $m-k\ge0$,
$\widetilde{(w-k)_{+}}(x)>m-k$ fails. So $x\notin\{v>m\}$ and
$x\notin A_{m}^{(j)}$ (since $A_{m}^{(j)}\subset\Omega_{j}'$). Both sets
exclude $x$, consistent with \eqref{eq:level}.

\smallskip

\emph{Case B: $x\in\Omega_{j}'$.}
Then $\widetilde{(w-k)_{+}}(x)=(w(x)-k)_{+}$, so
\[
\widetilde{(w-k)_{+}}(x)>m-k
\iff
(w(x)-k)_{+}>m-k.
\]
If $w(x)\le k$, then $(w(x)-k)_{+}=0$, and $0>m-k$ fails (as $m\ge k$). If
$w(x)>k$, then $(w(x)-k)_{+}=w(x)-k$, so
\[
w(x)-k>m-k
\iff
w(x)>m.
\]
Thus $\widetilde{(w-k)_{+}}(x)>m-k$ is equivalent to $w(x)>m$, i.e.
$x\in A_{m}^{(j)}$. This proves \eqref{eq:level}.\\

\medskip
We now apply Lemma~\ref{Lemma 2.2} to the function $v\in W^{1,1}(B_{\rho})$ on the ball $B_{\rho}$, with the levels $k<l$.\\
By Claim 3, respectively with $m=k$ and $m=l$; we get
$$\left\{x\in B_{\rho}:v(x)>k\right\}=A_{k}^{(j)}\quad \text{and}\quad \left\{x\in B_{\rho}:v(x)>l\right\}=A_{l}^{(j)}.$$
Moreover, by Claim 2, $Dv=Dw\,\chi_{A_{k}^{(j)}}$, and hence on the set
$\left\{k<v\le l\right\}=A_{k}^{(j)}\setminus A_{l}^{(j)}$, we have $$Dv=Dw.$$
Therefore, Lemma~\ref{Lemma 2.2} gives
\begin{equation}\label{eq:iso-local}
(l-k)\left|A_{l}^{(j)}\right|^{1-\frac{1}{n}}
\;\le\;
\frac{\beta\rho^{n}}{\left|B_{\rho}\setminus A_{k}^{(j)}\right|}
\int_{A_{k}^{(j)}\setminus A_{l}^{(j)}}|Dw|\,\mathrm{d}x.
\end{equation}

\medskip
Since $A_{k}^{(j)}\subset\Omega_{j}'$, we have
\[
B_{\rho}\setminus A_{k}^{(j)}
\;\supset\;
B_{\rho}\setminus\Omega_{j}'.
\]
By condition \eqref{AO} applied to the connected component $\Omega_{j}'$,
\[
\left|B_{\rho}\setminus\Omega_{j}'\right|
\;=\;
\left|B_{\rho}\right|-\left|\Omega_{j}'\right|
\;\ge\;
\left|B_{\rho}\right|-(1-\theta_{0})\left|B_{\rho}\right|
\;=\;
\theta_{0}\left|B_{\rho}\right|
\;=\;
\theta_{0}\omega_{n}\rho^{n}.
\]
Hence
\[
\left|B_{\rho}\setminus A_{k}^{(j)}\right|
\;\ge\;
\theta_{0}\omega_{n}\rho^{n}.
\]

\medskip
Substituting the lower bound into \eqref{eq:iso-local}
\[
(l-k)\left|A_{l}^{(j)}\right|^{1-\frac{1}{n}}
\;\le\;
\frac{\beta\rho^{n}}{\theta_{0}\omega_{n}\rho^{n}}
\int_{A_{k}^{(j)}\setminus A_{l}^{(j)}}|Dw|\,\mathrm{d}x
\;=\;
\frac{\beta}{\theta_{0}\omega_{n}}
\int_{A_{k}^{(j)}\setminus A_{l}^{(j)}}|Dw|\,\mathrm{d}x.
\]
\medskip
The above inequality holds for each connected component $\Omega_{j}'$. Summing
over $j$
\[
(l-k)\sum_{j}\left|A_{l}^{(j)}\right|^{1-\frac{1}{n}}
\;\le\;
\frac{\beta}{\theta_{0}\omega_{n}}
\sum_{j}\int_{A_{k}^{(j)}\setminus A_{l}^{(j)}}|Dw|\,\mathrm{d}x.
\]
The sets $A_{k}^{(j)}\setminus A_{l}^{(j)}$, $j$ varying, are pairwise
disjoint and their union is
$A_{k,\rho}\setminus A_{l,\rho}$ (because the $\Omega_{j}'$ partition
$\Omega_{\rho}$). Hence
\[
\sum_{j}\int_{A_{k}^{(j)}\setminus A_{l}^{(j)}}|Dw|\,\mathrm{d}x
\;=\;
\int_{A_{k,\rho}\setminus A_{l,\rho}}|Dw|\,\mathrm{d}x.
\]
For the left-hand side, we use the subadditivity of $s\mapsto s^{1-\frac{1}{n}}$ on
$[0,\infty)$ (which holds because $1-\frac{1}{n}\in(0,1)$):
\[
\sum_{j}a_{j}^{1-\frac{1}{n}}
\;\ge\;
\Bigl(\sum_{j}a_{j}\Bigr)^{1-\frac{1}{n}},
\qquad a_{j}\ge0.
\]
Applying this with $a_{j}:=\left|A_{l}^{(j)}\right|$ and using
$\sum_{j}\left|A_{l}^{(j)}\right|=\left|A_{l,\rho}\right|$, we obtain
\[
\sum_{j}\left|A_{l}^{(j)}\right|^{1-\frac{1}{n}}
\;\ge\;
\left|A_{l,\rho}\right|^{1-\frac{1}{n}}.
\]
Combining:
\[
(l-k)\left|A_{l,\rho}\right|^{1-\frac{1}{n}}
\;\le\;
\frac{\beta}{\theta_{0}\omega_{n}}
\int_{A_{k,\rho}\setminus A_{l,\rho}}|Dw|\,\mathrm{d}x,
\]
which is exactly the claim of Lemma~\ref{lem:isoB}.
\end{proof}
Recall that $\overline{u}:=u_{|_{\partial\Omega}}$ denotes the trace of $u$ on $\partial\Omega$, which is H\"older continuous with exponent $\alpha_1$ by assumption. For $R>0$ we denote by $\sigma:=\mathrm{osc}_{(\partial\Omega)_R}\,\overline{u}$ the oscillation of the boundary values on $(\partial\Omega)_R$. The boundary class requires levels $k\ge\max_{(\partial\Omega)_R}w$ (see \eqref{eqlevel}), so that the trace of $(w-k)_+$ vanishes on $(\partial\Omega)_R$ and no boundary term appears. The next lemma shows that, when the boundary values are almost flat compared to the oscillation of $u$, that is $\sigma<\frac{\vartheta}{4}$, we can choose the sign $w=\pm u$ and a level $k'$ close to $\max_{\Omega_R}w$ (at distance $\psi$, a fixed fraction of $\vartheta$) which lies above the boundary values of $w$. Hence all levels $k\ge k'$ are admissible for the boundary class. In the opposite case $\sigma\ge\frac{\vartheta}{4}$, the oscillation of $u$ is itself controlled by the H\"older continuity of $\overline{u}$, namely $$\vartheta\le4\sigma\le8\left[\overline{u}\right]_{\alpha_1}R^{\alpha_1};$$ this dichotomy is used in Lemma~\ref{lem:balt}.

\begin{lemma}\label{lem:signB}
Let $R\le R_{0}$, $\vartheta:=\mathrm{osc}_{\Omega_{R}}u$,
$\sigma:=\mathrm{osc}_{(\partial\Omega)_R}\overline{u}$, and assume $\sigma<\frac{\vartheta}{4}$.
Put $\tau_{b}:=\max\left\{4,\frac{2}{\delta}\right\}$ and $\psi:=\frac{\vartheta}{\tau_{b}}$.
Then there is $w\in\{u,-u\}$ such that, with $\mu_{w}=\max_{\Omega_{R}}w$
and $k':=\mu_{w}-\psi$,
\[
k'\ge\max_{(\partial\Omega)_{R}}w,\qquad
k'\ge\mu_{w}-\delta M,\qquad\psi\le M.
\]
Consequently, for every $\rho\le R$ and every $k\ge k'$,
$$k\ge\max\left\{\max_{\Omega_{\rho}}w-\delta M,\max_{(\partial\Omega)_{\rho}}w\right\},$$
so \eqref{2.4} holds at all such levels, and all levels $k\ge k'$ are admissible
in Lemma~\ref{lem:isoB}.
\end{lemma}

\begin{proof}
Let $\mu:=\max_{\Omega_{R}}u$, $m:=\min_{\Omega_{R}}u$, so that
$\vartheta=\mu-m$, and let
$b^{+}:=\max_{(\partial\Omega)_{R}}\overline{u}$,
$b^{-}:=\min_{(\partial\Omega)_{R}}\overline{u}$, so that
$b^{+}-b^{-}=\sigma<\frac{\vartheta}{4}$ and, since $u$ takes its values in $[m,\mu]$ on $\Omega_R$ and so does its trace,
\[
m\;\le\;b^{-}\;\le\;b^{+}\;\le\;\mu.
\]
We distinguish two cases.

\medskip

\noindent\emph{Case 1:} $b^{+}\le\mu-\frac{\vartheta}{4}$.
Take $w=u$. Then $\mu_{w}=\mu$ and
\[
\max_{(\partial\Omega)_{R}}w=b^{+}\le\mu-\frac{\vartheta}{4}=\mu_{w}-\frac{\vartheta}{4}.
\]

\medskip

\noindent\emph{Case 2:} $b^{+}>\mu-\frac{\vartheta}{4}$.
From $\sigma=b^{+}-b^{-}<\frac{\vartheta}{4}$ we deduce
\[
b^{-}=b^{+}-\sigma>\Bigl(\mu-\frac{\vartheta}{4}\Bigr)-\frac{\vartheta}{4}
=\mu-\frac{\vartheta}{2}.
\]
Take $w=-u$. Then $\mu_{w}=-m$ and $\max_{(\partial\Omega)_{R}}w=-b^{-}$.
Hence
\[
\max_{(\partial\Omega)_{R}}w=-b^{-}<-\mu+\frac{\vartheta}{2}.
\]
We claim that $-\mu+\frac{\vartheta}{2}=\mu_{w}-\frac{\vartheta}{2}$. Indeed, using
$\vartheta=\mu-m$ and $\mu_{w}=-m$,
\[
-\mu+\frac{\vartheta}{2}=-\mu+\frac{\mu-m}{2}=\frac{-\mu-m}{2}
=\frac{-m-\mu}{2}=-m-\frac{\vartheta}{2}=\mu_{w}-\frac{\vartheta}{2}.
\]
Therefore
\[
\max_{(\partial\Omega)_{R}}w<\mu_{w}-\frac{\vartheta}{2}\le\mu_{w}-\frac{\vartheta}{4}.
\]
\medskip
In both cases, there is $w\in\{u,-u\}$ such that
\begin{equation}\label{eq:maxbd}
\max_{(\partial\Omega)_{R}}w\le\mu_{w}-\frac{\vartheta}{4}.
\end{equation}
Since $\psi=\frac{\vartheta}{\tau_{b}}$ with $\tau_{b}\ge4$, we have $\psi\le\frac{\vartheta}{4}$.
Combining with \eqref{eq:maxbd},
\begin{equation}\label{eq:maxbd2}
\max_{(\partial\Omega)_{R}}w\le\mu_{w}-\frac{\vartheta}{4}\le\mu_{w}-\psi=k'.
\end{equation}
Moreover, $\vartheta\le2M$, hence
\[
\psi=\frac{\vartheta}{\tau_{b}}\le\frac{2M}{\tau_{b}}.
\]
Since $\tau_{b}\ge \frac{2}{\delta}$, we obtain $\psi\le\delta M$, so
\[
k'=\mu_{w}-\psi\ge\mu_{w}-\delta M.
\]
And since $\tau_{b}\ge4$, we obtain $\psi\le \frac{M}{2}\le M$.\\
\medskip
Finally, let $\rho\le R$ and $k\ge k'$. Since $\Omega_{\rho}\subset\Omega_{R}$,
\[
\max_{\Omega_{\rho}}w\le\max_{\Omega_{R}}w=\mu_{w},
\]
so
\[
\max_{\Omega_{\rho}}w-\delta M\le\mu_{w}-\delta M\le k'\le k.
\]
And since $(\partial\Omega)_{\rho}\subset(\partial\Omega)_{R}$,
\[
\max_{(\partial\Omega)_{\rho}}w\le\max_{(\partial\Omega)_{R}}w\le k'\le k.
\]
Therefore
\[
k\;\ge\;\max\Bigl\{\max_{\Omega_{\rho}}w-\delta M,\;\max_{(\partial\Omega)_{\rho}}w\Bigr\},
\]
which means all levels $k\ge k'$ are admissible in \eqref{2.4} and in
Lemma~\ref{lem:isoB}.
\end{proof}

Next, we see the boundary versions of Lemmas~\ref{lem:boundary38} and~\ref{lem:boundary39}.\\

Let $x_{0}\in\partial\Omega$
and $R\le R_{0}$ (so that $R<R_{*}$ and $R\le m_{0}$), and $w,\psi,k'$ are as
in Lemma~\ref{lem:signB} (in particular $0<\psi\le M$).
\begin{lemma}\label{lem:boundary38}
There exists
\(
\theta_b=\theta_b(n,\beta,\theta_0,\mu_*,g^-,d_\ell,
\kappa,\gamma,\gamma_1)\in(0,1)
\)
such that, whenever \(R<H\le\psi\), setting \(k_0:=\mu_w-H\), the condition
\[
\bigl|A_{k_0,\frac R2}\bigr|\le\theta_b R^n
\]
implies
\[
\max_{\Omega_{\frac R4}}w\le k_0+\frac H2.
\]
\end{lemma}

\begin{proof}
Put $t_{0}:=\frac{H}{R}\in\left(1,\frac{M}{R}\right]$, and let $\Phi,f$ be as in
 Lemma~\ref{lem22}. For $h=0,1,2,\ldots$, define
\[
\rho_{h}:=\frac{R}{4}+\frac{R}{2^{h+2}},\quad
k_{h}:=k_{0}+\frac{H}{2}-\frac{H}{2^{h+1}},
\]
\[
Y_{h}:=R^{-n}\left|A_{k_{h},\rho_{h}}\right|,\quad
D_{h+1}:=A_{k_{h},\rho_{h+1}}\setminus A_{k_{h+1},\rho_{h+1}}.
\]
Then $\rho_{h}\downarrow \frac{R}{4}$, $\rho_{0}=\frac{R}{2}$,
$\rho_{h}-\rho_{h+1}=R2^{-(h+3)}$, $k_{h}\uparrow k_{0}+\frac{H}{2}$,
$k_{h+1}-k_{h}=H2^{-(h+2)}$, and
$k_{h}\ge k_{0}=\mu_{w}-H\ge\mu_{w}-\psi=k'$, so \eqref{2.4} is available at
the levels $k_{h}$ for balls $B_{\rho_{h}}$ centred at $x_{0}$.

\medskip
On $A_{k_{h},\rho_{h}}\subset\Omega_{R}$,
$0<w-k_{h}\le\mu_{w}-k_{h}=H\left(\tfrac{1}{2}+2^{-h-1}\right)\le H$. Hence
\[
\frac{w-k_{h}}{\rho_{h}-\rho_{h+1}}\le 2^{h+3}t_{0}.
\]
By monotonicity of $G(x,\cdot)$ and  Lemma~\ref{lem22}(4)
with $\lambda=2^{h+3}$, $t=t_{0}$, and then \eqref{2.4} with $\rho=\rho_{h}$,
$\rho-\sigma\rho=\rho_{h+1}$,
\begin{equation}\label{eqbh1}
\int_{A_{k_{h},\rho_{h+1}}}G(x,|Dw|)\,\mathrm{d}x
\le
\left[\gamma\kappa\left(2^{h+3}\right)^{\mu_{*}}\Phi(t_{0})+\gamma_{1}\right]R^{n}Y_{h}
=:E_{h}.
\end{equation}

\medskip
By Lemma~\ref{lem:conv}, $t\le\lambda+\frac{\lambda}{g^{-}G(x,\lambda)}G(x,t)$
for $t\ge0$, $\lambda>0$. Take $\lambda=t_{0}$, $t=|Dw(x)|$, and use
$G(x,t_{0})\ge\Phi(t_{0})\ge d_{\ell}$. Integrating over
$D_{h+1}\subset A_{k_{h},\rho_{h+1}}$, with $|D_{h+1}|\le R^{n}Y_{h}$ and \eqref{eqbh1},
\begin{equation}\label{eq1bb}
\int_{D_{h+1}}|Dw|\,\mathrm{d}x
\le
t_{0}R^{n}Y_{h}\Bigl[1+\frac{1}{g^{-}}\Bigl(\gamma\kappa\left(2^{h+3}\right)^{\mu_{*}}
+\frac{\gamma_{1}}{d_{\ell}}\Bigr)\Bigr]
\le
K2^{h\mu_{*}}t_{0}R^{n}Y_{h},
\end{equation}
with $K:=1+\frac{1}{g^{-}}\left(\gamma\kappa 2^{3\mu_{*}}+\frac{\gamma_{1}}{d_{\ell}}\right)$
(using $2^{h\mu_*}\ge1$).

\medskip
Lemma~\ref{lem:isoB} with $k=k_{h}$, $l=k_{h+1}$, $\rho=\rho_{h+1}\le m_{0}$ gives
\[
H2^{-(h+2)}R^{n-1}Y_{h+1}^{1-\frac{1}{n}}
\le
\frac{\beta}{\theta_{0}\omega_{n}}\int_{D_{h+1}}|Dw|\,\mathrm{d}x,
\]
i.e.
\begin{equation}\label{eq2bb}
Y_{h+1}^{1-\frac{1}{n}}
\le
\frac{4\beta}{\theta_{0}\omega_{n}}\frac{2^{h}}{HR^{n-1}}
\int_{D_{h+1}}|Dw|\,\mathrm{d}x.
\end{equation}

\medskip
Combining \eqref{eq1bb}--\eqref{eq2bb} and $t_{0}\frac{R}{H}=1$,
\[
Y_{h+1}^{1-\frac{1}{n}}\le\frac{4\beta K}{\theta_{0}\omega_{n}}
2^{h(\mu_{*}+1)}Y_{h},
\]
hence $Y_{h+1}\le cb^{h}Y_{h}^{1+\epsilon}$ with
\[
\epsilon=\frac{1}{n-1},\quad
b=\left(2^{\mu_{*}+1}\right)^{\frac{n}{n-1}},\quad
c=\left(\frac{4\beta K}{\theta_{0}\omega_{n}}\right)^{\frac{n}{n-1}}.
\]
Let $\theta_{b}:=\min\left\{\tfrac{1}{2},c^{-\frac{1}{\epsilon}}b^{-\frac{1}{\epsilon^{2}}}\right\}$.
Since $Y_{0}=R^{-n}\left|A_{k_{0},\frac{R}{2}}\right|\le\theta_{b}$, Lemma~\ref{Lemma 2.1}
gives $Y_{h}\to0$, so
\[
\max_{\Omega_{\frac{R}{4}}}w\le\lim_{h}k_{h}=k_{0}+\frac{H}{2}.
\]
\end{proof}

\begin{lemma}\label{lem:boundary39}
There is an integer
$s=s(n,\beta,\theta_{0},p^{-},q^{+},d_{\ell},\bar{a},\ell(1),\kappa,\gamma,\gamma_{1},\theta_{b})>2$
such that either $$\psi\le2^{s}R,$$ or $$\left|A_{k_{0},\frac{R}{2}}\right|\le\theta_{b}R^{n}$$ with
$k_{0}:=\mu_{w}-2^{-s+1}\psi$.
\end{lemma}

\begin{proof}
Assume $\psi>2^{s}R$. Let $\Phi,f$ be as in
 Lemma~\ref{lem22}, and for $t=0,\ldots,s-1$, let
\[
k_{t}:=\mu_{w}-2^{-t}\psi,\quad
D_{t}:=A_{k_{t},\frac{R}{2}}\setminus A_{k_{t+1},\frac{R}{2}},\quad
\tau_{t}:=\frac{2^{-t+1}\psi}{R}.
\]
For $t\le s-2$: $8<\tau_{t}\le2\frac{M}{R}$, and $k_{t}\ge k'$ (so \eqref{2.4}
and Lemma~\ref{lem:isoB} apply).

\medskip
On $A_{k_{t},R}\subset\Omega_{R}$, $0<w-k_{t}\le2^{-t}\psi$, so
\[
\frac{2(w-k_{t})}{R}\le\tau_{t}
\]
and using Lemma~\ref{lem22}(4), $\lambda=1$, we get
$$G\bigl(x,\frac{2(w-k_{t})}{R}\bigr)\le G(x,\tau_{t})\le\kappa\Phi(\tau_{t}).$$
By \eqref{2.4} with $\rho=R$, $\sigma\rho=\frac{R}{2}$, and
 Lemma~\ref{lem22}(1) integrated over $A_{k_{t},\frac{R}{2}}$,
\[
\int_{A_{k_{t},\frac{R}{2}}}\Phi(|Dw|)\,\mathrm{d}x
\le
\bigl[\gamma\kappa\Phi(\tau_{t})+\gamma_{1}+c_{\mathrm{sm}}\bigr]\left|A_{k_{t},R}\right|.
\]

\medskip
 Lemma~\ref{lem22}(2) with $\lambda=\tau_{t}$ and (3) give,
with $\left|A_{k_{t},R}\right|\le\omega_{n}R^{n}$,
\[
\int_{A_{k_{t},R/2}}|Dw|^{f}\,\mathrm{d}x\le C_{1}\tau_{t}^{f}R^{n},\qquad
C_{1}:=\omega_{n}\Bigl[1+\gamma\kappa
+\frac{\gamma_{1}+c_{\mathrm{sm}}}{d_{\ell}}\Bigr].
\]

\medskip
Using H\"older's inequality, we get
\[
\int_{D_{t}}|Dw|\,\mathrm{d}x\le C_{1}^{\frac{1}{f}}R^{\frac{n}{f}}\tau_{t}|D_{t}|^{1-\frac{1}{f}}.
\]

\medskip
Lemma~\ref{lem:isoB} with $k=k_{t}$, $l=k_{t+1}$, $\rho=\frac{R}{2}\le m_{0}$,
together with $A_{k_{s-1},\frac{R}{2}}\subset A_{k_{t+1},\frac{R}{2}}$ and
$k_{t+1}-k_{t}=2^{-t-1}\psi$, $\tau_{t}=4\cdot2^{-t-1}\frac{\psi}{R}$,
\[
2^{-t-1}\psi\left|A_{k_{s-1},\frac{R}{2}}\right|^{1-\frac{1}{n}}
\le
\frac{\beta}{\theta_{0}\omega_{n}}C_{1}^{\frac{1}{f}}R^{\frac{n}{f}}\tau_{t}|D_{t}|^{1-\frac{1}{f}},
\]
which gives
\[
\left|A_{k_{s-1},\frac{R}{2}}\right|^{\frac{(n-1)f}{n(f-1)}}
\le C_{2}^{\frac{f}{f-1}}R^{\frac{n-f}{f-1}}|D_{t}|,
\]
with $C_{2}:=4\beta\theta_{0}^{-1}\omega_{n}^{-1}C_{1}^{\frac{1}{f}}$, for
$t=0,\ldots,s-2$.

\medskip
The left side does not depend on $t$. Summing over $t=0,\ldots,s-2$, using
$$\sum_{t}|D_{t}|\le\left|\Omega_{\frac{R}{2}}\right|\le\omega_{n}2^{-n}R^{n}$$ and
$$\frac{n-f}{f-1}+n=\frac{f(n-1)}{f-1},$$
we get
\[
\left|A_{k_{s-1},\frac{R}{2}}\right|
\le
\Bigl(\frac{C_{3}}{s-1}\Bigr)^{\epsilon(f)}R^{n},\quad
C_{3}:=C_{2}^{\frac{f}{f-1}}\omega_{n}2^{-n},\quad
\epsilon(r):=\frac{n(r-1)}{(n-1)r}.
\]
For $f\in[p^{-},q^{+}]$: $C_{1}^{\frac{1}{f}}\le\max\{1,C_{1}\}$ and
$r\mapsto \frac{r}{r-1}$ is decreasing, so
\[
C_{3}\le\bar{C}_{3}:=
\Bigl(\max\left\{1,4\beta\theta_{0}^{-1}\omega_{n}^{-1}\max\{1,C_{1}\}\right\}\Bigr)^{\frac{p^{-}}{p^{-}-1}}
\omega_{n}2^{-n}.
\]
Since $\epsilon$ is increasing, $s-1\ge\bar{C}_{3}$ gives
$$\left(\frac{C_{3}}{s-1}\right)^{\epsilon(f)}\le\left(\frac{\bar{C}_{3}}{s-1}\right)^{\epsilon(p^{-})}.$$
Choose $s>2$ with $$s-1\ge\bar{C}_{3}\theta_{b}^{-\frac{1}{\epsilon(p^{-})}}.$$
Then $\left|A_{k_{s-1},\frac{R}{2}}\right|\le\theta_{b}R^{n}$ and
$k_{s-1}=\mu_{w}-2^{-s+1}\psi=k_{0}$.
\end{proof}

We now establish the boundary and interior alternative and boundary estimate lemmas.\\ 

Fix once and for all
\[
R_{0}:=\frac{1}{2}\min\{R_{*},m_{0},1\}
\qquad\text{(so }R_{0}<R_{*}\text{)}
\]
and set $\epsilon_{0}:=\min\{1,\alpha_{1}\}=\alpha_1$.

\begin{lemma}\label{lem:balt}
Let $x_{0}\in\partial\Omega$, $R\le R_{0}$,
$\vartheta(r):=\mathrm{osc}_{\Omega_{r}}u$.
There are $s_{b}$ (from Lemma~\ref{lem:boundary39}) and
$c_{1}=c_{1}(\alpha_{1},[\overline{u}]_{\alpha_{1}},\delta,s_{b})$ such that either
\[
\vartheta(R)\le c_{1}R^{\epsilon_{0}}
\qquad\text{or}\qquad
\vartheta\left(\frac{R}{4}\right)\le\theta_{1}\vartheta(R),\quad
\theta_{1}:=1-\tau_{b}^{-1}2^{-s_{b}}.
\]
\end{lemma}

\begin{proof}
Let $\vartheta=\vartheta(R)$ and
$\sigma:=\mathrm{osc}_{(\partial\Omega)_R}\overline{u}\le[\overline{u}]_{\alpha_{1}}(2R)^{\alpha_{1}}$,
where the last inequality follows from the $\alpha_{1}$-H\"older continuity of
$\overline{u}$ on $\partial\Omega$ and from $|x-y|\le 2R$ for $x,y\in(\partial\Omega)_{R}$.

\medskip

\noindent\emph{Case (a): $\sigma\ge\frac{\vartheta}{4}$.}
Then
\[
\vartheta\le4\sigma\le4[\overline{u}]_{\alpha_{1}}(2R)^{\alpha_{1}}
=4\cdot 2^{\alpha_{1}}[\overline{u}]_{\alpha_{1}}R^{\alpha_{1}}
\le8[\overline{u}]_{\alpha_{1}}R^{\epsilon_{0}},
\]
using $2^{\alpha_{1}}\le2$ (as $\alpha_{1}\le1$) and $R^{\alpha_{1}}\le R^{\epsilon_{0}}$
(as $R\le1$ and $\epsilon_{0}\le\alpha_{1}$).

\medskip

\noindent\emph{Case (b): $\sigma<\frac{\vartheta}{4}$.}
Take $w,\psi,k'$ from Lemma~\ref{lem:signB}; in particular
$\psi=\frac{\vartheta}{\tau_{b}}$ and $\psi\le M$.
Apply Lemma~\ref{lem:boundary39} (with the $\theta_{b}$ of
Lemma~\ref{lem:boundary38}).
If $\psi\le2^{s_{b}}R$, then
\[
\vartheta=\tau_{b}\psi\le\tau_{b}2^{s_{b}}R\le\tau_{b}2^{s_{b}}R^{\epsilon_{0}},
\]
using $R\le R^{\epsilon_{0}}$ (as $R\le1$ and $\epsilon_{0}\le1$).
Otherwise $\psi>2^{s_{b}}R$, and Lemma~\ref{lem:boundary39} gives
$$\left|A_{k_{0},\frac{R}{2}}\right|\le\theta_{b}R^{n}$$
with
$k_{0}=\mu_{w}-2^{-s_{b}+1}\psi$.\\
Put $H:=2^{-s_{b}+1}\psi$; then
\[
H=2^{-s_{b}+1}\psi>2^{-s_{b}+1}\cdot 2^{s_{b}}R=2R>R
\]
and
\[
H=2^{-s_{b}+1}\psi<\psi
\]
(as $s_{b}>2$). Thus Lemma~\ref{lem:boundary38} applies with this $H$ and
gives
\[
\max_{\Omega_{\frac{R}{4}}}w\le k_{0}+\frac{H}{2}
=\mu_{w}-H+\frac{H}{2}
=\mu_{w}-2^{-s_{b}}\psi.
\]
Since $\min_{\Omega_{\frac{R}{4}}}w\ge\min_{\Omega_{R}}w$ (as $\Omega_{\frac{R}{4}}\subset\Omega_{R}$)
and $\mathrm{osc}\,w=\mathrm{osc}\,u$ (as $w=\pm u$),
\begin{align*}
\vartheta\left(\frac{R}{4}\right)
&=\mathrm{osc}_{\Omega_{\frac{R}{4}}}u
=\mathrm{osc}_{\Omega_{\frac{R}{4}}}w\\
&\le\max_{\Omega_{\frac{R}{4}}}w-\min_{\Omega_{R}}w\\
&\le(\mu_{w}-2^{-s_{b}}\psi)-\min_{\Omega_{R}}w\\
&=\vartheta-2^{-s_{b}}\psi
=\vartheta\bigl(1-\tau_{b}^{-1}2^{-s_{b}}\bigr)
=\theta_{1}\vartheta,
\end{align*}
where we used $\mu_{w}-\min_{\Omega_{R}}w=\vartheta$ and $\psi=\frac{\vartheta}{\tau_{b}}$.

\medskip

Take $c_{1}:=\max\{8[\overline{u}]_{\alpha_{1}},\tau_{b}2^{s_{b}}\}$. This completes
the proof of the lemma.
\end{proof}

\begin{lemma}\label{prop:boundaryest}
There are $C_{b},\alpha_{b}\in(0,1]$, depending only on the data and on
$\theta_{0},\alpha_{1},[\overline{u}]_{\alpha_{1}},\delta,M$, such that for every
$x_{0}\in\partial\Omega$ and $r\le R_{0}$,
\[
\mathrm{osc}_{\Omega\cap B_{r}(x_{0})}u\le C_{b}r^{\alpha_{b}}.
\]
\end{lemma}

\begin{proof}
Apply Lemma~\ref{lem:iter} to $\vartheta(r)=\mathrm{osc}_{\Omega_{r}}u$
(nondecreasing in $r$) on $(0,\rho_{1}]$, $\rho_{1}:=R_{0}$, with
\[
\epsilon:=\epsilon_{0},\qquad \theta:=\theta_{1},\qquad
\alpha_{b}:=\min\{\epsilon_{0},-\log_{4}\theta_{1}\}.
\]
Since $\theta_{1}\in(0,1)$, we have $-\log_{4}\theta_{1}>0$; hence
$\alpha_{b}\in(0,\epsilon_{0}]$ and
\[
\theta_{1}\le4^{-\alpha_{b}},
\]
as required by Lemma~\ref{lem:iter}.

\smallskip
\noindent Verification of the hypothesis of Lemma~\ref{lem:iter}.
Let $r\in\left(0,\frac{R_{0}}{4}\right]$, so $4r\le R_{0}$. Lemma~\ref{lem:balt} applied at
the radius $4r$ gives either
\[
\vartheta(4r)\le c_{1}(4r)^{\epsilon_{0}}
\qquad\text{or}\qquad
\vartheta(r)=\vartheta\left(\frac{(4r)}{4}\right)\le\theta_{1}\vartheta(4r).
\]
In the first case, by monotonicity $\vartheta(r)\le\vartheta(4r)$, hence
\[
\vartheta(r)\le c_{1}(4r)^{\epsilon_{0}}
=4^{\epsilon_{0}}c_{1}r^{\epsilon_{0}}.
\]
In the second case, the second alternative of Lemma~\ref{lem:iter} holds
directly. Thus the hypothesis of Lemma~\ref{lem:iter} is satisfied with
\[
\epsilon=\epsilon_{0},\quad
c_{1}':=4^{\epsilon_{0}}c_{1},\quad
\theta=\theta_{1}.
\]

\smallskip
\noindent Application of Lemma~\ref{lem:iter}.
Lemma~\ref{lem:iter} yields, for all $r\in(0,R_{0}]$,
\[
\vartheta(r)\le4^{\alpha_{b}}\max\left\{c_{1}'R_{0}^{\epsilon_{0}-\alpha_{b}},
\vartheta(R_{0})R_{0}^{-\alpha_{b}}\right\}r^{\alpha_{b}}.
\]
Since $\vartheta(R_{0})\le2M$, we obtain
\[
\vartheta(r)\le C_{b}r^{\alpha_{b}},
\]
with
\[
C_{b}:=4^{\alpha_{b}}\max\left\{4^{\epsilon_{0}}c_{1},2MR_{0}^{-\alpha_{b}}\right\}.
\]
This completes the proof.
\end{proof}

\begin{lemma}\label{prop:intalt}
There are $s_{i}$ and $\tau=\max\left\{2,\frac{2}{\delta}\right\}$ such that for every ball
$B_{\rho}(y)\subset\Omega$ with $\rho\le R_{0}$, either
\[
\mathrm{osc}_{B_{\rho}(y)}u\le\tau2^{s_{i}}\rho
\]
or
\[
\mathrm{osc}_{B_{\frac{\rho}{4}}(y)}u\le\left(1-\tau^{-1}2^{-s_{i}}\right)\mathrm{osc}_{B_{\rho}(y)}u.
\]
\end{lemma}

\begin{proof}
This is Lemma~\ref{Lemma 2.8} (whose proof, based on Lemmas~\ref{Lemma 2.5}--\ref{Lemma 2.7}, only uses that $B_\rho(y)\subset\Omega$, $\rho<R_*$ and the class inequality \eqref{2.4} on $B_\rho(y)$), with $s_i:=s$. For the reader's convenience we recall the argument. Let $\vartheta:=\mathrm{osc}_{B_{\rho}(y)}u$ and let $w\in\{u,-u\}$ be the
sign satisfying the normalization \eqref{2.8} at scale $\rho$. Put
$\psi:=\tau^{-1}\vartheta$, so that $\psi\le\delta M$ and $\psi\le M$. Apply Lemma~\ref{Lemma 2.6} with the constant $\theta$ of Lemma~\ref{Lemma 2.5}.

\smallskip
\noindent\emph{Case 1: $\psi\le2^{s_{i}}\rho$.}
Then
$\mathrm{osc}_{B_{\rho}(y)}u=\vartheta=\tau\psi\le\tau2^{s_{i}}\rho$,
which is the first alternative.

\smallskip
\noindent\emph{Case 2: $\psi>2^{s_{i}}\rho$.}
Then $|A_{k_{0},\frac{\rho}{2}}|\le\theta\rho^{n}$ with $k_{0}:=\max_{B_{\rho}}w-H$, $H:=2^{-s_{i}+1}\psi$. Moreover $H<\psi$ (since $s_{i}>2$) and $H>2\rho>\rho$. Thus Lemma~\ref{Lemma 2.5} applies and gives
\[
\max_{B_{\frac{\rho}{4}}}w\le k_{0}+\frac{H}{2}
=\max_{B_{\rho}}w-2^{-s_{i}}\psi.
\]
Since $\min_{B_{\frac{\rho}{4}}}w\ge\min_{B_{\rho}}w$
and $\mathrm{osc}\,w=\mathrm{osc}\,u$, we obtain
\[
\mathrm{osc}_{B_{\frac{\rho}{4}}(y)}u
\le\vartheta-2^{-s_{i}}\psi
=\bigl(1-\tau^{-1}2^{-s_{i}}\bigr)\mathrm{osc}_{B_{\rho}(y)}u.
\]
This is the second alternative and completes the proof.
\end{proof}
\emph{\textbf{Conclusion of the proof of Theorem~\ref{Theorem 2.2}:}}
Let $\theta_{i}:=1-\tau^{-1}2^{-s_{i}}$, $\theta_{*}:=\max\{\theta_{1},\theta_{i}\}\in(0,1)$ and
\[
\alpha:=\min\left\{\epsilon_{0},-\log_{4}\theta_{*}\right\}.
\]
We prove that
\[
|u(x)-u(y)|\le C|x-y|^{\alpha}
\qquad\text{for all }x,y\in\overline{\Omega},
\]
where, on $\partial\Omega$, we consider the continuous extension of $u$,
whose existence follows from Lemma~\ref{prop:boundaryest}.

By our choice of $\alpha$, we have $\alpha\le\epsilon_{0}$ and
$\alpha\le\alpha_b$. Moreover,
\[
\alpha\le-\log_{4}\theta_{*},
\qquad\text{and hence}\qquad
\theta_{*}\le4^{-\alpha}.
\]
Note that $\theta_1$ and $\theta_i$ are in general different, since $\tau\ne\tau_b$ and $s_i\ne s_b$; this is why the maximum $\theta_*$ is needed.

Let $r:=|x-y|$. If $r\ge \frac{R_{0}}{4}$, then
\[
|u(x)-u(y)|
\le 2M
\le 2M\left(\frac{4}{R_{0}}\right)^{\alpha}r^{\alpha}.
\]
We may therefore assume that $r<\frac{R_{0}}{4}$. Without loss of generality,
suppose that
\[
d:=\operatorname{dist}(x,\partial\Omega)
\ge \operatorname{dist}(y,\partial\Omega).
\]

\medskip
\noindent\emph{Case 1: $r>\frac{d}{2}$.}
Choose $z\in\partial\Omega$ such that $|x-z|=d<2r$. By the triangle
inequality,
\[
|y-z|\le |y-x|+|x-z|<3r<R_{0}.
\]
Consequently, $x,y\in\Omega\cap B_{3r}(z)$, and
Lemma~\ref{prop:boundaryest} yields
\[
\begin{aligned}
|u(x)-u(y)|
&\le \operatorname{osc}_{\Omega\cap B_{3r}(z)}u\le C_b(3r)^{\alpha_b}
\le 3C_b r^\alpha,
\end{aligned}
\]
where we used $\alpha\le\alpha_b$ and $r<1$.

\medskip
\noindent\emph{Case 2: $r\le \frac{d}{2}$.}
Set
\[
\rho_1:=\min\{d,R_0\},
\qquad
\vartheta_x(\rho):=\operatorname{osc}_{B_\rho(x)}u.
\]
Then $B_{\rho_1}(x)\subset\Omega$ and $2r\le\rho_1$ (because $r\le\frac{d}{2}$ and
$r<\frac{R_0}{4}$); in particular $y\in B_{2r}(x)\subset B_{\rho_1}(x)$.

We first verify that the hypotheses of Lemma~\ref{lem:iter} hold on
$(0,\rho_1]$ with $\epsilon=\epsilon_0$ and $\theta=\theta_*$. Let
$r'\in\bigl(0,\frac{\rho_1}{4}\bigr]$. By Lemma~\ref{prop:intalt}, applied at radius
$4r'\le\rho_1$, either the first alternative holds, and then
\[
\vartheta_x(r')
\le \vartheta_x(4r')
\le 4\tau 2^{s_i}r'
\le c_1'(r')^{\epsilon_0},
\qquad
c_1':=4\tau 2^{s_i},
\]
because $r'\le1$ and $\epsilon_0\le1$; or the second alternative holds, and then
\[
\vartheta_x(r')\le\theta_i\,\vartheta_x(4r')\le\theta_*\,\vartheta_x(4r').
\]
Thus Lemma~\ref{lem:iter} applies with $c_1'$ as above and
$\theta=\theta_*\le4^{-\alpha}$.

It remains to control the initial oscillation. We distinguish two
subcases.

If $d\le \frac{R_{0}}{2}$, then $\rho_1=d$. Choose $z\in\partial\Omega$ such that
$|x-z|=d$. Since
\[
B_d(x)\subset\Omega\cap B_{2d}(z),
\]
Lemma~\ref{prop:boundaryest} implies
\[
\vartheta_x(\rho_1)\le C_b(2d)^{\alpha_b}.
\]
Therefore,
\[
\vartheta_x(\rho_1)\rho_1^{-\alpha}
\le 2^{\alpha_b}C_b d^{\alpha_b-\alpha}
\le 2C_b,
\]
where we used $d\le1$ and $\alpha\le\alpha_b$.

If $d>\frac{R_{0}}{2}$, then $\rho_1=\min\{d,R_0\}\ge \frac{R_0}{2}$, and hence
\[
\vartheta_x(\rho_1)\rho_1^{-\alpha}
\le 2M\left(\frac{R_0}{2}\right)^{-\alpha}.
\]

In either case, we obtain
\[
\vartheta_x(\rho_1)\rho_1^{-\alpha}\le C.
\]
Furthermore, since $\rho_1\le1$ and $\alpha\le\epsilon_0$,
\[
c_1'\rho_1^{\epsilon_0-\alpha}\le c_1'.
\]
Applying Lemma~\ref{lem:iter}, we conclude that
\[
\vartheta_x(\rho)\le C'\rho^\alpha
\qquad\text{for all }\rho\le\rho_1,
\]
where $C'$ depends only on the data. Since $y\in B_{2r}(x)$ and $2r\le\rho_1$,
this gives
\[
|u(x)-u(y)|
\le\vartheta_x(2r)
\le C'(2r)^\alpha\le 2C'r^\alpha.
\]

Combining the two cases, we obtain
\[
|u(x)-u(y)|\le C|x-y|^\alpha
\qquad\text{for all }x,y\in\overline{\Omega}.
\]
Thus
\[
u\in C^{0,\alpha}(\overline{\Omega}),
\]
where $\alpha$ depends only on the data of the class, the data in
\eqref{D22} and \eqref{GG3}, as well as on
$\theta_{0}$, $\delta$, and $\alpha_{1}$.
\end{proof}

\section{Application to Problems in Musielak--Orlicz--Sobolev Spaces}

This section is devoted to the regularity of weak solutions to problem \eqref{P}, with the ultimate goal of proving Theorems~\ref{Theorem 4.3} and~\ref{Theorem 4.4}.

A fundamental step in establishing Hölder continuity ($C^{0,\alpha}$) is to first prove that weak solutions are bounded. Leveraging the embedding results of Diening and Cianchi~\cite{Cianchi2024}, we begin by demonstrating the boundedness of weak solutions to \eqref{P} under both Dirichlet and Neumann boundary conditions, encompassing both critical and subcritical growth conditions on the nonlinearity. This is the subject of Theorems~\ref{Theorem 4.1w} and~\ref{Theorem 6.1w}.

Our strategy for proving boundedness is inspired by the work of Ho, Kim, Winkert, and Zhang~\cite{KKPZ}. We extend their approach to provide a unified treatment for both critical and subcritical cases, thereby yielding a more general result.

\subsection{Boundedness of Weak Solutions to Problems with Dirichlet Boundary Condition}

In this subsection, we focus on proving the boundedness of weak solutions to the Dirichlet problem~\eqref{PD}, as stated in Theorem~\ref{Theorem 4.1w}. To this end, we need the following lemma concerning the geometric convergence of sequences of numbers will be needed for the De Giorgi iteration arguments below. It can be found in Ho-Sim \cite[Lemma 4.3]{Ho-Sim-2015}. The case $\mu_1=\mu_2$ is contained in Lady{\v{z}}enskaja-Solonnikov-Ural{\cprime}ceva \cite[Chapter II, Lemma 5.6]{OL}, see also DiBenedetto \cite[Chapter I, Lemma 4.1]{EDI}.
\begin{lemma}\label{leRecur}
	Let $\{Z_h\}_{h\geq 0}$ be a sequence of positive numbers satisfying the recursion inequality
\[
Z_{h+1} \leq K b^h \left(Z_h^{1+\mu_1} + Z_h^{1+\mu_2}\right), \quad h = 0,1,2,\dots,
\]
for some constants $K > 0$, $b > 1$, and exponents $\mu_2 \geq \mu_1 > 0$. Suppose that either
\[
Z_0 \leq \min \left(1, (2K)^{-\frac{1}{\mu_1}} b^{-\frac{1}{\mu_1^2}} \right),
\]
or
\[
Z_0 \leq \min \left( (2K)^{-\frac{1}{\mu_1}} b^{-\frac{1}{\mu_1^2}}, (2K)^{-\frac{1}{\mu_2}} b^{-\frac{1}{\mu_1 \mu_2} - \frac{\mu_2 - \mu_1}{\mu_2^2}} \right).
\]
Then there exists an index $h_0=0,1,2,\ldots$ such that $Z_h \leq 1$ for all $h \geq h_0$, and moreover:
\[
Z_h \leq \min \left(1, (2K)^{-\frac{1}{\mu_1}} b^{-\frac{1}{\mu_1^2}} b^{-\frac{h}{\mu_1}} \right), \quad \forall h \geq h_0.
\]
In particular, $Z_h \to 0$ as $h \to \infty$.
\end{lemma}
\begin{proof}[\textbf{Proof of Theorem \ref{Theorem 4.1w}}]
The compactness of \( \overline{\Omega} \) implies that for any \( R > 0 \), there exists a finite open cover \( \{B_i(R)\}_{i=1}^{m} \) of balls with radius \( R \), such that
\[
\overline{\Omega} \subset \bigcup_{i=1}^{m} B_i,
\]
where each \( \Omega_i := B_i \cap \Omega \) (for \( i = 1, \ldots, m \)) is also a Lipschitz domain (see, for example, Papageorgiou–Winkert~\cite[ Theorem 1.4.86]{Papageorgiou-Winkert-2018}).

Let $u$ be a weak solution of problem \eqref{PD}. Let $k_*\geq 1$ be sufficiently large such that
	\begin{equation}\label{k*}
		\int_{A_{k_*}}G(x,|D u|){\rm d}x+\int_{A_{k_*}}H_1(x,\vert u\vert){\rm d}x<1,
	\end{equation}
	where $A_k:=\{x\in\Omega\,:\, u(x)>k\}$ for $k\in\mathbb{R}$.\\	
	We define
	\begin{equation}\label{Zn.def}
		Z_h:=\int_{A_{k_h}}G(x,|D u|){\rm d}x+\int_{A_{k_h} }H_1(x,(u-k_{h})){\rm d}x,
	\end{equation}
	where
	\begin{equation*}
		k_h:=k_*\left(2-\frac{1}{2^h}\right)
		\quad \text{for }h =0,1,2,\ldots.
	\end{equation*}
	Obviously, $\text{for }h =0,1,2,\ldots$, it hol{\rm d}s
	\begin{equation}\label{k_n}
		k_h \nearrow 2k_* \quad \text{and} \quad k_* \leq k_h <2k_*.
	\end{equation}
	Since $k_{h}<k_{h+1}$ and $A_{k_{h+1}}\subset A_{k_h}$ for all $h=0,1,2,\ldots$, we have
	\begin{equation}\label{Zn.decreasing}
		Z_{h+1}\leq Z_h\quad\text{for all }h =0,1,2,\ldots.
	\end{equation}
	Moreover, for $x \in A_{k_{h+1}}$,  we see that
	\begin{equation*}
		u(x)- k_h \geq u(x)\left(1-\frac{k_h}{k_{h+1}}\right) = \frac{u(x)}{2^{h+2}-1}.
	\end{equation*}
	Hence, we obtain
	\begin{equation} \label{est.u}
		u(x)\leq (2^{h+2}-1)(u(x)-k_h)\quad \text{for a.\,a.\,}x\in A_{k_{h+1}}\text{ and for all }h=0,1,2,\ldots.
	\end{equation}
	Furthermore, on $A_{k_{h+1}}$ we have $s:=\frac{u-k_h}{k_{h+1}-k_h}\ge1$, hence $H_1(x,s)\ge H_1(x,1)\ge F_1^{-1}$ by \eqref{H111}, and by Proposition \ref{zoo} (applied to $H_1$) and $k_{h+1}-k_h=k_*2^{-(h+1)}$, $k_*\ge1$,
	\begin{align}\label{|A_{k_{h+1}}|}
		\begin{split}
			|A_{k_{h+1}}|
			& \leq F_1\int_{A_{k_{h+1}}}H_1\left (x,\frac{u-k_h}{k_{h+1}-k_h}\right )d x\\
            & \leq F_1\max\left\lbrace \frac{1}{(k_{h+1}-k_h)^{h^-_1}},\frac{1}{(k_{h+1}-k_h)^{h^+_1}} \right\rbrace\int_{A_{k_{h}}} H_1\left (x,u-k_h\right )d x\\
             & \leq F_1\, 2^{h^+_1(h+1)}\int_{A_{k_{h}}} H_1\left (x,u-k_h\right )d x\\
			& \leq F_1\, 2^{h^+_1(h+1)} Z_h\quad \text{for all }h=0,1,2,\ldots.
		\end{split}
	\end{align}
In the remainder of the proof, we denote by \( C_i \) (for \( i = 1, 2, \ldots \)) positive constants that are independent of \( h \) and \( k_* \).\\	
    {\bf Claim 1:}
	There exist positive constants $\mu_1$ and $\mu_2$ such that
	\begin{equation*}
		\int_{A_{k_{h+1}}}H_1(x,(u-k_{h+1})){\rm d}x
		\leq  C_1 2^{\frac{h\left(h^+_1\right)^2}{g^-}}\left(Z_{h}^{1+\mu_1}+Z_{h}^{1+\mu_2}\right)\quad \text{for all }h=0,1,2,\ldots.
	\end{equation*}	
First, note that
	\begin{align}\label{decompose1}
		\begin{split}
			\int_{A_{k_{h+1}}}H_1(x,(u-k_{h+1})){\rm d}x
			&=\int_{\Omega}H_1(x,(u-k_{h+1})^+){\rm d}x\\
			&\leq \sum_{i=1}^m\int_{\Omega_{i}}H_1(x,(u-k_{h+1})^+){\rm d}x.
		\end{split}
	\end{align}
Let $i\in\{1,\ldots,m\}$. From \eqref{k*}, Propositions \ref{zoo}, \ref{zoo*} and \ref{embb} for $\Omega=\Omega_{i}$, we have
	\begin{align*}
		\int_{\Omega_{i}}H_1(x,(u-k_{h+1})^+){\rm d}x
		&\leq \|(u-k_{h+1})^+\|_{L^{H_1(x,t)}(\Omega_{i})}^{h^-_1}\\
		& \leq C_2\left[\|D (u-k_{h+1})^+\|_{L^{G(x,t)}(\Omega_{i})}+\|(u-k_{h+1})^+\|_{L^{G(x,t)}(\Omega_{i})}\right]^{h^-_1}.
	\end{align*}
Then, again by Propositions \ref{zoo}, \ref{zoo*} and \ref{embb} along with \eqref{k*}, this leads to
	\begin{align*}
		\int_{\Omega_{i}}H_1(x,(u-k_{h+1})^+){\rm d}x&\leq C_3\left(\int_{\Omega_{i}}G(x,|D (u-k_{h+1})^+|){\rm d}x+\int_{\Omega_{i}}G(x,(u-k_{h+1})^+)d x\right)^{\frac{h^-_1}{g^+}}\\
		&\leq C_3\left(\int_{A_{k_{h+1}}}G(x,|D u|)d x+\int_{A_{k_{h+1}}}H_1(x,(u-k_{h+1})){\rm d}x+|A_{k_{h+1}}|\right)^{\frac{h^-_1}{g^+}}.
	\end{align*}
From this, \eqref{Zn.def}, \eqref{Zn.decreasing} and \eqref{|A_{k_{h+1}}|}, we obtain
	\begin{align*}
		\int_{\Omega_{i}}H_1(x,(u-k_{h+1})^+){\rm d}x
		\leq C_42^{\frac{hh^+_1h^-_1}{g^+}}Z_{h}^{\frac{h^-_1}{g^+}}.
	\end{align*}
	Combining this with \eqref{decompose1}, we deduce that 
	\begin{equation*}
		\int_{\Omega}H_1(x,(u-k_{h+1})^+{\rm d}x
		\leq C_52^{\frac{h\left(h^+_1\right)^2}{g^-}}\left(Z_{h}^{1+\mu_1}+Z_{h}^{1+\mu_2}\right),
	\end{equation*}
	where
	\begin{align*}
		0<\mu_1:=\frac{h^-_1}{g^+}-1=:\mu_2.
	\end{align*}
 This proves Claim 1.\\
 {\bf Claim 2:} It hol{\rm d}s that
	\begin{equation*}
		\int_{A_{k_{h+1}}}G(x,|D u|){\rm d}x\leq  C_6 2^{h\left[\frac{\left(h^+_1\right)^2}{g^-}+h^+_1\right]}\left(Z_{h-1}^{1+\mu_1}+Z_{h-1}^{1+\mu_2}\right)\quad \text{for all }h=1,2,\ldots.
	\end{equation*}
Testing \eqref{4.9} with $\varphi =(u-k_{h+1})^{+} \in W_0^{1,G(x,t)}(\Omega)$, it yields 
	\begin{align*}
		\int_{\Omega} A(x,u,D u) \cdot D(u-k_{h+1})^{+} d x
		=\int_{\Omega }B(x,u,D u)(u-k_{h+1})^{+} {\rm d}x,
	\end{align*}
which can be written as
	\begin{equation}\label{D.var.Eq}
		\int_{A_{k_{h+1}}} A(x,u,D u) \cdot D u d x
		=\int_{A_{k_{h+1}} }B(x,u,D u)(u-k_{h+1}) d x.
	\end{equation}
We proceed to estimate the terms in \eqref{D.var.Eq}, starting with the left-hand side.

Since \( u \geq u - k_{h+1} > 0 \), and \( u > k_{h+1} \geq 1 \) on the set \( A_{k_{h+1}} \), we have $H_1(x,u)\ge H_1(x,1)\ge F_1^{-1}$ and $u\le F_1H_1(x,u)$ there (by \eqref{H111} and Proposition~\ref{zoo}). Therefore, by assumption \eqref{6.31},
\begin{align*}
\int_{A_{k_{h+1}}} A(x,u,D u)\cdot D u \, \mathrm{d}x
&\geq a_4 \int_{A_{k_{h+1}}} G(x,|D u|) \, \mathrm{d}x - a_5 \int_{A_{k_{h+1}}} H_1(x,u) \, \mathrm{d}x - a_6|A_{k_{h+1}}| \\
&\geq a_4 \int_{A_{k_{h+1}}} G(x,|D u|) \, \mathrm{d}x 
- \max\{a_5, F_1a_6\} \int_{A_{k_{h+1}}} H_1(x,u) \, \mathrm{d}x.
\end{align*}
We now estimate the right-hand side of \eqref{D.var.Eq}. Let \( \varepsilon \in (0,1] \) be chosen such that  
\[
b_1 \varepsilon^{\frac{h^+_1}{h^+_1 - 1}} \leq \frac{a_4}{4}.
\]  
 Then, employing assumption \eqref{6.51}, Young’s inequality \eqref{Yi}, the estimate \eqref{L1}, and Propositions \ref{zoo} and \ref{zoo*}, we find
\begin{align*}
\int_{A_{k_{h+1}}} B(x,u,D u)(u - k_{h+1}) \, \mathrm{d}x 
&\leq b_1 \int_{A_{k_{h+1}}} M(x,|D u|) u \, \mathrm{d}x 
+ b_2 \int_{A_{k_{h+1}}} h_1(x,u) u \, \mathrm{d}x 
+ \int_{A_{k_{h+1}}} b_3(u - k_{h+1}) \, \mathrm{d}x \\
&\leq b_1 \int_{A_{k_{h+1}}} \frac{\varepsilon}{\varepsilon} M(x,|D u|) u \, \mathrm{d}x 
+ \max\{h^+_1 b_2, F_1b_3\} \int_{A_{k_{h+1}}} H_1(x,u) \, \mathrm{d}x \\
&\leq b_1 \int_{A_{k_{h+1}}} \widetilde{H_1}\left(x, \varepsilon M(x,|D u|) \right) \, \mathrm{d}x 
+ b_1 \int_{A_{k_{h+1}}} H_1\left(x, \frac{1}{\varepsilon} u \right) \, \mathrm{d}x \\
&\quad + \max\{h^+_1 b_2, F_1b_3\} \int_{A_{k_{h+1}}} H_1(x,u) \, \mathrm{d}x \\
&\leq b_1 \varepsilon^{\frac{h^+_1}{h^+_1 - 1}} \int_{A_{k_{h+1}}} \widetilde{H_1}\left(x, M(x,|D u|) \right) \, \mathrm{d}x 
+ b_1 \varepsilon^{-h^+_1} \int_{A_{k_{h+1}}} H_1(x,u) \, \mathrm{d}x \\
&\quad + \max\{h^+_1 b_2, F_1b_3\} \int_{A_{k_{h+1}}} H_1(x,u) \, \mathrm{d}x \\
&\leq b_1 \varepsilon^{\frac{h^+_1}{h^+_1 - 1}} \int_{A_{k_{h+1}}} G(x,|D u|) \, \mathrm{d}x 
+ C_7 \int_{A_{k_{h+1}}} H_1(x,u) \, \mathrm{d}x,
\end{align*}
where 
\[
C_7 := \max\left\{h^+_1 b_2, F_1b_3, b_1 \varepsilon^{-h^+_1} \right\}.
\]
Using the choice of \( \varepsilon \), we conclude that
\[
\int_{A_{k_{h+1}}} B(x,u,D u)(u - k_{h+1}) \, \mathrm{d}x 
\leq \frac{a_4}{4} \int_{A_{k_{h+1}}} G(x,|D u|) \, \mathrm{d}x 
+ C_7 \int_{A_{k_{h+1}}} H_1(x,u) \, \mathrm{d}x.
\]
Combining the estimates above with \eqref{D.var.Eq} and then using \eqref{est.u}, we obtain
	\begin{align*}
		\int_{A_{k_{h+1}}}G(x,|D u|){\rm d}x &
		\leq  C_8\int_{A_{k_{h+1}}}H_1(x,u) {\rm d}x\\
		& \leq C_8\int_{A_{k_{h+1}}}H_1\left(x,(2^{h+2}-1)(u-k_h)\right) {\rm d}x.
	\end{align*}
	Hence,
	\begin{equation*}
		\int_{A_{k_{h+1}}}G(x,|D u|) {\rm d}x
		\leq C_92^{hh^+_1}\int_{\Omega}H_1(x,(u-k_h)^+) {\rm d}x.
	\end{equation*}
Then, Claim~2 follows from the preceding inequality together with Claim~1.\\
From Claims 1 and 2, along with \eqref{Zn.decreasing}, we conclude that
	\begin{equation}\label{Recur}
		Z_{h+1}\leq C_{10} b^h\left(Z_{h-1}^{1+\mu_1}+Z_{h-1}^{1+\mu_2}\right)\quad \text{for all }h=1,2,\ldots,
	\end{equation}
	where
	\begin{align*}
		b:=2^{\big[\frac{\left(h^+_1\right)^2}{g^-}+h^+_1\big]}>1.
	\end{align*}
	This yields to
	\begin{equation*}
		Z_{2(h+1)}\leq C_{10} b^{2h+1}\left(Z_{2h}^{1+\mu_1}+Z_{2h}^{1+\mu_2}\right)\quad \text{for all }h=0,1,2,\ldots,
	\end{equation*}
	that is,
	\begin{equation}\label{Recur1}
		\tilde{Z}_{h+1}\leq bC_{10} \tilde{b}^h\left(\tilde{Z}_h^{1+\mu_1}+\tilde{Z}_h^{1+\mu_2}\right) \quad \text{for all }h=0,1,2,\ldots,
	\end{equation}
	where $\tilde{Z}_h:=Z_{2h}$ and $\tilde{b}:=b^2$. Applying Lemma \ref{leRecur} to \eqref{Recur1} yields
	\begin{equation}\label{Recur+1}
		Z_{2h}=\tilde{Z}_h \to 0 \quad \text {as }  h\to \infty
	\end{equation}
	provided that
	\begin{equation}\label{Z_0}
		\tilde{Z}_{0}\leq \min\left\{(2bC_{10})^{-\frac{1}{\mu_1}}\ \tilde{b}^{-\frac{1}{\mu_1^{2}}},\left(2bC_{10}\right)^{-\frac{1}{\mu_2}}\ \tilde{b}^{-\frac{1}{\mu_1\mu_2}-\frac{\mu_2-\mu_1}{\mu_2^{2}}}\right\}.
	\end{equation}
	Again, from \eqref{Recur} we obtain
	\begin{equation*}
		Z_{2(h+1)+1}\leq C_{10} b^{2(h+1)}\left(Z_{2h+1}^{1+\mu_1}+Z_{2h+1}^{1+\mu_2}\right)\quad \text{for all } h=0,1,2,\ldots,
	\end{equation*}
	which can be written as
	\begin{equation}\label{Recur2}
		\hat{Z}_{h+1}\leq \tilde{b}C_{10} \tilde{b}^h\left(\hat{Z}_h^{1+\mu_1}+\hat{Z}_h^{1+\mu_2}\right) \quad \text{for all }h=0,1,2,\ldots,
	\end{equation}
	where $\hat{Z}_h:=Z_{2h+1}$. From Lemma \ref{leRecur} applied to \eqref{Recur2} it follows that
	\begin{equation}\label{Recur+2}
		Z_{2h+1}=\hat{Z}_h \to 0 \quad \text{as } h\to \infty
	\end{equation}
	provided that
	\begin{equation}\label{Z_0'}
		\hat{Z}_{0}\leq \min\left\{(2\tilde{b}C_{10})^{-\frac{1}{\mu_1}}\ \tilde{b}^{-\frac{1}{\mu_1^{2}}},\left(2\tilde{b}C_{10}\right)^{-\frac{1}{\mu_2}}\ \tilde{b}^{-\frac{1}{\mu_1\mu_2}-\frac{\mu_2-\mu_1}{\mu_2^{2}}}\right\}.
	\end{equation}
	Note that
	\begin{align*}
		\hat{Z}_{0}=Z_1\leq Z_0=\tilde{Z}_0\leq \int_{A_{k_*}}G(x,|D u|) d x+\int_{A_{k_*} }H_1(x,u){\rm d}x.
	\end{align*}
	Therefore, by choosing $k_*>1$ sufficiently large we have
	\begin{align*}
		\int_{A_{k_*}}G(x,|D u|) d x+\int_{A_{k_*} }H_1(x,u){\rm d}x
		&\leq \min\left\{1,(2\tilde{b}C_{10})^{-\frac{1}{\mu_1}}\ \tilde{b}^{-\frac{1}{\mu_1^{2}}},\left(2\tilde{b}C_{10}\right)^{-\frac{1}{\mu_2}}\ \tilde{b}^{-\frac{1}{\mu_1\mu_2}-\frac{\mu_2-\mu_1}{\mu_2^{2}}}\right\}.
	\end{align*}
	Hence, \eqref{k*}, \eqref{Z_0} and \eqref{Z_0'} are fulfilled and we obtain \eqref{Recur+1} and \eqref{Recur+2}. This means that
	\begin{align*}
		Z_h=\int_{A_{k_h}}G(x,|D u|) d x+\int_{A_{k_h} }H_1(x,(u-k_h)) d x\to 0 \quad \text{as } h\to\infty.
	\end{align*}
	In particular, we have
	\begin{align*}
		\int_{\Omega}H_1(x,(u-2k_*)^+)d x=0.
	\end{align*}
	Consequently, by Proposition \ref{zoo}, $(u-2k_*)^{+}=0$ a.\,e.\,in $\Omega$ and so
	\begin{align*}
		\underset{\Omega}{\mathop{\rm ess\,sup}}\ u \leq 2k_*.
	\end{align*}
Replacing $u$ by $-u$ in the arguments above, we also obtain
	\begin{align*}
		\underset{\Omega}{\mathop{\rm ess\,sup}}\ (-u) \leq 2k_*.
	\end{align*}
Hence, $\|u\|_\infty\leq 2k_*$. This finishes the proof.
\end{proof}
\subsection{Boundedness of Weak Solutions to Problems with Neumann Boundary Condition}\label{subsection3}

In this subsection, we address the boundedness of weak solutions to the Neumann problem~\eqref{PN}, as formulated in Theorem~\ref{Theorem 6.1w}.\\

We emphasize that the Neumann boundary condition case requires additional assumptions compared to the Dirichlet case. The main difference lies in the use of Fan's Sobolev conjugate function $G_*(x,t)$ and Fan's embedding results, which replace the Cianchi-Diening Sobolev conjugate $G^*(x,t)$. This is necessary because the Neumann case involves boundary integrals and therefore requires a trace embedding theorem. Such a theorem has been established by Liu, Wang, and Zhao \cite{LWZ} within the framework of Fan's Sobolev conjugate $G_*(x,t)$. To the best of our knowledge, a corresponding trace embedding result based on the Cianchi-Diening conjugate $G^*(x,t)$ is not yet available.

\begin{proof}[\textbf{Proof of Theorem \ref{Theorem 6.1w}}]
	As before, since $\overline{\Omega}$ is compact, for any \( R > 0 \), there exists a finite open cover \( \{B_i(R)\}_{i=1}^m \) of balls \( B_i := B_i(R) \), each of radius \( R \), such that 
\[
\overline{\Omega} \subset \bigcup_{i=1}^{m} B_i,
\]
and each intersection \( \Omega_i := B_i \cap \Omega \) (for \( i = 1, \ldots, m \)) is also a Lipschitz domain.\\ 
We denote by \( I \subset \{1, \ldots, m\} \) the set of indices for which \( B_i \cap \partial\Omega \neq \emptyset \), and we choose \( R > 0 \) sufficiently small.\\
Let \( u \) be a weak solution to problem~\eqref{PN}, and let \( k_* \geq 1 \) be sufficiently large such that 
	\begin{equation}\label{N.k*11}
		\int_{A_{k_*}}G(x,|D u|){\rm d}x+\int_{A_{k_*}}H_1(x,u){\rm d}x+\int_{\Gamma_{k_*}}H_2(x,u){\rm d}\sigma<1,
	\end{equation}
	where
	\begin{align*}
		A_k:=\{x\in\Omega \, : \, u(x)>k\} \quad\text{and}\quad \Gamma_k:=\{x\in\partial\Omega \,:\, u(x)>k\} \quad\text{for }k\in\mathbb{R}.
	\end{align*}
	We define
	\begin{equation}\label{N.Zn.def11}
		Z_h:=\int_{A_{k_h}}G(x,|D u|){\rm d}x+\int_{A_{k_h} }H_1\left(x,(u-k_h)\right){\rm d}x+\int_{\Gamma_{k_h}}H_2\left(x,(u-k_h)\right){\rm d}\sigma,
	\end{equation}
	where $\{k_{h}\}_{h\geq 0}$ is a sequence defined as in the proof of Theorem \ref{Theorem 4.1w}. Since $k_{h}<k_{h+1}$, we have $A_{k_{h+1}}\subset A_{k_{h}}$ and $\Gamma_{k_{h+1}}\subset \Gamma_{k_{h}}$ for all $h=0,1,2,\ldots$. Hence, it hol{\rm d}s
	\begin{equation}\label{N.Zn.decreasing11}
	Z_{h+1}\leq Z_h\quad\text{for all }h=0,1,2,\ldots.
	\end{equation}
	Moreover, we have the following estimates (see the proof of Theorem \ref{Theorem 4.1w})
	\begin{align}
		u(x)&\leq \left(2^{h+2}-1\right)(u(x)-k_{h})\quad \text{for a.\,a.\,}x\in A_{k_{h+1}} \text{ and for all } h=0,1,2,\ldots,\label{N.est.u111}\\
		u(x)&\leq \left(2^{h+2}-1\right)(u(x)-k_{h})\quad \text{for a.\,a.\,}x\in \Gamma_{k_{h+1}}\text{ and for all } h=0,1,2,\ldots. \label{N.est.u211}
	\end{align}
	As shown in \eqref{|A_{k_{h+1}}|} the following estimate holds true
	\begin{equation} \label{N.|A_k|1}
	|A_{k_{h+1}}|  \leq F_1\,2^{h_1^+(h+1)} Z_h\quad \text{for all }h=0,1,2,\ldots.
	\end{equation}
	In the following, we will denote by $C_i$ ($i=1,2,3,\ldots$) positive constants independent of $h$ and $k_*$.\\
	{\bf Claim 1:} It hol{\rm d}s that
	\begin{equation*}
		\int_{A_{k_{h+1}}}H_1\left(x,(u-k_{h+1})\right){\rm d}x\leq  C_1 2^{\frac{h\left(h^+_1\right)^2}{g^-}}\left(Z_{h}^{1+\nu_1}+Z_{h}^{1+\nu_2}\right)\quad \text{for all } h=0,1,2,\ldots,
	\end{equation*}
	where
	\begin{align*}
		0<\nu_1:=\frac{h_1^-}{g^+}-1=:\nu_2.
	\end{align*}
	The proof of Claim 1 is the same as Claim 1 in Theorem \ref{Theorem 4.1w}.\\	
	{\bf Claim 2:} There exist positive constants $\nu_3$ and $\nu_4$ such that
	\begin{equation*}
		\int_{\Gamma_{k_{h+1}}}H_2\left(x,(u-k_{h+1})\right){\rm d} \sigma\leq  C_2 2^{\frac{h\left(h^+_2\right)^2}{g^-}}\left(Z_{h}^{1+\nu_3}+Z_{h}^{1+\nu_4}\right)\quad \text{for all } h=0,1,2,\ldots.
	\end{equation*}
	First, we see that
	\begin{align}\label{N.deco.Bdr11}
		\begin{split}
			\int_{\Gamma_{k_{h+1}}}H_2(x,(u-k_{h+1})){\rm d} \sigma
			&=\int_{\partial\Omega}H_2(x,(u-k_{h+1})^+){\rm d} \sigma\\
			& \leq \sum_{i=1}^m\int_{\partial\Omega\cap \partial \Omega_{i}}H_2(x,(u-k_{h+1})^+){\rm d} \sigma.
		\end{split}
	\end{align}
	Let \( i \in \{1, \ldots, m\} \). From  Propositions~\ref{zoo} and \ref{fzoo*}, inequality \eqref{N.k*11}, and by the trace embedding result \(W^{1,G(x,t)}(\Omega) \hookrightarrow L^{H_2(x,t)}(\partial\Omega)\) ( see Proposition~\ref{ffembd}), for \( \Omega = \Omega_i \), we have
	\begin{align*}
		& \int_{\partial\Omega\cap \partial \Omega_{i}}H_2(x,(u-k_{h+1})^+){\rm d} \sigma\\
		& \leq \|(u-k_{h+1})^+\|_{L^{H_2(x,t)}(\partial\Omega\cap \partial \Omega_{i})}^{h_2^-}\\
        & \leq \|(u-k_{h+1})^+\|_{L^{H_2(x,t)}( \partial \Omega_{i})}^{h_2^-}\\
		&\leq C_3\left[\|D (u-k_{h+1})^+\|_{L^{G(x,t)}(\Omega_{i})}+\|(u-k_{h+1})^+\|_{L^{G(x,t)}(\Omega_{i})}\right]^{h_2^-}.
	\end{align*}
	Then, by Propositions \ref{zoo}, \ref{fzoo*} and inequality \eqref{N.k*11} it follows
	\begin{align*}
		& \int_{\partial\Omega\cap \partial \Omega_{i}}H_2(x,(u-k_{h+1})^+){\rm d} \sigma\\
		&\leq C_4\left(\int_{\Omega_{i}}G(x,|\nabla (u-k_{h+1})^+|){\rm d}x+\int_{\Omega_{i}}G(x,(u-k_{h+1})^+){\rm d}x\right)^{\frac{h_2^-}{g^+}}\\
	\notag&\leq C_4\left(\int_{A_{k_{h+1}}}G(x,|D u|){\rm d}x+\int_{A_{k_{h+1}}}H_1(x,(u-k_{h+1})){\rm d}x+|A_{k_{h+1}}|\right)^{\frac{h_2^-}{g^+}}.
	\end{align*}
	From this combined with \eqref{N.Zn.def11}, \eqref{N.Zn.decreasing11} as well as \eqref{N.|A_k|1} we obtain
	\begin{align}\label{N.grad311}
		\int_{\partial\Omega\cap \partial \Omega_{i}}H_2(x,(u-k_{h+1})^+){\rm d} \sigma&\leq C_52^{\frac{hh_2^+h_2^-}{g^+}}Z_{h}^{\frac{h_2^-}{g^+}}.
	\end{align}
	Here we used that, by \eqref{GG1}, \eqref{H111} and $g^+<h_1^-$, $G(x,s)\le F(1+F_1H_1(x,s))$ for $s\ge0$, so that the term $\int_{\Omega_i}G(x,(u-k_{h+1})^+)\,dx$ is controlled by $|A_{k_{h+1}}|$ and $\int_{A_{k_{h+1}}}H_1(x,u-k_{h+1})\,dx$.
Consequently, by \eqref{N.deco.Bdr11} and \eqref{N.grad311}, we conclude that
	\begin{equation*}
		\int_{\partial\Omega}H_2(x,(u-k_{h+1})^+){\rm d}\sigma\leq C_62^{\frac{h\left(h_2^+\right)^2}{g^-}}\left(Z_{h}^{1+\nu_3}+Z_{h}^{1+\nu_4}\right),
	\end{equation*}
	where
	\begin{align*}
		0<\nu_3:=\frac{h_2^-}{g^+}-1=:\nu_4.
	\end{align*}
 This proves Claim 2.\\	
	{\bf Claim 3:} It hol{\rm d}s that
	\begin{equation*}
		\int_{A_{k_{h+1}}}G(x,|D u|){\rm d}x\leq  C_7 2^{h\left[\frac{\left(h_1^++h_2^+\right)^2}{g^-}+h_1^++h_2^+\right]}\left(Z_{h-1}^{1+\mu_1}+Z_{h-1}^{1+\mu_2}\right)\quad \text{for all }h=1,2,3,\ldots,
	\end{equation*}
	where
	\begin{align*}
		0<\mu_1:=\nu_1=\nu_2=\frac{h_1^-}{g^+}-1\quad  \text{    and    }\quad 0<\mu_2:=\nu_3=\nu_4=\frac{h_2^-}{g^+}-1.
	\end{align*}
	Taking $\varphi =(u-k_{h+1})^{+} \in W^{1,G(x,t)}(\Omega)$ as test function in \eqref{FNV} gives
	\begin{align*}
		\int_{\Omega}
		A(x,u,D u) \cdot \nabla
		\varphi {\rm d}x =\int_{\Omega }B(x,u,D u)\varphi {\rm d}x+\int_{\partial\Omega}C(x,u)\varphi {\rm d}\sigma,
	\end{align*}
	which can be written as
	\begin{align}\label{estimate_311}
		\begin{split}
			&\int_{A_{k_{h+1}}} A(x,u,D u) \cdot D u {\rm d}x\\ 
			&=\int_{A_{k_{h+1}} }B(x,u,D u)(u-k_{h+1}) {\rm d}x +\int_{\Gamma_{k_{h+1}} }C(x,u)(u-k_{h+1}) {\rm d} \sigma.
		\end{split}
	\end{align}
	As done in the proof of Theorem \ref{Theorem 4.1w}, we apply assumptions \eqref{6.31}, \eqref{6.51} and \eqref{6.61}, as well as Young's inequality \eqref{Yi}, \eqref{L1} and Propositions \ref{zoo} and \ref{fzoo*}, and we take $\varepsilon \in (0,1]$ such that
    $$b_1\varepsilon^{\frac{h_1^+}{h_1^+-1}}\leq\frac{a_4}{4},$$
  we have the estimates
	\begin{align}\label{estimate_411}
		\begin{split}
			&\int_{A_{k_{h+1}}} A(x,u,D u)\cdot D u {\rm d}x\\
			&\geq a_4\int_{A_{k_{h+1}} }G(x,|D u|) {\rm d}x-\max\{a_5,F_1a_6\}\int_{A_{k_{h+1}} }H_1(x,u){\rm d}x,
		\end{split}
	\end{align}
	and
	\begin{align}\label{estimate_511}
		\begin{split}
			&\int_{A_{k_{h+1}}} B(x,u,D u)(u-k_{h+1}){\rm d}x\\
            & \leq b_1\int_{A_{k_{h+1}}}  M(x,|D u|)u {\rm d}x + b_2 h_1^+\int_{A_{k_{h+1}}}  H_1(x, u) {\rm d}x+b_3F_1\int_{A_{k_{h+1}}}  H_1(x,u) d x\\
        & \leq b_1\int_{A_{k_{h+1}} }\widetilde{H_1}\left(x,\varepsilon M(x,|D u|)\right){\rm d}x + b_1\int_{A_{k_{h+1}} }H_1\left(x,\frac{1}{\varepsilon}u\right){\rm d}x
        +\max\{h_1^+b_2,F_1b_3\}\int_{A_{k_{h+1}}}H_1(x,u){\rm d}x \\
        & \leq b_1\varepsilon^{\frac{h_1^+}{h_1^+-1}}\int_{A_{k_{h+1}} }G(x,|D u|){\rm d}x +\max\{h_1^+b_2,F_1b_3,b_1\varepsilon^{-h_1^+}\}\int_{A_{k_{h+1}}}H_1(x,u){\rm d}x\\ 
	& = \frac{a_4}{4}\int_{A_{k_{h+1}} }G(x,|D u|){\rm d}x +C_7\int_{A_{k_{h+1}}}H_1(x,u){\rm d}x.
		\end{split}
	\end{align}
 Note that $u\geq u-k_{h+1} >0$ and $u> k_{h+1} \geq 1$ on $\Gamma_{k_{h+1}}$. So, applying assumption \eqref{6.61} and \eqref{NH111}, it yields that
	\begin{align}\label{estimate_611}
		\begin{split}
			\int_{\Gamma_{k_{h+1}} }C(x,u)(u-k_{h+1}) {\rm d}\sigma
			& \leq c_1 \int_{\Gamma_{k_{h+1}} }h_2(x,u) u{\rm d}\sigma+c_2 \int_{\Gamma_{k_{h+1}} }u {\rm d}\sigma\\
            & \leq c_1 h_2^+ \int_{\Gamma_{k_{h+1}} }H_2(x,u) {\rm d}\sigma+c_2F_2 \int_{\Gamma_{k_{h+1}} }H_2(x,u) {\rm d}\sigma\\
			& \leq (c_1h_2^++F_2c_2)\int_{\Gamma_{k_{h+1}} }H_2(x,u) {\rm d}\sigma.
		\end{split}
	\end{align}
	Combining \eqref{estimate_311}, \eqref{estimate_411}, \eqref{estimate_511} and \eqref{estimate_611}, and using \eqref{N.est.u111} as well as \eqref{N.est.u211}, we obtain
	\begin{align*}
		&\int_{A_{k_{h+1}}}G(x,|D u|){\rm d}x\\
		& \leq  C_8\int_{A_{k_{h+1}}}H_1(x,u){\rm d}x+C_8\int_{\Gamma_{k_{h+1}}}H_2(x,u){\rm d} \sigma\\
		&\leq C_8\int_{A_{k_{h+1}}}H_1\left(x,\left(2^{h+2}-1\right)(u-k_{h})\right){\rm d}x\\
		&\qquad +C_8\int_{\Gamma_{k_{h+1}}}H_2\left(x,\left(2^{h+2}-1\right)(u-k_{h})\right){\rm d} \sigma.
	\end{align*}
	Hence,
	\begin{equation*}
		\int_{A_{k_{h+1}}}G(x,|D u|){\rm d}x
		\leq C_92^{h(h_1^++h_2^+)}\left[\int_{A_{k_{h}}}H_1(x,(u-k_{h})){\rm d}x+\int_{\Gamma_{k_h}}H_2(x,(u-k_{h}) ){\rm d}\sigma\right].
	\end{equation*}
	Then, Claim 3 follows from the last inequality and Claims 1 and 2.\\
	From Claims 1, 2, and 3, along with \eqref{N.Zn.decreasing11} one has
	\begin{equation}\label{N.Recur1}
		Z_{h+1}\leq C_{10} b^h\left(Z_{h-1}^{1+\mu_1}+Z_{h-1}^{1+\mu_2}\right)\quad \text{for all } h=1,2,3,\ldots,
	\end{equation}
	where
	\begin{align*}
		b:=2^{\left[\frac{\left(h_1^++h_2^+\right)^2}{g^-}+h_1^++h_2^+\right]}>1.
	\end{align*}
	Repeating the same arguments used in the proof of Theorem~\ref{Theorem 4.1w}, by choosing \( k_* > 1 \) sufficiently large such that
	\begin{align*}
		&\int_{A_{k_*}}G(x,|D u|){\rm d}x+\int_{A_{k_*} }H_1(x,u){\rm d}x+\int_{\Gamma_{k_*} }H_2(x,u){\rm d} \sigma\\
		&\leq \min\left\{1,(2\tilde{b}C_{10})^{-\frac{1}{\mu_1}}\ \tilde{b}^{-\frac{1}{\mu_1^{2}}},\left(2\tilde{b}C_{10}\right)^{-\frac{1}{\mu_2}}\ \tilde{b}^{-\frac{1}{\mu_1\mu_2}-\frac{\mu_2-\mu_1}{\mu_2^{2}}}\right\},
	\end{align*}
	where $\tilde{b}:=b^2$, we deduce from \eqref{N.Recur1} that
	\begin{align*}
		Z_h=\int_{A_{k_{h}}}G(x,|D u|){\rm d}x+\int_{A_{k_{h}} }H_1(x,(u-k_{h})){\rm d}x+\int_{\Gamma_{k_{h}} }H_2(x,(u-k_{h})){\rm d} \sigma\to 0
	\end{align*}
	as $h\to \infty$. This implies that
	\begin{align*}
		\int_{\Omega}H_1(x,(u-2k_*)^+) {\rm d}x+\int_{\partial\Omega}H_2(x,(u-2k_*)^+) {\rm d} \sigma=0.
	\end{align*}
	Therefore, $(u-2k_*)^{+}=0$ a.\,e.\,in $\Omega$ and $(u-2k_*)^{+}=0$ a.\,e.\,on $\partial\Omega$. This means that $\operatorname{ess\,sup}_\Omega u\le 2k_*$ and $\operatorname{ess\,sup}_{\partial\Omega}u\le 2k_*$.
By replacing \( u \) with \( -u \) in the arguments above, we can show in the same way that
\[
\underset{\Omega}{\mathop{\rm ess\,sup}} (-u)\le 2k_*,\qquad \underset{\partial\Omega}{\mathop{\rm ess\,sup}} (-u) \leq 2k_*.
\]
Hence,
\[
\|u\|_{L^\infty(\Omega)}\le 2k_*,\qquad \|u\|_{L^\infty(\partial\Omega)} \leq 2k_*.
\]
The proof is complete.
\end{proof}

\subsection{Hölder Continuity of Weak Solutions to Problems with Dirichlet or Neumann Boundary Conditions}
We now turn to the Hölder continuity of weak solutions to equation~\eqref{P}. Since the boundedness condition~\eqref{M} required in Theorem~\ref{Theorem 4.2} is ensured by Theorems~\ref{Theorem 4.1w} and~\ref{Theorem 6.1w}, we are in a position to establish the Hölder regularity of weak solutions to Problem~\eqref{P}: interior regularity under both Dirichlet~\eqref{PD} and Neumann~\eqref{PN} boundary conditions, and regularity up to the boundary for the Dirichlet problem. This result is summarized in Theorems~\ref{Theorem 4.2}–\ref{Theorem 4.4}.
\begin{proof}[\textbf{Proof of Theorem \ref{Theorem 4.2}}]
Let \(\bar{B}_t \subset B_s \subset \Omega\) and select a \(C^\infty\)-function \(\xi\) satisfying
\begin{equation}\label{55eq}
0 \leq \xi \leq 1, \quad \text{supp } \xi \subset B_s, \quad \xi \equiv 1 \text{ on } B_t, \quad |D\xi| \leq \frac{2}{s - t}.
\end{equation}
Define
\begin{equation}\label{4.29}
 \delta = 
\begin{cases}
\min\left\{ \dfrac{a_4}{4b_3M}, 2 \right\}, & \text{if } b_3 > 0, \\[1.2ex]
2, & \text{if } b_3 = 0,
\end{cases}
\end{equation}
and set \(v := \xi^{g^+} (u - k)^+ = \xi^{g^+} \max \{u - k, 0\}\), where
\begin{equation}\label{4.30}
k \geq \max_{B_s} u(x) - \delta M.
\end{equation}
Clearly, \(v \in W_0^{1,G(x,t)}(\Omega)\). Substituting into equation \eqref{4.9} yields
\begin{align}\label{eq55}
      \int_{A_{k,s}} \xi^{g^+} A(x, u, Du) \cdot Du  \mathrm{d}x &+ \int_{A_{k,s}} g^+\xi^{g^+-1} (u - k) A(x, u, Du) \cdot D\xi  \mathrm{d}x \\
      &= \int_{A_{k,s}} B(x, u, Du) \xi^{g^+} (u - k)  \mathrm{d}x.    
\end{align}
In light of conditions \eqref{6.31}--\eqref{6.41} and \eqref{6.51}, we derive the inequality
\begin{align}\label{4.31}
\begin{split}
a_4 \int_{A_{k,s}} G(x,|Du|) \xi^{g^+}  \mathrm{d}x &\leq a_5 \int_{A_{k,s}} H_1(x,|u|) \xi^{g^+}  \mathrm{d}x + a_6 \int_{A_{k,s}} \xi^{g^+}  \mathrm{d}x \\
&\quad + a_1 g^+ \int_{A_{k,s}} g(x,|Du|) \xi^{g^+-1} |D\xi|(u - k)  \mathrm{d}x \\
&\quad + a_2 g^+\int_{A_{k,s}} P(x,|u|) \xi^{g^+-1} |D\xi|(u - k)  \mathrm{d}x \\
&\quad + a_3g^+ \int_{A_{k,s}} |D\xi| \xi^{g^+-1}(u - k)  \mathrm{d}x \\
&\quad + b_1 \int_{A_{k,s}}M(x,|Du|) \xi^{g^+} (u - k)  \mathrm{d}x \\
&\quad + b_2 \int_{A_{k,s}}h_1(x,|u|) \xi^{g^+} (u - k)  \mathrm{d}x \\
&\quad + b_3 \int_{A_{k,s}} \xi^{g^+} (u - k)  \mathrm{d}x.
\end{split}
\end{align}
We proceed to estimate each term on the right-hand side of \eqref{4.31}.\\
 Employing \eqref{L1} and Young's inequality \eqref{Yi}, select \(\varepsilon_1 \in (0,1)\) such that
\[
a_1 (g^+)^2 \varepsilon_1^{\frac{g^+}{g^+-1}} \leq \frac{a_4}{4},
\]
 gives that
\begin{align}\label{4.32}
\begin{split}
&a_1 g^+ \int_{A_{k,s}} g(x,|Du|) \xi^{g^+-1} |D\xi|(u - k)  \mathrm{d}x \\
&\quad \leq a_1 g^+ \int_{A_{k,s}} \widetilde{G}\left(x, g(x,|Du|) \xi^{g^+-1}\varepsilon_1\right)  \mathrm{d}x + a_1 g^+ \int_{A_{k,s}} G\left(x, \frac{|D\xi|(u - k)}{\varepsilon_1}\right)  \mathrm{d}x \\
&\quad \leq a_1 g^+ \varepsilon_1^{\frac{g^+}{g^+-1}} \int_{A_{k,s}} \widetilde{G}\left(x, g(x,|Du|) \right) \xi^{g^+}  \mathrm{d}x + a_1 g^+ \varepsilon_1^{-g^+} \int_{A_{k,s}} G\left(x, \frac{2(u - k)}{s-t}\right)  \mathrm{d}x \\
&\quad \leq a_1 (g^+)^2 \varepsilon_1^{\frac{g^+}{g^+-1}} \int_{A_{k,s}} G(x,|Du|)  \xi^{g^+}  \mathrm{d}x + a_1 g^+ \varepsilon_1^{-g^+} 2^{g^+} \int_{A_{k,s}} G\left(x, \frac{u - k}{s-t}\right)  \mathrm{d}x \\
&\quad \leq \frac{a_4}{4} \int_{A_{k,s}} G(x,|Du|) \xi^{g^+}  \mathrm{d}x + d_1 \int_{A_{k,s}} G\left(x, \frac{u - k}{s - t} \right)  \mathrm{d}x,
\end{split}
\end{align}
with \(d_1 = d_1(a_1, a_4, g^+) > 0\).\\
 An application of Young's inequality \eqref{Yi} and Proposition \ref{zoo}, with \(\varepsilon_2 \in (0,1)\) chosen so that
\[
a_3g^+ 2^{g^+} \varepsilon_2^{g^-} \leq 1,
\]
yields that
\begin{align}\label{4.33}
\begin{split}
&a_3g^+ \int_{A_{k,s}} |D\xi|(u - k) \xi^{g^+-1}  \mathrm{d}x \leq a_3g^+ \int_{A_{k,s}} |D\xi|(u - k)  \mathrm{d}x \\
&\quad \leq a_3g^+ \int_{A_{k,s}} G\left(x, \varepsilon_2 |D\xi| (u - k)\right)  \mathrm{d}x + a_3g^+ \int_{A_{k,s}} \widetilde{G}\left(x, \frac{1}{\varepsilon_2} \right)  \mathrm{d}x \\
&\quad \leq a_3g^+\varepsilon_2^{g^-} 2^{g^+} \int_{A_{k,s}} G\left(x, \frac{u - k}{s-t}\right)  \mathrm{d}x + a_3g^+ \int_{A_{k,s}} \widetilde{G}\left(x, \frac{1}{\varepsilon_2} \right)  \mathrm{d}x \\
&\quad \leq \int_{A_{k,s}} G\left(x, \frac{u - k}{s - t} \right)  \mathrm{d}x + c_1 |A_{k,s}|,
\end{split}
\end{align}
where \(c_1 = c_1(a_3, g^\pm, F) > 0\) (here we used that, by \eqref{GG1} and Proposition~\ref{zoo}, $\widetilde G(x,1)=\sup_{s\ge0}(s-G(x,s))\le \sup_{s\ge0}(s-F^{-1}\min\{s^{g^-},s^{g^+}\})\le C(F,g^\pm)$, and hence $\widetilde G(x,1/\varepsilon_2)\le C(F,g^\pm)\max\{\varepsilon_2^{-\widetilde{g^+}},\varepsilon_2^{-\widetilde{g^-}}\}$).\\
 From \eqref{H111}, \eqref{M}, and Proposition \ref{zoo}, we find the following.
\begin{align}\label{11eq}
    \begin{split}
        a_5 \int_{A_{k,s}} H_1(x,|u|) \xi^{g^+}  \mathrm{d}x &\leq a_5 \int_{A_{k,s}} H_1(x, M) \xi^{g^+}  \mathrm{d}x \\
        &\leq a_5 \int_{A_{k,s}} F_1 M^{h_1^+} \xi^{g^+}  \mathrm{d}x \leq c_2 |A_{k,s}|,
    \end{split}
\end{align}
with \(c_2 = c_2(a_5, F_1, h_1^+, M) > 0\).\\
 Invoking Young's inequality \eqref{Yi} with the factor $1$ in front of $\widetilde G$ (no smallness is needed here), Proposition \ref{zoo}, inequality \eqref{6.81}, \eqref{11eq}, and the fact that $\widetilde G(x,P(x,|u|)\xi^{g^+-1})\le\widetilde G(x,P(x,|u|))=H_1(x,|u|)$, we obtain
\begin{align}\label{22eq}
    \begin{split}
        &a_2 g^+\int_{A_{k,s}} P(x,|u|) \xi^{g^+-1} |D\xi|(u - k)  \mathrm{d}x \\
        &\quad \leq a_2 g^+ \int_{A_{k,s}} \widetilde{G}\left(x, P(x,|u|) \xi^{g^+-1}\right)  \mathrm{d}x + a_2 g^+ \int_{A_{k,s}} G\left(x, |D\xi|(u - k)\right)  \mathrm{d}x \\
        &\quad \leq a_2 g^+ \int_{A_{k,s}} H_1(x,|u|) \mathrm{d}x + a_2 g^+ 2^{g^+} \int_{A_{k,s}} G\left(x, \frac{u - k}{s-t}\right)  \mathrm{d}x \\
        &\quad \leq c_3 |A_{k,s}| + d_2 \int_{A_{k,s}} G\left(x, \frac{u - k}{s - t} \right)  \mathrm{d}x,
    \end{split}
\end{align}
where \(c_3 = c_3(a_2, F_1, h_1^+, M, g^+) > 0\) and \(d_2 = d_2(a_2, g^+) > 0\).\\
 In light of \eqref{6.71}, \eqref{H111}, \eqref{M}, Proposition \ref{zoo}, and Young's inequality \eqref{Yi}, choose \(\varepsilon_4 \in (0,1)\) such that
\[
b_1 \varepsilon_4^{\frac{h^+_1}{h^+_1-1}} \leq \frac{a_4}{4},
\]
we see that, using $0<u-k\le\delta M\le 2M$ on $A_{k,s}$ and $\frac{2M}{\varepsilon_4}\ge1$,
\begin{align}\label{33eq}
    \begin{split}
        &b_1 \int_{A_{k,s}} M(x,|Du|) \xi^{g^+} (u - k)  \mathrm{d}x \\
        &\quad \leq b_1 \int_{A_{k,s}} \widetilde{H_1}\left(x, \varepsilon_4 M(x,|Du|) \xi^{g^+}\right)  \mathrm{d}x + b_1 \int_{A_{k,s}} H_1\left(x, \frac{u - k}{\varepsilon_4}\right)  \mathrm{d}x \\
        &\quad \leq b_1 \varepsilon_4^{\frac{h^+_1}{h^+_1-1}} \int_{A_{k,s}} \widetilde{H_1}\left(x, M(x,|Du|) \right) \xi^{g^+\frac{h^+_1}{h^+_1-1}}  \mathrm{d}x + b_1 \int_{A_{k,s}} H_1\left(x, \frac{2M}{\varepsilon_4}\right)  \mathrm{d}x \\
        &\quad \leq b_1 \varepsilon_4^{\frac{h^+_1}{h^+_1-1}} \int_{A_{k,s}} G(x,|Du|) \xi^{g^+}  \mathrm{d}x + b_1 (2M)^{h_1^+} \varepsilon_4^{-h^+_1} F_1 |A_{k,s}| \\
        &\quad \leq \frac{a_4}{4} \int_{A_{k,s}} G(x,|Du|) \xi^{g^+}  \mathrm{d}x + c_4 |A_{k,s}|,
    \end{split}
\end{align}
with \(c_4 = c_4(b_1, M, F_1, h_1^+, a_4) > 0\).\\
 From \eqref{L1}, \eqref{H111}, \eqref{M}, Proposition \ref{zoo}, and Young's inequality \eqref{Yi}, we get
\begin{align}\label{44eq}
    \begin{split}
        &b_2 \int_{A_{k,s}} h_1(x,|u|) \xi^{g^+} (u - k)  \mathrm{d}x \\
        &\quad \leq b_2 \int_{A_{k,s}} \widetilde{H_1}\left(x, h_1(x,|u|)\right)  \mathrm{d}x + b_2 \int_{A_{k,s}} H_1(x, (u - k))  \mathrm{d}x \\
        &\quad \leq b_2 h^+_1 \int_{A_{k,s}} H_1\left(x, |u|\right)  \mathrm{d}x + b_2 \int_{A_{k,s}} H_1(x, 2M)  \mathrm{d}x \\
        &\quad \leq b_2 h^+_1 F_1 M^{h_1^+} |A_{k,s}| + b_2 F_1 (2M)^{h^+_1} |A_{k,s}| \\
        &\quad \leq c_5 |A_{k,s}|,
    \end{split}
\end{align}
where \(c_5 = c_5(b_2, M, F_1, h_1^+) > 0\).\\
 The remaining terms are bounded directly. From \eqref{4.29} and \eqref{4.30}, we have
\begin{align}
    a_6 \int_{A_{k,s}} \xi^{g^+}  \mathrm{d}x &\leq a_6 |A_{k,s}|, \\
    b_3\int_{A_{k,s}} \xi^{g^+} (u - k)  \mathrm{d}x &\leq b_3\int_{A_{k,s}} (u - k)  \mathrm{d}x \leq b_3(\max_{B_s} u(x) - k) |A_{k,s}| \leq b_3\delta M |A_{k,s}| \leq \frac{a_4}{4} |A_{k,s}|. \label{4.35}
\end{align}
Combining estimates \eqref{4.31}--\eqref{4.35} (the two terms $\frac{a_4}{4}\int G\xi^{g^+}$ from \eqref{4.32} and \eqref{33eq} are absorbed in the left-hand side of \eqref{4.31}, leaving $\frac{a_4}{2}\int G\xi^{g^+}$), we conclude that
\[
\int_{A_{k,s}} G(x,|Du|) \xi^{g^+}  \mathrm{d}x \leq \gamma \int_{A_{k,s}} G\left(x, \frac{u - k}{s - t} \right)  \mathrm{d}x + \gamma_1 |A_{k,s}|,
\]
and consequently, since $\xi\equiv1$ on $B_t$,
\begin{equation}\label{4.36}
\int_{A_{k,t}} G(x,|Du|)  \mathrm{d}x \leq \gamma \int_{A_{k,s}} G\left(x, \frac{u - k}{s - t} \right)  \mathrm{d}x + \gamma_1 |A_{k,s}|,
\end{equation}
for \(\bar{B}_t \subset B_s \subset \Omega\) and \(k\) satisfying \eqref{4.30}, where \(\gamma=\frac{2}{a_4}(d_1+1+d_2)\) and \(\gamma_1=\frac{2}{a_4}(c_1+c_2+c_3+c_4+c_5+a_6+\frac{a_4}{4})\) are positive constants with the dependences stated in Theorem~\ref{Theorem 4.2}. Inequality \eqref{4.36} (with $t=\rho-\sigma\rho$, $s=\rho$, $s-t=\sigma\rho$) implies \(u \in \mathcal{B}_{G(x,t)}(\Omega, M, \gamma, \gamma_1, \delta)\), the same for $-u$ (which solves a problem of the same type with the same structure constants). For the boundary class, let $B_s$ be a ball meeting $\partial\Omega$ and $$k\ge\max\{\max_{\Omega_s}w-\delta M,\max_{(\partial\Omega)_s}w\},$$ with $w=\pm u$. Since $u=0$ on $\partial\Omega$ we have $k\ge0$ and the trace of $(w-k)^+$ vanishes on $\partial\Omega$; therefore $\xi^{g^+}(w-k)^+\in W^{1,G(x,t)}_0(\Omega)$ is an admissible test function, and the above computation is repeated verbatim with $A_{k,s}=\{x\in\Omega_s: w>k\}$. Hence \(u \in \mathcal{B}_{G(x,t)}(\overline{\Omega}, M, \gamma, \gamma_1, \delta)\), which completes the proof of Theorem \ref{Theorem 4.2}.
\end{proof}
\begin{proof}[\textbf{Proof of Theorem \ref{Theorem 4.21}}]
Let $\overline{B_s}\subset\Omega$, $B_t\subset B_s$, and choose $\xi$ as in \eqref{55eq} and $\delta$ as in \eqref{4.29}. The test function
\[
v:=\xi^{g^+}(u-k)^+,\qquad k\ge\max_{B_s}u-\delta M,
\]
belongs to $W^{1,G(x,t)}(\Omega)$ and has compact support in $B_s\subset\Omega$; in particular $v=0$ on $\partial\Omega$, so that the boundary term $\int_{\partial\Omega}C(x,u)v\,\mathrm{d}\sigma$ in \eqref{FNV} vanishes and \eqref{FNV} reduces exactly to the identity \eqref{eq55} used in the Dirichlet case. Using \eqref{6.31}--\eqref{6.41}, \eqref{6.51} and \eqref{M} and repeating verbatim the estimates \eqref{4.31}--\eqref{4.36} in the proof of Theorem~\ref{Theorem 4.2}, we obtain
\[
\int_{A_{k,t}} G(x, |Du|) \, {\rm d}x
\leq \gamma \int_{A_{k,s}} G\left(x, \frac{u-k}{s-t}\right) \, {\rm d}x + \gamma_1 |A_{k,s}|,
\]
with the same constants $\gamma,\gamma_1,\delta$. The same holds for $-u$. Hence $u\in\mathcal{B}_{G(x,t)}(\Omega,M,\gamma,\gamma_1,\delta)$.
\end{proof}
\begin{remark}
For balls meeting $\partial\Omega$, the test function $\xi^{g^+}(u-k)^+$ at levels $k\ge\max_{(\partial\Omega)_s}u$ does not see the boundary term in the Neumann problem only if the trace of $u$ is controlled from above by $k$, and then one would need the Hölder continuity of the trace to start the boundary iteration (Theorem~\ref{Theorem 2.2}). Since this information is not available for~\eqref{PN}, no statement on the Hölder continuity up to the boundary is made for the Neumann problem.
\end{remark}
\begin{proof}[\textbf{Proof of Theorems~\ref{Theorem 4.3} and~\ref{Theorem 4.4}}]
Theorem~\ref{Theorem 4.3} follows from Theorem~\ref{Theorem 4.2} (respectively Theorem~\ref{Theorem 4.21}), which gives $u\in\mathcal{B}_{G(x,t)}(\Omega,M,\gamma,\gamma_1,\delta)$, together with Theorem~\ref{Theorem 2.1}. The interior statement of Theorem~\ref{Theorem 4.4} is Theorem~\ref{Theorem 4.3}. For the Dirichlet problem, Theorem~\ref{Theorem 4.2} gives $u\in\mathcal{B}_{G(x,t)}(\overline{\Omega},M,\gamma,\gamma_1,\delta)$; since $u|_{\partial\Omega}=0\in C^{0,1}(\partial\Omega)$ and $\partial\Omega$ satisfies \eqref{AO}, Theorem~\ref{Theorem 2.2} yields $u\in C^{0,\alpha}(\overline\Omega)$.
\end{proof}
\ \\
 \textbf{Author Contributions:} The authors contributed equally to this work.\\
\ \\
\textbf{ Data Availability:} Data sharing is not applicable to this article as no new data were created or analyzed in this
 study.\\
 \ \\
 \textbf{Declarations}\\
 \ \\
\textbf{ Ethical Approval:} Not applicable.\\
\ \\
\textbf{Conflict of interest:} The authors declare that they have no Conflict of interest.

\end{document}